\documentclass[11pt,reqno]{amsart}

\usepackage{graphicx,amsfonts,amssymb,amsmath,amsthm,url,verbatim,amscd}
\usepackage{subcaption}
\usepackage{color}
\usepackage[usenames,dvipsnames]{xcolor}
\newcommand{\revise}[1]{{\color{Bittersweet}   #1}}					
\newcommand{\temp}[1]{{\color{Gray} {\footnotesize  #1 }}}			

\usepackage[normalem]{ulem}
\usepackage{pdfsync}
\usepackage{booktabs}
\usepackage{todonotes}
\usepackage{mathrsfs} 			
\definecolor{link}{RGB}{11,0,128}
\usepackage[hyperfootnotes=false, colorlinks, citecolor=PineGreen, linkcolor=RoyalBlue, urlcolor=link]{hyperref}
\usepackage[parfill]{parskip}
\usepackage[alphabetic,lite]{amsrefs} 	
\usepackage{cleveref}
\usepackage{moreenum}			
\usepackage{colonequals}			
\usepackage{stmaryrd}			
\usepackage[all,cmtip]{xy}			
\usepackage{subfiles}
\usepackage{float}
\usepackage[in]{fullpage}

\usepackage[OT2, T1]{fontenc}

\usepackage{aliascnt}
\newaliascnt{numberingbase}{subsection}
\numberwithin{equation}{numberingbase}

\newtheoremstyle{thms}{0.7em}{0pt}{\itshape}{}{\bfseries}{.}{ }{}
\theoremstyle{thms}
\newtheorem*{theorem*}{Theorem}
\newtheorem{theorem} [numberingbase] {Theorem}
\newtheorem{thm}[numberingbase]{Theorem}
\newtheorem{lemma}      [numberingbase]{Lemma}
\newtheorem{corollary}  [numberingbase]{Corollary}
\newtheorem{proposition}[numberingbase]{Proposition}
\newtheorem{question}   [numberingbase]{Question}
\newtheorem{prop}		[numberingbase]{Proposition}
\newtheorem{cor}		[numberingbase]{Corollary}

\newtheorem{variant}[numberingbase]{Variant}

\Crefname{thm}{Theorem}{Theorems}

\newtheoremstyle{defs}{0.7em}{0pt}{}{}{\bfseries}{.}{ }{}
\theoremstyle{defs}

\newtheorem{remark}		[numberingbase] {Remark}

\newtheorem{rem}		[numberingbase]{Remark}

\newtheorem{defn} [numberingbase]{Definition}

\newtheorem{example}    [numberingbase]{Example}

\newtheorem*{rems}{Remarks}

\Crefname{variant}{Variant}{Variants}

\newtheorem{eg}[numberingbase]{Example}
\newtheorem*{egs}{Examples}

\newtheorem{claim}               {Claim}

\allowdisplaybreaks
\newcommand{\leftexp}[2]{{\vphantom{#2}}^{#1}%
      \kern-\scriptspace%
      {#2}}

\newcommand{\gL}{\lambda}

\newcommand{\bA}{\mathbb{A}}

\newcommand{\bC}{\mathbb{C}}

\newcommand{\bF}{\mathbb{F}}

\newcommand{\bP}{\mathbb{P}}
\newcommand{\bQ}{\mathbb{Q}}

\newcommand{\bZ}{\mathbb{Z}}

\newcommand{\fc}{\mathfrak{c}}

\newcommand{\fm}{\mathfrak{m}}

\newcommand{\fH}{\mathfrak{H}}

\newcommand{\fX}{\mathfrak{X}}

\newcommand{\sH}{\mathscr{H}}

\newcommand{\sO}{\mathscr{O}}

\newcommand{\sR}{\mathscr{R}}

\newcommand{\sU}{\mathscr{U}}

\newcommand{\sX}{\mathscr{X}}

\newcommand{\cE}{\mathcal{E}}

\newcommand{\cJ}{\mathcal{J}}

\newcommand{\cO}{\mathcal{O}}

\newcommand{\cW}{\mathcal{W}}

\DeclareMathOperator{\Char}{char}		
\DeclareMathOperator{\Coker}{Coker}	
\DeclareMathOperator{\Frac}{Frac}		
\DeclareMathOperator{\Ker}{Ker}		
\DeclareMathOperator{\Lie}{Lie}		
\DeclareMathOperator{\Pic}{Pic}		
\DeclareSymbolFont{cyrletters}{OT2}{wncyr}{m}{n}
\DeclareMathSymbol{\Sha}{\mathalpha}{cyrletters}{"58}	
\DeclareMathOperator{\Spec}{Spec}		
\DeclareMathOperator{\Spf}{Spf}		

\newcommand{\ce}{\colonequals}
\newcommand{\cont}{{\mathrm{cont}}}		

\newcommand{\eps}{\eps}
\newcommand{\hra}{\hookrightarrow}
\renewcommand{\i}{^{-1}}

\newcommand{\isomto}{\overset{\sim}{\longrightarrow}}

\newcommand{\llb}{\llbracket}			
\newcommand{\llp}{(\!(}			
\newcommand{\ov}{\overline}
\providecommand{\ppp}[1]{\left(#1\right)}

\newcommand{\ra}{\rightarrow}
\newcommand{\Ra}{\Rightarrow}
\newcommand{\reg}{\mathrm{reg}}		
\newcommand{\rrb}{\rrbracket}			
\newcommand{\rrp}{)\!)}			
\newcommand{\sh}{\mathrm{sh}}		
\newcommand{\sm}{\mathrm{sm}}			

\newcommand{\surjects}{\twoheadrightarrow}
\newcommand{\tensor}{\otimes} 			

\newcommand{\wh}{\widehat}
\newcommand{\wt}{\widetilde}
\newcommand{\xra}{\xrightarrow}

\providecommand{\up}[1]{{\upshape(}#1{\upshape)}}
\providecommand{\uref}[1]{{\upshape\ref{#1}}}
\providecommand{\uS}{{\upshape\S}}
\newcommand{\cusps}{\mathrm{cusps}}

\renewcommand{\b}{\textbf}
\providecommand{\ucolon}{{\upshape:} }
\providecommand{\uscolon}{{\upshape;} }

\newcommand{\brems}{\begin{rems} \hfill \begin{enumerate}[label=\b{\thenumberingbase.},ref=\thenumberingbase]}

\newcommand{\erems}{\end{enumerate} \end{rems}}
\newcommand{\begs}{\begin{egs} \hfill \begin{enumerate}[label=\b{\thenumberingbase.},ref=\thenumberingbase]}

\newcommand{\eegs}{\end{enumerate} \end{egs}}
\newcommand{\m}{\item}
\newcommand{\bsm}{\begin{smallmatrix}}
\newcommand{\esm}{\end{smallmatrix}}
\newcommand{\blem}{\begin{lemma}}
\newcommand{\elem}{\end{lemma}}

\newcommand{\bconj}{\begin{conj}}
\newcommand{\econj}{\end{conj}}
\newcommand{\bprob}{\begin{Problem}}
\newcommand{\eprob}{\end{Problem}}

\newcommand{\bq}{\begin{Q}}
\newcommand{\eq}{\end{Q}}
\newcommand{\benum}{\begin{enumerate}[label={{\upshape(\alph*)}}]}
\newcommand{\benuma}{\begin{enumerate}[label={{\upshape(\arabic*)}}]}
\newcommand{\benumr}{\begin{enumerate}[label={{\upshape(\roman*)}}]}
\newcommand{\eenum}{\end{enumerate}}
\newcommand{\bitem}{\begin{itemize}}
\newcommand{\eitem}{\end{itemize}}
\newcommand{\bc}{}
\newcommand{\bd}{\begin{defn}}
\newcommand{\ed}{\end{defn}}

\newcommand{\beg}{\begin{eg}}
\newcommand{\eeg}{\end{eg}}

\newcommand{\bcl}{\begin{claim}}
\newcommand{\ecl}{\end{claim}}
\newcommand{\lab}{\label}

\newcommand{\msk}{\medskip}

\newcommand{\x}{\text}
\newcommand{\rv}{\revise}

\newcommand{\q}{\quad}
\providecommand{\qxq}[1]{\quad\text{#1}\quad}
\providecommand{\qx}[1]{\quad\text{#1}}
\providecommand{\xq}[1]{\text{#1}\quad}
\newcommand{\qq}{\quad\quad}
\newcommand{\qqq}{\quad\quad\quad}
\newcommand{\qqqq}{\quad\quad\quad\quad}

\newcommand{\qqqqqq}{\quad\quad\quad\quad\quad\quad}

\newcommand{\qqqqqqqq}{\quad\quad\quad\quad\quad\quad\quad\quad}

\newcommand{\ba}{\begin{aligned}}
\newcommand{\ea}{\end{aligned}}
\newcommand{\be}{\begin{equation}}
\newcommand{\ee}{\end{equation}}
\newcommand{\bpf}{\begin{proof}}
\newcommand{\epf}{\end{proof}}
\newcommand{\bthm}{\begin{thm}}
\newcommand{\ethm}{\end{thm}}
\newcommand{\bprop}{\begin{prop}}
\newcommand{\eprop}{\end{prop}}
\newcommand{\bcor}{\begin{cor}}
\newcommand{\ecor}{\end{cor}}
\newcommand{\brem}{\begin{rem}}
\newcommand{\erem}{\end{rem}}
\newcommand*{\QED}{\hfill\ensuremath{\qed}}

\newtheoremstyle{subsection-tweak}
   {11pt}
   {3pt}%
   {}
   {}%
   {\bfseries}
   {}%
   {.5em}
   {\thmnumber{\@{#1}{}\@{#2}.}%
    \thmnote{~{\bfseries#3.}}}

\theoremstyle{subsection-tweak}
\newtheorem{pp}[numberingbase]{}
\newcommand{\bpp}{\begin{pp}}
\newcommand{\epp}{\end{pp}}

\newtheorem{ppt}[equation]{}
\newcommand{\bppt}{\begin{ppt}}
\newcommand{\eppt}{\end{ppt}}

\newcommand{\A}{{\mathbb A}}

\renewcommand{\a}{{\mathfrak a}}

\newcommand{\F}{{\mathcal F}}

\newcommand{\Q}{{\mathbb Q}}
\newcommand{\Z}{{\mathbb Z}}

\newcommand{\R}{{\mathbb R}}
\newcommand{\C}{{\mathbb C}}
\newcommand{\bs}{\backslash}

\newcommand{\OF}{{\mathcal O_F}}
\newcommand{\GL}{{\rm GL}}
\newcommand{\PGL}{{\rm PGL}}
\newcommand{\SL}{{\rm SL}}

\newcommand{\St}{{\rm St}}

\newcommand{\FF}{\mathbb{F}_{F}}

\newcommand{\val}{{\rm val}}

\newcommand{\Gal}{{\rm Gal}}

\newcommand{\Ind}{{\rm Ind}}

\newcommand{\Hom}{{\rm Hom}}

\newcommand{\sabcd}[4]{\left(\begin{smallmatrix}#1&#2\\#3& #4\end{smallmatrix}\right)}
\newcommand{\abcd}[4]{\left(\begin{smallmatrix}#1&#2\\#3& #4\end{smallmatrix}\right)}
\newcommand{\qF}{q_F}

\newcommand{\Aut}{{\rm Aut}}
	\newcommand{\michalis}[1]{{\color{OliveGreen} \sf \it{ {\scriptsize ``#1''} }}}

\newcommand{\tst}{\textstyle}

\def\vol{\operatorname{vol}}

\def\eps{\varepsilon}

\newcommand{\Cx}{\mathbb{C}^{\times}}

\newcommand{\Fp}{\mathbb{F}_{p}}

\newcommand{\Oix}{\mathcal{O}_F^{\times}}
\newcommand{\Oi}{\mathcal{O}_F}

\newcommand{\Xc}{\mathfrak{X}}

\newcommand{\Ga}{\mathfrak{G}}

\newcommand{\Fx}{F^{\times}}

\newcommand{\Wh}{\mathcal{W}}

\newcommand{\dxy}{d^{\times}y}

\newcommand{\im}{\operatorname{Im}}

\newcommand{\andy}{\quad{\rm and}\quad}

\newcommand{\abs}[1]{\lvert{#1}\rvert}
\newcommand{\absF}[1]{\lvert{#1}\rvert_F}

\makeatletter
 \def\l@section{\@tocline{1}{0pt}{1pc}{}{}}

\renewcommand{\tocsection}[3]{%
  \indentlabel{\@ifnotempty{#2}{\makebox[1.3em][l]{%
    \ignorespaces#1 \bfseries{#2}.\hfill}}}\bfseries{#3}
    \vspace{1.5pt}}

\renewcommand{\tocsubsection}[3]{%
  \indentlabel{\@ifnotempty{#2}{\hspace*{-0.5em}\makebox[2.1em][l]{%
    \ignorespaces#1#2.\hfill}}}#3
    \vspace{1.5pt}}

\makeatother

\makeatletter 
\newcommand\appendix@section[1]{%
  \refstepcounter{section}%
  \orig@section*{Appendix \@Alph\c@section. #1}%
}
\let\orig@section\section
\g@addto@macro\appendix{\let\section\appendix@section}
\makeatother

\begin{document}

\bibliographystyle{plain}

\title{The Manin constant and the modular degree}

\author{K\k{e}stutis \v{C}esnavi\v{c}ius}
\address{CNRS, UMR 8628, Laboratoire de Math\'{e}matiques d'Orsay, Univ.~Paris-Sud, Universit\'{e} Paris-Saclay, 91405 Orsay, France}
\email{kestutis@math.u-psud.fr}

\author{Michael Neururer}
\address{Fachbereich Mathematik, Technische Universit\"{a}t Darmstadt, Schlo\ss gartenstr.~7, 64289 Darmstadt, Germany.}
\email{neururer@mathematik.tu-darmstadt.de}

\author{Abhishek Saha}
\address{School of Mathematical Sciences\\
  Queen Mary University of London\\
  London E1 4NS\\
  UK}
  \email{abhishek.saha@qmul.ac.uk}

\date{\today}


\maketitle

\begin{abstract}
The Manin constant $c$ of an elliptic curve $E$ over $\bQ$ is the nonzero integer that scales the differential $\omega_f$
determined by the normalized newform $f$ associated to $E$ into the pullback of a N\'{e}ron differential under a minimal 
parametrization $\phi\colon X_0(N)_\bQ \surjects E$. Manin conjectured that $c = \pm 1$ for optimal parametrizations, and we prove that in general $c \mid \deg(\phi)$  under a minor assumption at $2$ and $3$ that is not needed for cube-free $N$ or for parametrizations by $X_1(N)_\bQ$. Since $c$ is supported at the additive reduction primes, which need not divide
 $\deg(\phi)$, this improves the status of the Manin conjecture for many $E$. Our core result that gives this divisibility is the containment $\omega_f \in H^0(X_0(N), \Omega)$, which we establish by combining automorphic methods with techniques from arithmetic geometry; here the modular curve $X_0(N)$ is considered over $\bZ$ and $\Omega$ is its relative dualizing sheaf over $\bZ$. We reduce this containment  to  $p$-adic bounds on denominators 
of the Fourier expansions of $f$ at \emph{all}
the cusps of $X_0(N)_\bC$ and then use the recent basic identity for 
the $p$-adic Whittaker newform to establish stronger bounds in the more general setup of newforms of weight $k$ on $X_0(N)$. 
To overcome obstacles at $2$ and $3$, we 
analyze nondihedral supercuspidal representations of $\GL_2(\bQ_2)$ and exhibit new cases in which $X_0(N)_\bZ$ has rational singularities.
\end{abstract}

\hypersetup{
    linktoc=page,     
}
\renewcommand*\contentsname{}
\quad\\
\tableofcontents

\section{Introduction}
By the Shimura--Taniyama conjecture settled in 
\cite{Wil95}, 
\cite{TW95}, and 
\cite{BCDT00}, for every elliptic curve $E$ over $\bQ$ of conductor $N$ and every subgroup $\Gamma_1(N) \subset \Gamma \subset \Gamma_0(N)$ of $\GL_2(\wh{\bZ})$, there is~a
\[
\xq{surjection} \phi\colon (X_\Gamma)_\bQ \surjects E \qxq{from the modular curve} (X_\Gamma)_\bQ.
\]
Most commonly, $\Gamma$ is $\Gamma_0(N)$ or $\Gamma_1(N)$, so that $X_\Gamma$ is $X_0(N)$ or $X_1(N)$, but for different $\Gamma$ different $E$ may be more canonical within the same isogeny class: for instance, $X_1(11)_\bQ$ and $X_0(11)_\bQ$ are distinct isogenous elliptic curves. The multiplicity one theorem ensures that the $\phi$-pullback of a N\'{e}ron differential $\omega_E$ is a nonzero multiple of the differential $\omega_{f} \in H^0((X_\Gamma)_\bQ, \Omega^1)$ associated to the normalized newform $f$ whose Hecke eigenvalues agree with the Frobenius traces of $E$:
\[
\phi^*(\omega_E) = c_\phi \cdot \omega_f \qxq{for a unique} c_\phi \in \bQ^\times,
\]
and one knows that
\footnote{It seems that the integrality of $c_\phi$ 
was first noticed by Gabber during his PhD studies. To establish it, one reduces to the case $\Gamma = \Gamma_1(N)$ and then uses $q$-expansions, see \Cref{const-in-Z} and its proof.} $c_\phi \in \bZ$ (we abuse notation: $\omega_E$ is nonunique, so $\phi$ determines only $\pm c_\phi$). 
For fixed $\Gamma$ and $E$ there are many $\phi$, 
so it is common to normalize 
 $\phi$ to be \emph{optimal}, that is, $\deg (\phi)$ to be the least possible as $E$ varies in its isogeny class and $\Gamma$ is fixed (any $\phi$ factors through an optimal one, see the proof of \Cref{const-in-Z} and use multiplicity one). 
For optimal $\phi$, Manin conjectured~that
\[
c_\phi \overset{?}{=} \pm 1,
\]
see \cite{Man71}*{Section 10.3}.\footnote{Manin considered $\Gamma = \Gamma_0(N)$, and this implies the general case by \Cref{const-in-Z}.  In \cite{Ste89}, Stevens argued that minimal degree parametrizations by $X_1(N)_\bQ$ are the most natural ones, and he conjectured that $c_\phi = \pm 1$ for them.} From the theoretical point of view, the natural approach to the Manin conjecture is to argue that $p \nmid c_\phi$ for every prime $p$: geometrically, this $p$-adic statement translates to studying the arithmetic properties of the ``reduction modulo $p$'' of the parametrization $\phi$. This is not so in the computational approach, where for explicit $E$ one computes with modular symbols to check ``directly'' that $c_\phi = \pm 1$: indeed, Cremona used the computational approach to prove in \cite{Cre19} that the Manin conjecture holds whenever $N \le 500000$. The divergence of the two approaches gives this overwhelming computational evidence for the Manin conjecture even more weight. 

The initial theoretical results on the Manin conjecture were based on exactness properties of N\'{e}ron models and showed that $p \nmid c_\phi$ for those $p > 2$ at which $E$ has semistable reduction, see \cite{Maz78} (and \cite{AU96}, \cite{ARS06} for some sharpenings). By passing to a minimal extension $K$ of $\bQ_p$ over which $E$ acquires semistable reduction and analyzing a stable integral model of $X_0(N)_{\ov{\bQ}_p}$, Edixhoven was able to extend this approach to some primes $p$ at which $E$ has additive reduction: in \cite{Edi91}*{Theorem~3}, he showed that $p \nmid c_\phi$ for any prime $p \ge 11$ at which $E$ does not have an additive potentially ordinary reduction of Kodaira type II, III, or IV.\footnote{In the unfinished manuscript \cite{Edi01}, he attempted to remove this assumption on Kodaira types (still for $p \ge 11$).} In these geometric approaches, the key input to the required exactness properties is Raynaud's result from \cite{Ray74a} on uniqueness of commutative, finite, flat group schemes with a fixed generic fiber over a complete discrete valuation ring of mixed charcateristic $(0, p)$ and absolute ramification index $e < p - 1$. Raynaud's results were later subsumed into integral $p$-adic Hodge theory but the requirement $e < p - 1$ for exactness properties persisted, so there seems to be little hope that this approach is the ``right'' one for the Manin conjecture.

The conclusion $p \nmid c_\phi$ was established for all primes $p$ of semistable reduction for $E$ by a different method in \cite{Manin-semistable}. The key novelty was to analyze the Hecke module structure of the Lie algebra of the N\'{e}ron model of $J_0(N)$ using a multiplicity one result in characteristic $p$, and this showed that automorphic rather than purely algebro-geometric techniques that were tried previously may be better suited for the Manin conjecture. The latter is most interesting in the remaining case of a prime $p$ of additive reduction for $E$, since then the relevant  arithmetic geometry is the most~delicate. 

In this article, we combine automorphic methods with those of arithmetic geometry to settle a subconjecture of the Manin conjecture, reviewed as \eqref{wf-str-int-ann} below. We then show that this subconjecture has the following divisibility consequences for the Manin constant.

\bthm[\Cref{X1N-main-pf}] \label{main-X1n}
For an elliptic curve $E$ over $\bQ$ of conductor $N$, every surjection
\[
\phi \colon X_1(N)_\bQ \surjects E \qxq{satisfies} c_\phi \mid \deg(\phi).
\]
\ethm


\bthm[\Cref{thm:main-result-proof}] \label{main-general}
For an elliptic curve $E$ over $\bQ$ of conductor $N$, and for a level $\Gamma$  with $\Gamma_1(N) \subset \Gamma \subset \Gamma_0(N)$, every surjection
\[
\phi \colon (X_\Gamma)_\bQ \surjects E \qxq{satisfies} c_\phi \mid 6\cdot \deg(\phi), 
\]
and if $N$ is cube-free \up{that is, if $8 \nmid N$ and $27 \nmid N$}, then even
\[
c_\phi \mid \deg(\phi).
\]
More precisely, under these assumptions, for every prime $p$ we have
\[
\val_p(c_\phi) \le \val_p(\deg(\phi)) + \begin{cases}
1 &\x{if $p = 2$ with $\val_2(N) \ge 3$ and there is no $p'\mid N$ with $p' \equiv 3 \bmod 4$,} \\
1 &\x{if $p = 3$ with $\val_3(N) \ge 3$ and there is no $p'\mid N$ with $p' \equiv 2 \bmod 3$,} \\
0 &\x{otherwise,}
\end{cases}
\]
and, more generally, if for some $\Gamma \subseteq \Gamma' \subseteq \Gamma_0(N)$ the singularities of $(X_{\Gamma'})_{\bZ_{(p)}}$ are rational, then
\[
\val_p(c_\phi) \le \val_p(\deg(\phi)).
\]
\ethm

The modular degree $\deg(\phi)$ is often even, for instance, if $\Gamma = \Gamma_0(N)$ and $\phi$ factors through some Atkin--Lehner quotient, but otherwise it is somewhat mysterious. In particular, for many $E$ this degree is coprime with $N$, to the effect that the new upper bound $\val_p(c_\phi) \le \val_p(\deg(\phi))$ supplied by \Cref{main-X1n,main-general} eliminates\footnote{The bounds in \Cref{main-X1n,main-general} hold for any parametrization $\phi$, although it is only for optimal $\phi$ that the Manin constant $c_\phi$ is conjectured to equal $\pm 1$ (and known to be divisible only by the primes of additive reduction). For example, when $E$ equals the elliptic curve with Cremona label 11a3, which is a model of $X_1(11)_{\Q}$, and $\phi \colon X_0(N)_\bQ \surjects E$ is the isogeny of least degree, one has $c_\phi = \deg(\phi)=5$, which is consistent with our bounds.} some additive primes that could divide $c_\phi$ for optimal $\phi$.

To illustrate, in the following figure we plotted in green the fraction of those isogeny classes of $E$ over $\bQ$ of conductor $N \le 300000$ that have an odd additive prime $p$ but for which no such $p$ divides $\deg \phi$, where $\phi$ is the optimal parametrization by $X_0(N)_\bQ$; if $p=3$ with $\val_3(N) \ge 3$, then we also require that there exist a $p'\mid N$ with $p' \equiv  2 \bmod 3$. \Cref{main-general} shows that the Manin constant for such $E$  is a power of $2$ (the semistable primes are eliminated by earlier results, as reviewed above). Furthermore, we plotted in yellow the fraction of those isogeny classes as above for which some odd $p$ of additive reduction does not divide $\deg \phi$ and some other does, with the same caveat for $p = 3$, so that \Cref{main-general} eliminates at least one odd additive prime.
Even though in all of these small conductor cases the full Manin conjecture is known by Cremona's verification \cite{Cre19}, the figure shows the scope of the improvement supplied by \Cref{main-X1n,main-general}.

{\begin{center} \includegraphics[scale=.75]{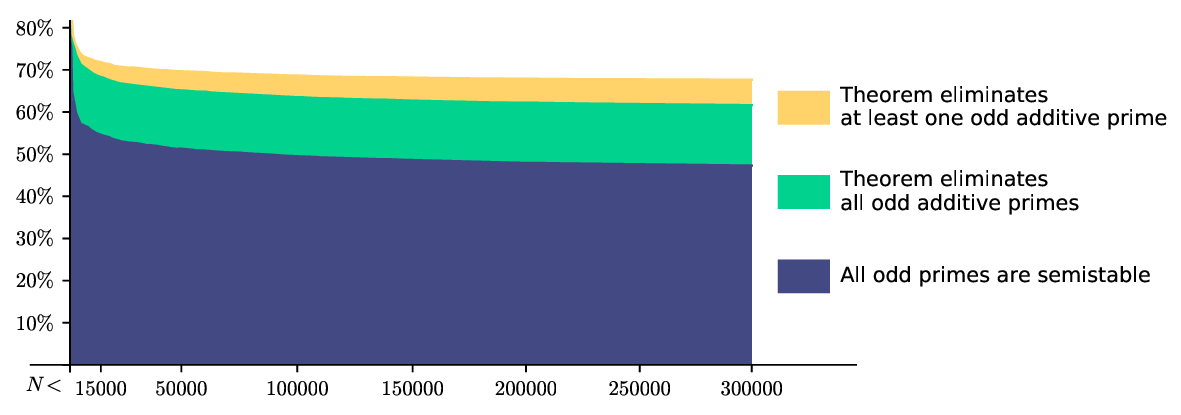}
	\end{center}}

The key input to \Cref{main-X1n,main-general} and the core result of this article is the following  integrality property of $\omega_f$ that follows from the Manin conjecture. Namely, we argue in \Cref{wf-integral} that
\be \label{wf-str-int-ann} \tag{$\star$}
\omega_f \qxq{lies in the $\bZ$-lattice} H^0(X_0(N), \Omega) \subset H^0(X_0(N)_\bQ, \Omega^1),
\ee
where the modular curve $X_0(N)$ is over $\bZ$ and $\Omega$ is its relative dualizing sheaf over $\bZ$. In addition to being implied by the Manin conjecture, the containment \eqref{wf-str-int-ann} is actually \emph{necessary} for attacking it: except for unforeseen radically new approaches, all indications point to \eqref{wf-str-int-ann} being used in future work on the remaining cases of the Manin conjecture. 

The containment \eqref{wf-str-int-ann} is straight-forward in the semistable case, that is, for squarefree $N$, thanks to $q$-expansions and the Atkin--Lehner involution. More generally, since the formal completion of $X_0(N)$ along $\infty$ is $\Spf(\bZ\llb q \rrb)$,  the weaker containment $\omega_f \in H^0(X_0(N)^\infty, \Omega^1)$ amounts to the integrality of the Fourier expansion of $f$ at  $\infty$, where $X_0(N)^\infty \subset X_0(N)$ is the ($\bZ$-smooth) open complement of those $\bZ$-fibral irreducible components 
that do not meet the $\bZ$-point given by the cusp $\infty$. Similarly, \eqref{wf-str-int-ann} amounts to certain bounds on the $p$-adic valuations of the denominators of the Fourier coefficients of $f$ at \emph{all} the cusps of $X_0(N)_\bC$---at least up to difficulties caused by the lack of a modular interpretation of the coarse space $X_0(N)$ that we overcome in \S\ref{canonical-lattice} by exploiting the Deligne--Mumford stack $\sX_0(N)$ 
and its ``relative dualizing'' sheaf $\Omega$. We compute the precise required bounds in \Cref{explicit-crit}, and an important step for this is to compute the differents of the extensions of discrete valuation rings obtained by localizing the finite flat cover $\sX_0(N) \ra \sX(1)$ at the generic points of the $\bF_p$-fiber of $\sX_0(N)$, which we do in \Cref{compute-dy}.

To show that the required bounds are met, we use automorphic methods to establish the following stronger bounds. In \Cref{exmp:Fourier expansions} we show that these bounds are \emph{sharp} in the case of newforms associated to elliptic curves (and $p\le 11$) and we discuss their computational potential. 

\begin{theorem}[\Cref{main-estimates} and \Cref{indep-of-c}]\label{t:fourierintro}
For a prime $p$, a  cuspidal, normalized newform $f$ of weight $k$ on $\Gamma_0(N)$,  an isomorphism $\bC \simeq \ov{\bQ}_p$, the resulting $\val_p\colon \bC \ra \bQ \cup \{ \infty\}$ with $\val_p(p) = 1$, and a cusp $\fc \in X_0(N)(\bC)$ of denominator $L$ \up{see \uS\uref{pp:cusps}}, the Fourier coefficients $a_f(r;\fc)$ satisfy
\[
\val_p(a_f(r;\fc)) \geq\tst  - \frac k2 \val_p\ppp{\frac{N}{\gcd(L^2,\, N)}} + \begin{cases}
0  &\text{if }~\val_p(\gcd(L, \frac NL)) = 0, \\
0  &\text{if }~\val_p(\gcd(L, \frac NL)) = 1,~\val_p(N) >2, \\
- \frac12 & \text{if }~\val_p(L) = \frac{1}{2}\val_p(N) = 1,  \\
1 -   \frac 12 \val_p(\gcd(L, \frac NL)) \!\!\! & \text{otherwise,}
\end{cases}
\]
as well as the following stronger bounds in the case $p = 2$\ucolon
\[
\tst \val_2(a_f(r;\fc))\geq \tst- \frac k2 \val_2\ppp{\frac{N}{\gcd(L^2,\, N)}} + \begin{cases}
 0 & \text{if }~\val_2(L) = \frac{1}2\val_2(N) =1,\\
\frac{k}2 & \text{if }~\val_2(L) = \frac{1}2\val_2(N) \in \{2, 3, 4\}, \\ \frac{k}2 + 1 - \frac{1}{4}\val_2(N)\! &\text{if }~\val_2(L) = \frac{1}2\val_2(N) > 4, \\
 0  & \text{if }~\val_2(\gcd(L, \frac NL)) = 3,~\val_2(N) > 6.
\end{cases}
\]
Moreover, $\min_r(\val_p(a_f(r;\fc)))$ only depends on $f$ and $L$, and not on the cusp $\fc$ with denominator~$L$.
\end{theorem}

To argue the above bounds we pass to the automorphic side by expressing the  ``$p$-part'' of $a_f(r;\fc)$ in terms 
of the local Whittaker newform $W_{f,\, p}$ of the irreducible, admissible representation $\pi_{f,\, p}$ of $\GL_2(\bQ_p)$ determined by $f$ (see \Cref{prop:localgloballink} and its proof). 
Thus, \Cref{t:fourierintro} hinges on the $p$-adic analysis of the values of $W_{f,\, p}$, which is a purely local question about $\pi_{f,\, p}$. To access these values, we use the local Fourier expansion of $W_{f,\, p}$ and analyze the resulting local Fourier coefficients $c_{t,\,\ell}(\chi)$ with the help of the recent ``basic identity'' (reviewed in \S\ref{pp:basic-identity}) that was derived by the third-named author in \cite{Sah16} from the $\GL_2$ local functional equation of Jacquet--Langlands \cite{JL70}.

The coefficients $c_{t,\,\ell}(\chi) \in \bC$ are indexed by characters $\chi\colon \bZ_p^\times \ra \bC^\times$ (the relevant $t$ and $\ell$ are determined by $N$, $L$, and $r$), and reasonably explicit formulas for the $c_{t,\,\ell}(\chi)$ were worked out in special cases in \cite{Sah16} and appeared in general in the recent work of Assing \cite{Ass19}. These formulas involve the Jacquet--Langlands $\GL_2$ local $\eps$-factors, 
which for $p \ne 2$ can be expressed in terms of the $\GL_1$ local $\eps$-factors of Tate, equivalently, in terms of Gauss sums of characters of $F^\times$ for at most quadratic extensions $F/\Q_p$. In effect, $p$-adically bounding the values of $W_{f,\, p}$, which is a problem on $\GL_2$,
reduces to $p$-adically bounding Gauss sums of characters, which is an approachable problem on $\GL_1$. We study the latter in \S\ref{s:prelim-local} and then bound the values of $W_{f,\, p}$ in the key \Cref{T1,T2}. Their most delicate case $p = 2$ uses a classification of nondihedral supercuspidal  representations of $\GL_2(\bQ_2)$ derived via the local Langlands correspondence (see \Cref{lem:Type1b}) and, to go beyond the na\"ive bounds, takes into account cancellations between the $c_{t,\,\ell}(\chi)$. Thanks in part to this additional attention to $p = 2$, we obtain the integrality result \eqref{wf-str-int-ann} without any exceptions.


In a more restrictive setting and by a different method, bounds on $p$-adic valuations of Fourier expansions were investigated by Edixhoven in \S3 of his unfinished manuscript \cite{Edi01}. There 
he also hoped for a more conceptual approach that would be based on studying the Kirillov model of $\pi_{f,\,p}$, and the work of our \S\S\ref{s:prelim-local}--\ref{s:global p-adic valuations} realizes this prediction (we use the Whittaker model instead). 

The automorphic approach to \eqref{wf-str-int-ann} seems much sharper and more natural than those based on arithmetic geometry alone. For instance, as explained in Conrad's \cite{BDP17}*{Appendix~B}, one may use intersection theory on the regular stacky arithmetic surface $\sX_0(N)$ to bound the denominator of $\omega_f$ with respect to the lattice 
\[
H^0(X_0(N), \Omega) \cong H^0(\sX_0(N), \Omega)
\]
(see \Cref{AM-X0n} for this identification). The bounds obtained in this way are far from those needed for \eqref{wf-str-int-ann}, but the intersection-theoretic approach is not specific to $\omega_f$---in essence it bounds the exponent of the finite group 
\[
H^0(X_0(N)^\infty, \Omega^1)/H^0(X_0(N), \Omega).
\]
\emph{Loc.~cit.}~carries it out\footnote{Unfortunately, beyond the case $\val_p(N) = 1$ treated in \cite{DR73}*{Chapitre VII, Section 3.19, Proposition 3.20}, the explicit bounds stated in \cite{BDP17}*{Theorem B.3.2.1}  suffer from a typo in the values of the multiplicities of the components of $\sX_0(N)_{\bF_p}$ stated in \cite{BDP17}*{Theorem B.3.1.3} (by \cite{KM85}*{Section (13.5.6)}, the correct multiplicity of the $(a, b)$-component for $a, b > 0$ is $p^{\min(a,\, b) - 1}(p - 1)$). 
Consequently, the asymptotic behavior in $p$ of the stated bounds differs from the case $\val_p(N) = 1$.} for the line bundle $\omega^{\tensor k}$ 
in place of $\Omega$.

Turning back to \Cref{main-general}, the only role of its rational singularity assumption  is to ensure that $\Pic^0_{X_0(N)/\bZ}$ is the N\'{e}ron model $\cJ_0(N)$ of the Jacobian $J_0(N)$ (here we chose $\Gamma' = \Gamma_0(N)$ to simplify), and so to deduce from \eqref{wf-str-int-ann} that $\omega_f$ lies in an even \textit{a priori} smaller lattice $H^0(\cJ_0(N), \Omega^1)$ that seems otherwise inaccessible. We do not know any $N$ for which this assumption fails, in fact, for a prime $p$ we show in \Cref{rat-sing-main} that $X_0(N)_{\bZ_{(p)}}$ has rational singularities in the following cases:
\benumr
\m
if $p \ge 5$; or

\m
if $p = 3$ and either $\val_p(N) \le 2$ or there is a prime $p' \mid N$ with $p' \equiv 2 \bmod 3$; or

\m
if $p = 2$ and either $\val_p(N) \le 2$ or there is a prime $p' \mid N$ with $p' \equiv 3 \bmod 4$.
\eenum
The bulk of this rational singularity criterion is due to Raynaud \cite{Ray91}, but we used low conductor instances of the Manin conjecture to add the cases $p \le 3$ with $\val_p(N) = 2$. 
The technique we develop for this also reduces the desired divisibility $c_\phi \mid \deg(\phi)$ in its few still outstanding cases to a \emph{finite} computational problem (albeit not one we know how to solve completely), see \Cref{reduce-to-compute}.


\bpp[Notation and conventions] \lab{conv}
For a prime $p$, we let $\val_p\colon \overline{\Q}_p \rightarrow \Q \cup \{\infty\}$ be the $p$-adic valuation with $\val_p(p)=1$. For a nonarchimedean local field $F$, we let $\cO_F$ be its integer ring, $\fm_F \subset \cO_F$ the maximal ideal, $\varpi_F \in \fm_F$ a uniformizer, $\bF_F \ce \cO_F/\fm_F$ the residue field, $\qF\ce\#\FF$ its order, and $W_F \subset \Gal(\ov{F}/F)$ the Weil group. 
We normalize local class field theory by letting \emph{geometric} Frobenii map 
to uniformizers 
(see \cite{BH06}*{Section 29.1}).
We normalize the absolute value $\absF{\,\cdot\,}$ on $F$  by $\absF{\varpi_F}=\frac 1\qF$. 
We set $\zeta_F(s) \ce \frac{1}{1-\qF^{-s}}$, 
for which we only need the values
\be \label{zeta-values}
\tst \zeta_F(1) = \frac{\qF}{\qF - 1} \qxq{and} \zeta_F(2) =\frac{\qF^2}{\qF^2 - 1}.
\ee
 For a (continuous) character $\chi\colon \Fx\rightarrow\Cx$, we let $a(\chi)$ be the \textit{conductor exponent}: the smallest $n > 0$ with $\chi(1 + \fm_F^n) = 1$ if $\chi(\cO_F^\times) \neq \{1\}$ and $0$ if $\chi(\cO_F^\times) = \{1\}$ (in which case $\chi$ is \textit{unramified}). 
 For a nontrivial additive character $\psi \colon F \ra \bC^\times$, we let $c(\psi)$ be the  smallest\footnote{In terms of the notation  $n(\psi)$ used in \cite{Tat79}*{Section (3.2.6)} or \cite{Del73b}*{Section 3.4}, we have $c(\psi)=-n(\psi)$.} $n \in \bZ$ with $\psi(\fm_F^{n}) = \{ 1\}$. 




For an open subgroup $\Gamma \subset \GL_2(\wh{\bZ})$, we let $\sX_\Gamma$ be the level $\Gamma$ modular Deligne--Mumford $\bZ$-stack defined in \cite{DR73}*{Chapitre IV, D\'{e}finition~3.3} via normalization, and $X_\Gamma$ its coarse moduli space, so that  $X_\Gamma$ is the usual projective modular curve over $\bZ$ of level $\Gamma$ and, whenever $\Gamma$ is small enough, $\sX_\Gamma = X_\Gamma$ (see \cite{modular-description}*{Section 4.1 and Section 6.1 up to Proposition 6.3} for a basic review of these objects). We let
\[\ba
\Gamma_0(N) \subset \GL_2(\wh{\bZ})  &\qxq{be the preimage of} \{ \ppp{\begin{smallmatrix} * & * \\ 0 & * \end{smallmatrix}} \} \subset \GL_2(\bZ/N\bZ), \qxq{and set} \sX_0(N) \ce \sX_{\Gamma_0(N)}; \\
\Gamma_1(N)\subset \GL_2(\wh{\bZ}) &\qxq{be the preimage of} \{ \ppp{\begin{smallmatrix} 1 & * \\ 0 & * \end{smallmatrix}} \} \subset \GL_2(\bZ/N\bZ), \qxq{and set} \sX_1(N) \ce \sX_{\Gamma_1(N)}; \\
\Gamma(N)\subset \GL_2(\wh{\bZ}) &\qxq{be the preimage of} \{ \ppp{\begin{smallmatrix} 1 & 0 \\ 0 & 1 \end{smallmatrix}} \} \subset \GL_2(\bZ/N\bZ), \qxq{and set} \sX(N) \ce \sX_{\Gamma(N)}.
\ea
\]
We write $X_0(N)$, $X_1(N)$, $X(N)$ for the coarse spaces and use
the $j$-invariant to identify $X(1)$ with $\bP^1_\bZ$ (see \cite{DR73}*{Chapitre VI, Th\'{e}or\`{e}me 1.1, Section 1.3}).  For a scheme $X$, we let $X^\reg \subset X$ be the set of $x \in X$ with $\sO_{X,\, x}$ regular. If $X$ is over a base $S$, we let $X^\sm \subset X$ be the open locus of $S$-smoothness. We let $\Omega^1_{X/S}$ denote the K\"{a}hler differentials. We let $\ov{x}$ be a geometric point over $x$ and let $\sO_{X,\, x}^\sh$ or $\sO_{X,\, \ov{x}}^\sh$ denote the resulting strict Henselization. We also use analogous notation when $X$ is merely a Deligne--Mumford stack.

We let $\ov{\bZ}$ be the integral closure of $\bZ$ in $\bC$, set $\zeta_n \ce e^{2 \pi i /n}$, and let $\bZ_{(p)}$ be the localization of $\bZ$ at the prime $(p)$. We let $\phi(m) \ce \#((\bZ/m\bZ)^\times)$ be the Euler totient function. For a field, a `finite extension' means a finite field extension. Rings are assumed to be commutative. Both $\subset$ and $\subseteq$ allow equality. We write $\cong$ for canonical isomorphisms (identifications), $\simeq$  for noncanonical ones, $\hra$ for monomorphisms, $\surjects$ for epimorphisms, and $\isomto$ for isomorphisms (in categories in question). Our representations and characters 
are continuous and  over $\bC$, 
and $\mathbf{1}$ is the trivial character.

\epp

\subsection*{Acknowledgements}
The first-named author thanks Naoki Imai for numerous discussions about the Manin constant during 2017--2018. We thank the referee for helpful comments. We thank the Technische Universit\"at Darmstadt and the Queen Mary University of London for hospitality during visits between the authors. The first-named author acknowledges the support of the ANR-19-CE40-0015-02 COLOSS. The second-named author was partially funded by the DFG research group 1920 and the LOEWE research unit ``Uniformized Structures in Arithmetic and Geometry.''
 The third-named author acknowledges the support of the Leverhulme Trust Research Project Grant RPG-2018-401.
This project has received funding from the European Research Council (ERC) under the European Union's Horizon 2020 research and innovation programme (grant agreement No.~851146).  


\section{$p$-adic properties of Gauss sums} \label{s:prelim-local}



Our ultimate source of $p$-adic properties of coefficients of $q$-expansions of newforms at cusps is the $p$-adic properties of Gauss sums of characters, relatedly, of $\eps$-factors of $\GL(1)$. Thus, we begin by explicating the latter in this section, especially, in \Cref{a(mu)=1} and \Cref{a(mu)>1}.


\bpp[Local field Gauss sums] \label{s:localgauss}
For a finite extension $F/\bQ_p$, a multiplicative character $\chi \colon \!F^\times \!\!\ra\!\bC^\times$, a nontrivial additive character $\psi \colon F \ra \bC^\times$, the \textit{Gauss sum} of $\chi$ with respect to $\psi$ is defined by
\begin{equation*}
\Ga_\psi(x,\chi) \ce \int_{\cO_F^\times}\chi(y)\,\psi(xy)\,\dxy \qxq{for} x\in\Fx, \qxq{with the normalization} \int_{\cO_F^\times} \dxy=1.
\end{equation*}
Since $\Ga_\psi(x,\chi)$ only sees $\chi|_{\cO_F^\times}$, it does not change when $\chi$ is multiplied by an unramified character, so we lose no generality if we assume that $\chi$ lies in the set
\[
\Xc \ce \{\x{continuous character }\chi\colon\Fx\rightarrow\Cx \x{ with } \chi(\varpi_F)=1 \} \cong \Hom_\cont(\cO_F^\times, \bC^\times).
\]
Characters in $\Xc$ are unitary and of finite order, and we also consider subsets of fixed conductor exponent:
\begin{equation*}
\Xc_{\le k} \ce \{\chi\in\Xc\,\vert\,a(\chi)\leq k\} \andy \Xc_{k} \ce \{\chi\in\Xc\,\vert\,a(\chi)= k\}
\end{equation*}
(to stress the underlying field, we also write $\Xc_F$, $\Xc_{F,\, \le k}$, and $\Xc_{F,\, k}$). The Gauss sum $\Ga_\psi(x,\chi)$ is related to the $\GL(1)$-epsilon factors $\eps(s, \chi, \psi)$ defined by Tate (see \cite{Tat79}*{Section~(3.2)} or \cite{Sch02}*{Section 1.1}): under the common normalization $c(\psi) = 0$, by \cite{CS18}*{Lemma 2.3}, for every $\chi \in \fX$, we have 
\be \label{e:gausseps}
\Ga_\psi(x,\chi)= \begin{cases}1 & \text{ if } a(\chi) = 0 \x{ and } \val_F(x) \ge 0, \\ -\frac{1}{\qF -1} &\text{ if } a(\chi) = 0 \x{ and } \val_F(x) = -1, \\  \frac{\qF^{1-a(\chi)/2}}{q_F - 1}\eps(\frac 12, \chi^{-1}, \psi)\chi(x\i)  &\text{ if } a(\chi) > 0 \x{ and } \val_F(x)=-a(\chi), \\ 0 & \text{ otherwise.} \end{cases}
\ee
We will use this together with properties of $\eps$-factors: for instance, for a multiplicative character $\chi\colon F^\times \ra \bC^\times$, a nontrivial additive character $\psi\colon F \ra \bC^\times$, and any $s \in \bC$, by \cite{Sch02}*{Section 1.1}, we have
\be\label{e:epspsi}
\tst \eps(s, \chi, \psi) = \eps(\frac 12, \chi, \psi) \qF^{(c(\psi)-a(\chi))(s - \frac 12)} \qxq{and} \eps(\frac 12, \chi, a\psi)  = \chi(a) \eps(\frac 12, \chi, \psi) \qxq{for}  a \in F^\times,
\ee
where $a\psi \colon F \ra \bC^\times$ is the character $x \mapsto \psi(ax)$. In particular, there is little harm in restricting to $s = \frac12$ and assuming the common normalization $c(\psi)=0$, under which, by \emph{loc.~cit.},~we have
\be \label{e:epsunramtwist}
\tst\eps(\frac 12, \chi\chi', \psi) = \chi'(\varpi_F)^{a(\chi)} \eps(\frac 12, \chi, \psi) \qxq{and} \eps(\frac 12, \chi', \psi) = 1 \qxq{whenever} a(\chi')=0,\ee
\be \label{e:epsduality}
\tst\eps(\frac 12, \chi, \psi) \eps(\frac 12, \chi^{-1}, \psi) = \chi(-1) \qxq{and, if $\chi$ is unitary, then} |\eps(\frac 12, \chi, \psi)| = 1.
\ee
\epp


Due to \eqref{e:gausseps}, the only case in which the study of the $p$-adic properties of $\Ga_\psi(x,\chi)$ has substance is when $\chi$ is ramified and $\val_F(x) = -a(\chi)$. Moreover, by a change of variables,
\[
\Ga_\psi(xu,\chi) = \chi(u^{-1})\Ga_\psi(x,\chi) \qxq{for} u \in \cO_F^\times, \qxq{so it suffices to consider} \Ga_\psi(\varpi_F^{-a(\chi)},\chi).
\]
We will analyze the latter below, and we begin in \Cref{p:maingauss} with the case $a(\chi)=1$, a case whose study reduces to that of classical Gauss sums of multiplicative characters of finite fields.

\bpp[Finite field Gauss sums] \label{rem:divibility of gauss sums by other primes} \label{pp:gauss-sum-finite-field}
For a finite extension $\bF/\bF_p$, a character $\chi:\bF^\times\to\C^\times$, and a nontrivial additive character $\psi:\bF\to \C^\times$,
the \emph{classical Gauss sum} of $\chi$ (with respect to $\psi$) is
\[
\tst g_\psi(\chi)\ce -\sum_{a\in\bF^\times}\chi(a)\psi(a), \qxq{so that} g_\psi(\chi) \in \Z[\zeta_{\#\bF-1},\zeta_p].
\]
By, for instance, \cite{Was97}*{Lemma~6.1}, we have
\[
g_\psi(\mathbf{1})=1 \qxq{and} g_\psi(\chi)\overline{g_\psi(\chi)} = \#\bF \qxq{for} \chi \neq \mathbf{1},
\]
so the prime ideals of $\Q(\zeta_{\#\bF-1},\zeta_p)$ that divide $g_\psi(\chi)$ all lie above $p$ and
\be\label{l:gaussquadfinitef}
\x{if} \q \chi^2=\mathbf{1}, \qx{then} \q g_\psi(\chi)^2 = \chi(-1)\cdot \#\bF. 
\ee
We will be interested in $\val_p(g_\psi(\chi))$ for the $p$-adic valuation $\val_p$ determined by a choice of an isomorphism $\iota\colon \overline{\Q}_p \simeq \C$. Via Teichm\"uller representatives, the latter determines a character
\[
\omega_{\bF}:\bF^\times\to\C^\times \qxq{of order} \#\bF-1 \qxq{such that} \omega_{\bF}(a) \equiv a \bmod p.
\]
Thus, every $\chi:\bF^\times\to\C^\times$ is of the form $\chi = \omega_{\bF}^{-\alpha(\chi)}$ for a unique
$0\le \alpha(\chi) < \#\bF-1$, and we set
\[
\tst s(\chi) \ce \sum_{i = 0}^{[\bF:\Fp]-1} a_i, \qxq{where}  \alpha(\chi) = \sum_{i = 0}^{[\bF:\Fp]-1} a_i p^{i}, \q 0\le a_i \le p-1, \qx{is the base-$p$ expansion}
\]
($s(\chi)$ and $\alpha(\chi)$ depend on the implicitly fixed $\iota$; abusively, we also extend this notation to characters $\wt{\chi} \colon F^\times \ra \bC^\times$ with $a(\wt{\chi}) \le 1$, where $F/\bQ_p$ is a finite extension with residue field $\bF$). Certainly,
\be \label{e:schiquadratic}
\tst 0 \le s(\chi) \le (p-1)[\bF:\Fp] \qxq{with} s(\chi) =\begin{cases} 0 &\x{if and only if $\chi = \b1$,} \\  \frac{(p-1)[\bF:\Fp]}{2} &\text{if $p$ is odd, }~\chi^2 = \mathbf{1},~\chi \ne \mathbf{1}.\end{cases}
\ee
By \cite[Lemma 6.11, Proposition 6.13]{Was97}, we have
\be \label{e:schiprop1}
s(\chi \chi') \equiv s(\chi)+s(\chi') \bmod{p-1}, \q
0 \le  s(\chi \chi') \le s(\chi) + s(\chi').
\ee
In particular, since, for a finite extension  $\bF'/\bF$, both $\omega_{\bF'}|_{\bF} = \omega_{\bF}$ and $\omega_\bF \circ \mathrm{Norm}_{\bF'/\bF} = \prod_{i = 0}^{[\bF' : \bF] - 1}\omega_{\bF'}^{ (\#\bF)^i }$,
\be \label{e:schiprop2}
s(\xi|_{\bF^\times}) \equiv s(\xi) \bmod{p-1} \ \ \x{and}\ \ s(\chi \circ \mathrm{Norm}_{\bF'/\bF}) \equiv [\bF' : \bF] s(\chi) \bmod p - 1
\ \ \x{for}\ \ \xi \colon \bF'^\times \ra \bC^\times.
\ee
By a special case of Stickelberger's congruence, that is, by \cite{Was97}*{Proposition 6.13 and before Lemma 6.11},
\[
\tst \val_p(g_\psi(\chi)) = \frac{s(\chi)}{p-1}, 
\]
and this key identity gives the following result.
\epp


\bprop\label{lem:gauss sum of conductor 1} \label{p:maingauss} \label{a(mu)=1}
For a finite extension $F/\Q_p$, a multiplicative character $\chi \colon F^\times \ra \bC^\times$ with $a(\chi)\le~\!1$, an additive character $\psi \colon F \ra \bC^\times$ with $c(\psi)=0$, an $x \in \varpi_F^{-1} \cO_F^\times$, and an isomorphism $\bC \simeq \ov{\bQ}_p$,
\[
\tst \val_p (\Ga_\psi(x,\chi)) = \frac{s(\chi)}{p-1},
\]
and $(\qF-1)\Ga_\psi(x,\chi)$ is an algebraic integer in $\Q(\zeta_{\qF-1},\zeta_p)$ that is a unit away from $p$.
\eprop

\begin{proof}
Since $a(\chi) \le 1$, we may view $\chi$ as a nontrivial character of $\bF_F^\times$. Moreover, since $c(\psi) = 0$, the character $\psi$ defines a nontrivial additive character $\ov{\psi} \colon \bF \ra \bC^\times$ by $\ov{\psi}(t \bmod \fm_F) \ce \psi(\varpi_F^{-1}t)$. The definitions reviewed in \S\S\ref{s:localgauss}--\ref{pp:gauss-sum-finite-field} give $\Ga_\psi(\varpi_F^{-1},\chi) = -\frac{g_{\ov{\psi}}(\chi)}{q_F-1}$, so \S\ref{pp:gauss-sum-finite-field} 
gives the claims.
\end{proof}


A similar analysis of $\Ga_\psi(\varpi_F^{-a(\chi)},\chi)$ for $a(\chi) \ge 2$ in \Cref{prop:main gauss for a(mu)>1} will use the following lemmas whose goal is to express this Gauss sum more or less explicitly.



\begin{lemma}\label{l:vunit}
For a finite extension $F/\bQ_p$, a multiplicative character $\chi \colon F^\times \ra \bC^\times$ with $a(\chi) \ge 2$, and an additive character $\psi \colon F \ra \bC^\times$ with $c(\psi) = 0$, there is a $u \in \cO_F^\times$ such that
\benumr
\item \label{item: mult -> additive, even case}
if $a(\chi)$ is even, then
\[
\qq \chi(1 + \varpi_F^{a(\chi)/2}x) = \psi(u \varpi_F^{-a(\chi)/2}x) \qxq{for all} x \in \cO_F;
\]

\item \label{item: mult -> additive, odd case}
if  $a(\chi)$ is odd, then 
\[
\qq \chi(1 + \varpi_F^{(a(\chi)+1)/2}x) = \psi(u \varpi_F^{-(a(\chi)-1)/2}x) \qxq{for all} x \in \cO_F;
\]

\item \label{item: mult -> additive, odd case 2}
if both $p$ and $a(\chi)$ are odd, then 
\[
\qq \tst \chi(1 + \varpi_F^{(a(\chi)-1)/2}x) = \psi(u(\varpi_F^{-(a(\chi)+1)/2} x - \frac{\varpi_F^{-1} x^2}{2})) \qxq{for all} x \in \OF.
\]
 \end{enumerate}
\end{lemma}

\begin{proof}
We set $\epsilon \ce 0$ if $a(\chi)$ is even and $\epsilon \ce 1$ if $a(\chi)$ is odd, so that the map $x\mapsto  \chi(1 + \varpi_F^{(a(\chi)+\epsilon)/2}x)$ is an additive character $\theta\colon F \ra \bC^\times$ with $c(\theta) = (a(\chi)-\epsilon)/2$. All such characters have the form $x\mapsto \psi(u\varpi_F^{-(a(\chi)-\epsilon)/2}x)$ for some $u\in\cO_F^\times$ (see \cite{BH06}*{Section 1.7, Proposition}), so \ref{item: mult -> additive, even case} and \ref{item: mult -> additive, odd case} follow.

For \ref{item: mult -> additive, odd case 2}, let $U \subset \Oix$ be a set of representatives of $\Oix/(1 + \fm_F^{(a(\chi)+1)/2})$ and  consider the maps
$$
\tst \chi_u \colon 1 + \fm_F^{(a(\chi) - 1)/2} \ra \bC^\times \qxq{for} u \in U
$$
defined by
\[
\tst \chi_u(1 + \varpi_F^{(a(\chi)-1)/2} x) \ce \psi(u(\varpi_F^{-(a(\chi)+1)/2}x - \frac{\varpi_F^{-1} x^2}{2})) = \psi(u\varpi_F^{-a(\chi)}(\varpi_F^{(a(\chi) -1)/2}x - \frac{(\varpi_F^{(a(\chi) -1)/2}x)^2}{2})).
\]
Thanks to the power series expansion $z - \frac{z^2}{2} + \dotsc$ of $\log(1 + z)$, the function $\chi_u$ is a multiplicative character that is trivial on $1 + \fm_F^{a(\chi)}$ but not on $1 + \fm_F^{a(\chi) - 1}$. Moreover, since the characters $(u\psi)|_{\fm_F^{-(a(\chi) + 1)/2}}$ are pairwise distinct (compare with the proof of \cite{BH06}*{Section 1.7, Proposition}), so are the $\chi_u$. Thus, since $\#U = q_F^{(a(\chi)-1)/2}(q_F-1)$, the $\chi_u$ exhaust the set of those multiplicative characters on $(1 + \fm_F^{(a(\chi) - 1)/2})/(1 + \fm_F^{a(\chi)})$ that are nontrivial on $1 + \fm_F^{a(\chi) - 1}$. Consequently, $\chi = \chi_u$ for some $u$, as desired, and, certainly, this $u$ is also a valid choice for part \ref{item: mult -> additive, odd case}.
\end{proof}


\begin{lemma}\label{l:evaluate gauss sums with a(mu)>1}
For a finite extension $F/\bQ_p$, a multiplicative character $\chi \colon F^\times \ra \bC^\times$ with $a(\chi) \ge 2$, an additive character $\psi \colon F \ra \bC^\times$ with $c(\psi) = 0$, and a $u \in \cO_F^\times$ as in Lemma \uref{l:vunit},
\benumr
\item \label{EG-a}
if $a(\chi)$ is even, then
\[
\tst \Ga_\psi(\varpi_F^{-a(\chi)},\chi) = \frac{\qF^{1-a(\chi)/2}}{\qF-1} \psi(-u\varpi_F^{-a(\chi)})\chi(-u).
\]
\item \label{EG-b}
if $a(\chi)$ is odd, then
$$
\tst\ \qqq\Ga_\psi(\varpi_F^{-a(\chi)},\chi) = \frac{\qF^{-(a(\chi)-1)/2}}{\qF-1}\psi(-u\varpi_F^{-a(\chi)}) \sum\limits_{t \in \cO_F/\fm_F} \chi(-u - ut\varpi_F^{(a(\chi)-1)/2}) \psi(-ut\varpi_F^{-(a(\chi)+1)/2})
$$
where we sum over coset representatives 
\up{their choice does not affect the summands}. 
\end{enumerate}
\end{lemma}

\begin{proof}
We again set $\epsilon \ce 0$ if $a(\chi)$ is even and $\epsilon \ce 1$ if $a(\chi)$ is odd. Letting $d^\times y$ and $dy$ be the Haar measures on $F^\times$ and $F$ normalized by $\int_{\Oi^\times} d^\times y =1$ and $\int_{\Oi} dy =1$, respectively, we then have
\begin{align*}
\Ga_\psi(\varpi_F^{-a(\chi)}, \chi) &=\int_{y\in \cO_F^\times}\psi(\varpi_F^{-a(\chi)}y)  \chi(y) d^\times y \\
&=\sum\limits_{v\in \cO_F^\times/(1+\fm_F^{(a(\chi)+\epsilon)/2})}  \chi(v) \int_{y \in (1+\fm_F^{(a(\chi)+\epsilon)/2})}\psi(\varpi_F^{-a(\chi)} vy)  \chi(y) d^\times y
\end{align*}
where the sum is over some \emph{fixed} coset representatives $v \in \cO_F^\times$. The integral in this sum equals
\begin{align*}
 \frac{\qF}{\qF-1} \int&_{y \in (1+\fm_F^{(a(\chi)+\epsilon)/2})} \psi(\varpi_F^{-a(\chi)} vy)  \chi(y) d y  \\
 &=\frac{\qF^{1-(a(\chi)+\epsilon)/2}}{\qF-1}\psi(\varpi_F^{-a(\chi)}v)\int_{y \in \cO_F}\psi(\varpi_F^{-(a(\chi)-\epsilon)/2} vy) \chi(1 + \varpi_F^{(a(\chi)+\epsilon)/2}y)  dy \\
 &\overset{\ref{l:vunit}}{=}\frac{\qF^{1-(a(\chi)+\epsilon)/2}}{\qF-1} \psi(\varpi_F^{-a(\chi)}v)\int_{y \in \cO_F}\psi( \varpi_F^{-(a(\chi)-\epsilon)/2} (v + u)y)dy.
\end{align*}
The latter vanishes unless the integrand defines the trivial additive character of $\cO_F$, that is, unless, $v\equiv -u\bmod \fm_F^{(a(\chi)-\epsilon)/2}$. If $a(\chi)$ is even, this happens precisely when $v$ is in the coset $-u(1+\fm_F^{a(\chi)/2})$, and \ref{EG-a} follows. If $a(\chi)$ is odd, then the same happens precisely when $v$ is in a coset of the form $(-u+t\varpi_F^{(a(\chi)-1)/2})(1+\fm_F^{(a(\chi)+1)/2})$ with $t\in\cO_F$, and two such cosets are distinct if and only if the corresponding $t$ are distinct modulo $\fm_F$. Thus, by choosing coset representatives $t$ for $\cO_F/\fm_F$ and readjusting our choices of coset representatives $v$, for odd $a(\chi)$ we obtain
\[
\Ga_\psi(\varpi_F^{-a(\chi)}, \chi) = \frac{\qF^{1-(a(\chi)+1)/2}}{\qF-1}\sum_{t\in \cO_F/\fm_F}   \chi\left(-u+t\varpi_F^{(a(\chi)-1)/2}\right) \psi\left(\varpi_F^{-a(\chi)}(-u+t\varpi_F^{(a(\chi)-1)/2})\right).
\]
To conclude \ref{EG-b}, it remains to adjust the representatives $t$ by replacing them by $-ut$. Finally, by \Cref{l:vunit}~\ref{item: mult -> additive, odd case}, the summands in \ref{EG-b} are independent of the coset representatives for $\cO_F/\fm_F$.
\end{proof}

\bthm\label{prop:main gauss for a(mu)>1} \label{a(mu)>1}
For a finite extension $F/\Q_p$, a multiplicative character $\chi \colon F^\times \ra \bC^\times$ with $a(\chi)\ge 2$, and an additive character $\psi \colon F \ra \bC^\times$ with $c(\psi)=0$,
\[
\tst  \qF^{-1 + a(\chi)/2}(\qF-1)\Ga_\psi(x,\chi) \qxq{is a root of unity for every} x \in \varpi_F^{-a(\chi)} \cO_F^\times.
\]
\ethm

\bpf
The case of an even $a(\chi)$ follows from  \Cref{l:evaluate gauss sums with a(mu)>1}~\ref{EG-a} (with \eqref{e:gausseps} to replace $x$ by $\varpi_F^{-a(\chi)}$). Thus, we assume that $a(\chi)$ is odd, choose a $u \in \cO_F^\times$ as in \Cref{l:vunit}, and, by \Cref{l:evaluate gauss sums with a(mu)>1}~\ref{EG-b}, need to show that $q_F^{-1/2}T$ is a root of unity where
\[
\tst T\ce \sum_{t \in \Oi/\fm_F} F(t) \qxq{with} F(t) \ce  \chi(1+t\varpi_F^{(a(\chi)-1)/2})\psi(-ut \varpi_F^{-(a(\chi)+1)/2}),
\]
so that $F(t)$ only depends on the class in $\cO_F/\fm_F$ of the representative $t$. For odd $p$, by \Cref{l:vunit}~\ref{item: mult -> additive, odd case 2},
\[
\tst T = \sum_{t \in \cO_F/\fm_F} \psi\left(- \frac{u t^2\varpi_F^{-1}}{2}\right).
\]
Thus, for odd $p$, letting $\psi'\colon \FF \ra \bC^\times$ be the nontrivial additive character  $t \mapsto  \psi\left(- \frac{u t\varpi_F^{-1}}{2}\right)$ and $\chi' \colon \FF^\times \ra \bC^\times$  the unique nontrivial quadratic character, we have
\[
\tst T =  1 + \sum_{t \in \FF^\times} \psi'(t^2) = 1 + \sum_{t \in \FF^\times} (\chi'(t)+1)\psi'(t) = -g_{\psi'}(\chi').
\]
Consequently, \eqref{l:gaussquadfinitef} shows that $q_F^{-1/2}T$ is a root of unity for odd $p$.

In the remaining case $p = 2$, we instead let $\psi'\colon \FF \ra \{\pm 1\} \subset \bC^\times$ be the nontrivial additive character $t \mapsto \chi(1+t\varpi_F^{a(\chi)-1})$ and seek to conclude by showing that $q_F^{-1}T^2$ is a root of unity. For this, we first note that, since $F(2t) = F(0) = 1$, the identity
\[
F(t)F(t') = \chi(1+(t+t')\varpi_F^{(a(\chi)-1)/2}+tt'\varpi_F^{a(\chi)-1})\psi(-u(t+t')\varpi_F^{-(a(\chi)+1)/2}) = F(t+t')\psi'(tt')
\]
applied in the case $t = t'$ shows that each $F(t)$ is a fourth root of unity. We obtain
\[
\tst T^2 = \sum_{t,\,t'\in\FF}F(t)F(t') = \sum_{t,\,t'\in\FF}F(t+t')\psi'(tt') = \sum_{s\in\FF}\ppp{F(s)\sum_{t\in\FF}\psi'(t^2+ts)},
\]
where, since $t \mapsto t^2$ is an automorphism of $\bF_F$ and $\psi'$ is nontrivial, the inner sum vanishes for $s = 0$. For $s \neq 0$, the kernel of the $\bF_2$-linear map $\bF_F \ra \bF_F$ given by $t \mapsto t^2 + ts$ is $\{0,s\}$, so its image is an $\bF_2$-hyperplane $H_s \subset \bF_F$, and hence the inner sum also vanishes if $H_s \neq \Ker(\psi')$ and else equals $q_F$. Thus, we are reduced to showing that there is a unique $s \in \bF_F \setminus \{ 0 \}$ with $H_s = \Ker(\psi')$ or, since the total number of $\bF_2$-hyperplanes in $\bF_F$ is $q_F - 1$, that the $H_s$ exhaust all such hyperplanes.


Scaling by a fixed $r \in \bF^\times_F$ is an $\bF_2$-linear automorphism of $\bF_F$, and the nonzero orbits of this automorphism all have the same order equal to the order $m$ of $r$ in the group $\bF_F^\times$. Thus, scaling by $r$ fixes no $\bF_2$-hyperplane $H \subset \bF_F$ unless $r = 1$: else $m$ would divide the consecutive integers $\#(H \setminus \{ 0 \})$ and $\#(\bF_F \setminus H)$. Consequently, by scaling, $\bF_F^\times$ acts transitively on the set of $\bF_2$-hyperplanes $H \subset \bF_F$ and it remains to note that scaling by an $r \in \bF_F^\times$ brings $H_s = \{t^2 + st \, \vert\, t \in \bF_F\}$ to another hyperplane of this form, namely, to $H_{r's}$ for the unique $r' \in \bF_F$ with $r'^2 = r$.
\epf

The above analysis of Gauss sums $\Ga_\psi(x,\chi)$ gives the following consequence for $\eps$-factors of $\GL(1)$.


\bcor\label{p:maingausseps}For a finite extension $F/\Q_p$, a multiplicative character $\chi \colon F^\times \ra \bC^\times$ of finite order, and a nontrivial additive character $\psi \colon F \ra \bC^\times$, 
we have
\be \label{MGE-eq}
\tst \eps(\frac 12, \chi, \psi) \in \overline{\Z}[\frac{1}{p}]^\times.
\ee
Moreover, for any  isomorphism $\bC \simeq \ov{\bQ}_p$,
\benumr
\item \label{MGE-i}
if $a(\chi)=1$, then, 
with the notation of \uS\uref{pp:gauss-sum-finite-field},
\[
\tst \val_p (\eps(\frac 12,\chi, \psi)) = - \frac{[\FF:\Fp]}{2} + \frac{s(\chi\i)}{p-1}; 
\]

\item \label{MGE-ii}
if $\chi^2=1$ or $a(\chi)>1$, then $\eps(\frac 12, \chi, \psi)$  is a root of unity, and so $\val_p (\eps(\frac 12,\chi, \psi))=0$.
    \end{enumerate}
\ecor

\begin{proof}
By \eqref{e:epspsi}, we may assume that $c(\psi)=0$. The twist by an unramified character formula \eqref{e:epsunramtwist} then settles the case $a(\chi) = 0$ and also allows us to assume that $\chi(\varpi_F)=1$, that is, that $\chi \in \fX$. In this remaining case of a $\chi \in \fX$ with $a(\chi) > 0$, by \eqref{e:gausseps}, we have
\[
\tst \eps(\frac 12, \chi, \psi) = \frac{q_F - 1}{\qF^{1-a(\chi)/2}}\Ga_\psi(\varpi_F^{-a(\chi)}, \chi\i)\chi(\varpi_F^{a(\chi)}).
\]
In particular, \Cref{p:maingauss} and \Cref{a(mu)>1} give $\eps(\frac 12, \chi, \psi) \in \overline{\Z}[\frac{1}{p}]^\times$ as well as \ref{MGE-i} and the $a(\chi) > 1$ case of \ref{MGE-ii}. The remaining $\chi^2 = 1$ case of \ref{MGE-ii} follows from \eqref{e:epsduality}.
\end{proof}

We conclude the section with an explicit analysis of the $\eps$-factors of quadratic characters of $\bQ_2^\times$. This will be useful for studying the $2$-adic properties of Fourier expansions of newforms.


\bpp[Quadratic characters of $\Q_2^\times$]\label{s:type1b}
There are eight characters $\beta\colon \Q_2^\times \ra \bC^\times$ with $\beta^2 = \mathbf{1}$:
$$\Xc^{\rm{quad}}_{\Q_2} := \{ \mathbf{1}, \beta_0, \beta_2, \beta_0 \beta_2, \beta_3, \beta_0 \beta_3, \beta_2\beta_3, \beta_0 \beta_2\beta_3\},$$
where $\mathbf{1}$ is the trivial character, $\beta_0$ is nontrivial and unramified, the conductor exponents of $\beta_2$ and $\beta_0 \beta_2$ are $2$, and those of $\beta_3$, $\beta_0 \beta_3$, $\beta_2\beta_3$, and $\beta_0 \beta_2\beta_3$ are $3$. To normalize for the sake of concreteness: via local class field theory, $\beta_0$ corresponds to the extension $\Q_2(\sqrt{5})/\Q_2$ and satisfies $\beta_0(2)=-1$, whereas $\beta_2$ corresponds to the extension $\Q_2(\sqrt{-1})$ and satisfies $\beta_2(2)=1$, and $\beta_3$ corresponds to $\Q_2(\sqrt{2})/\Q_2$ and satisfies $\beta_3(2) = 1$ (so $\beta_2\beta_3$ corresponds to $\Q_2(\sqrt{-2})/\Q_2$). In the notation of $\S \ref{s:localgauss}$,
$$\Xc_{\Q_2, 1} = \emptyset, \quad \Xc_{\Q_2, 2}=\{\beta_2\}, \quad \Xc_{\Q_2, 3}= \{\beta_3, \beta_2\beta_3\}.$$
 \epp


 \begin{lemma}\label{l:gausssums2}
 For an additive character $\psi: \Q_2 \rightarrow \C^\times$ with $c(\psi)=0$, there is an $a_\psi \in \Z_2^\times$ with
 \[
 \tst \eps(\frac 12,\beta_2,\psi) = \beta_2(a_\psi)\cdot i, \quad
 \eps(\frac 12,\beta_{3},\psi) = \beta_{3}(a_\psi), \quad \eps(\frac 12,\beta_2\beta_3,\psi) = (\beta_2\beta_3)(a_\psi)\cdot i.
 \]
 \end{lemma}
 \begin{proof}
 The collection of $\psi$ with $c(\psi) = 0$ is a $\bZ_2^\times$-torsor via the action $(a\psi)(x) \ce \psi(ax)$ (see \cite{BH06}*{Section 1.7, Proposition}), so the $\eps$-factor transformation formula \eqref{e:epspsi} reduces us to treating a single $\psi$. We then choose the following $\psi$ with $c(\psi) = 0$ for which we will argue the claim with $a_\psi \ce 1$:
 \[
\tst  \psi(x) \ce \exp(2\pi i\lambda(x)) \qxq{where} \lambda \colon \bQ_2 \surjects \bQ_2/\bZ_2 \hra \bQ/\bZ \cong \bigoplus_{\x{prime }p} \bQ_p/\bZ_p.
 \]
 With the shorthand $\zeta_n \ce e^{2 \pi i /n}$, we obtain
 \[\ba
\tst \Ga_{\psi}(\frac 14,\beta_2) &= \tst \frac12 \left(\zeta_4 \cdot \beta_2(1) + \zeta_4^3\cdot \beta_2(3)\right) = \frac12 \left(i + i\right) = i, \\
\tst \Ga_{\psi}(\frac 18,\beta_{3}) &=\tst  \frac14 \left(\zeta_8\cdot \beta_{3}(1) +  \zeta_8^3\cdot\beta_{3}(3)+ \zeta_8^5\cdot \beta_{3}(5) + \zeta_8^7\cdot \beta_{3}(7)\right)  = \frac14 \left(\zeta_8 -\zeta_8^3 -\zeta_8^5 +\zeta_8^7\right) = \frac{1}{2^{1/2}}, \\
\tst \Ga_{\psi}(\frac 18,\beta_2\beta_3) &=\tst  \frac14 \left(\zeta_8\cdot (\beta_2\beta_3)(1) + \dotsc  
+ \zeta_8^7\cdot (\beta_2\beta_3)(7)\right)  = \frac14 \left(\zeta_8 +\zeta_8^3 -\zeta_8^5 -\zeta_8^7\right) = \frac{1}{2^{1/2}}i.
\ea
\]
Thus, \eqref{e:gausseps} gives the desired
\[ 
\tst \eps(\frac 12,\beta_2,\psi) = i, \quad
 \eps(\frac 12,\beta_{3},\psi) = 1, \quad \eps(\frac 12, \beta_2\beta_3,\psi) = i. \qedhere
 \]
 \end{proof}


\section{$p$-adic properties of local Whittaker newforms} \label{s:valuations of whittaker newforms}


As we will see in \S\ref{s:global p-adic valuations}, the theory of Whittaker models translates the study of $p$-adic properties of Fourier expansions of newforms $f$ at cusps into the study of $p$-adic properties of the values of the Whittaker newform of the $p$-component of the associated cuspidal automorphic representation $\pi_f$. This transforms a global problem into a purely local one, and in this section we place ourselves in the resulting local setting. Namely, we use the theory of local Fourier expansions of the Whittaker newform $W_{\pi,\,\psi}$ of an irreducible, admissible, infinite-dimensional representation $\pi$ of $\GL_2(\bQ_p)$, the recent basic identity (reviewed in \S\ref{pp:basic-identity}) that explicates the resulting local Fourier coefficients, the work of \S\ref{s:prelim-local} on Gauss sums, and the classification of $\pi$ to derive in \Cref{T1,T2} explicit lower bounds on the $p$-adic valuations of values of $W_{\pi,\,\psi}$. We begin by reviewing the local Whittaker newform $W_{\pi,\,\psi}$ in \S\ref{def:localwhit} and its Fourier expansions in \S\ref{pp:basic-identity}.



\bpp[Representations of $\GL_2(F)$ and their conductors] \label{pp:rep-thy-notation}
Let $p$ be any prime, $F/\Q_p$ a finite extension and $\pi$ an irreducible, admissible, infinite-dimensional, complex representation of $\GL_2(F)$ with central character $\omega_{\pi}$ and contragredient $\wt{\pi}$. 
For a character $\chi\colon \Fx \ra \bC^\times$, the twist
\[
\chi\pi \qxq{is the complex representation of $\GL_2(F)$ given by} g \mapsto \chi(\det(g)) \tensor_\bC \pi(g),
\]
so that, for instance, $\omega_{\pi}^{-1}\pi \simeq \wt{\pi}$ 
(see \cite{Del73a}*{\'{e}quation (3.2.2.2)}).
For  $n\geq 0$, we consider the subgroup
\[
K_{1}(n) \ce \{ \ppp{\begin{smallmatrix} a&b\\c&d\end{smallmatrix}}\in \GL_2(\cO_F)\ \vert\ c\in \varpi_F^{n}\cO_F,\ a\in 1+\varpi_F^{n}\cO_F\} \subset \GL_2(\cO_F).
\]
There is the smallest $a(\pi)\geq 0$, the \emph{conductor exponent} of $\pi$, such that the space of $K_{1}(a(\pi))$-fixed vectors in $\pi$ is nonzero, and so necessarily is one-dimensional (see \cite{Del73a}*{Th\'{e}or\`{e}me 2.2.6, D\'{e}finition 2.2.7}). 
For computing $a(\chi\pi)$, we will use \cite{CS18}*{Lemma 2.7}: for $\pi$ and $\chi$ as above with $\omega_\pi = \mathbf{1}$, we have
	\begin{equation} \label{cond-twist}
	a(\chi\pi)\leq \max\{a(\pi),2a(\chi)\}
	\end{equation}
with equality if either $a(\chi) \neq \frac{a(\pi)}{2}$ or $\pi$ is \emph{twist-minimal} in the sense that $a(\pi) = \min_\chi(a(\chi\pi))$, so that, in particular, a $\pi$ with $\omega_\pi = \mathbf{1}$ is twist-minimal whenever $a(\pi)$ is odd. 

For a nontrivial additive character $\psi \colon F \ra \bC^\times$, similarly to \S\ref{s:localgauss}, we let $\eps(s, \pi, \psi) \in \bC^\times$ be the local $\eps$-factor of $\pi$ (see \cite{Sch02}*{Section 1.1} for its review) and abbreviate to $\eps(s, \pi)$ when $\psi$ satisfies $c(\psi) = 0$ (see \S\ref{conv}). This minor abuse is harmless when $\omega_\pi$ is unramified because, by \emph{loc.~cit.}, we have
\[
\tst \eps(s, \pi, \psi) = \eps(\frac 12, \pi, \psi) \qF^{(2c(\psi)-a(\pi))(s - \frac 12)} \qxq{and} \eps(\frac 12, \pi, a\psi)  = \omega_\pi(a) \eps(\frac 12, \pi, \psi) \qxq{for}  a \in F^\times
\]
(compare with \eqref{e:epspsi}). With the common normalization $c(\psi) = 0$, we also have (\emph{loc.~cit.})
\[
\eps(s, | \cdot |^t\pi, \psi) = \qF^{-a(\pi)t} \eps(s, \pi, \psi) \qxq{for} t \in \bC,
\]
\begin{equation} \label{eqn:eps-product}
\tst \eps(s, \pi, \psi)\eps(1 - s,\omega_{\pi}^{-1}\pi, \psi) = \omega_\pi(-1), \qxq{so} \eps(\frac12, \pi, \psi)= \pm 1 \qxq{whenever} \omega_\pi = \mathbf{1}.
\end{equation}
\epp



\bpp[The Whittaker newform of $\pi$]\label{def:localwhit} \label{pp:Whittaker}
For a nontrivial additive character $\psi \colon F \ra \bC^\times$, we set
\[
\cW_\psi \ce \{ \x{locally constant } W \colon \GL_2(F) \ra \bC \x{ with } W(
\ppp{\begin{smallmatrix}
1& x\\&1
\end{smallmatrix}} g)=\psi(x)W(g) \x{ for } 
x\in F,\, g\in \GL_2(F)\}.
\]
The group $\GL_2(F)$ acts on the $\bC$-vector space $\cW_\psi$ by $(g'W)(g) \ce W(gg')$ and, by \cite{Del73a}*{before Proposition 2.2.3}, each $\pi$ as in \S\ref{pp:rep-thy-notation} 
is isomorphic 
to the unique subspace $\Wh_\psi(\pi) \subset \cW_\psi$, the \emph{Whittaker model} of $\pi$. The \textit{normalized Whittaker newform} of $\pi$ is the unique $K_{1}(a(\pi))$-invariant element 
\[
W_{\pi,\,\psi}\in\Wh_\psi(\pi) \qxq{such that}  W_{\pi,\,\psi}(1)=1.
\]
For an unramified multiplicative character $\chi\colon F^\times \ra \bC^\times$, 
we have\footnote{
The map $\iota_\chi\colon W \mapsto (g \mapsto \chi(\det(g))W(g))$ is a $\bC$-linear automorphism of $\cW_\psi$ such that
\[
\chi(\det(g'))(\iota_\chi(g' W)) =  g' (\iota_\chi(W)) \qxq{for} g' \in \GL_2(F).
\]
Thus, $\iota_\chi$ induces a $\GL_2(F)$-isomorphism $\wt{\iota}_\chi\colon\cW_\psi \isomto \chi\i\cW_\psi$, so that $\iota_\chi(\cW_\psi(\pi)) = \cW_\psi(\chi\pi)$ and $\iota_\chi(W_{\pi,\,\psi}) = W_{\chi\pi,\,\psi}$.}
\be \label{lem:twist by unramified finite order char}
W_{\chi\pi,\,\psi}(g) = \chi(\det(g))W_{\pi,\,\psi}(g) \qxq{for all} g\in \GL_2(F).
\ee
\epp

\bpp[The coset representatives $g_{t,\,\ell,\,v}$] \label{elements-g}
The values of the Whittaker newform $W_{\pi,\,\psi}$ on the double coset $Z(F)U(F)gK_1(a(\pi))$, where  $Z \subset \GL_2$ is the center and $U \subset \GL_2$ the ``upper right'' unipotent subgroup, 
are determined by $W_{\pi,\,\psi}(g)$. We choose the representatives $g$ as follows: we set
\[
g_{t,\,\ell,\,v} \ce \ppp{\begin{smallmatrix}
\varpi_F^t&\\&1
\end{smallmatrix}} \ppp{\begin{smallmatrix}
& 1\\ -1&
\end{smallmatrix}} \ppp{\begin{smallmatrix}
1& v\varpi_F^{-\ell} \\&1
\end{smallmatrix}}
= \sabcd{}{\varpi_F^{t}}{-1}{-v\varpi_F^{-\ell}} \in \GL_2(F) \qxq{for} t,\,\ell\in\Z \qxq{and} v\in\cO_F^\times
\]

and recall from \cite{Sah16}*{Lemma~2.13} that, letting $v$ range over the indicated coset representatives,\footnote{\label{foot-dec}One argues the decomposition as follows. For the upper triangular Borel $B \subset \GL_2$, the valuative criterion of properness for $B\backslash\GL_2$
and the vanishing $H^1(\cO_F, B) = \{ * \}$ show that $\GL_2(\cO_F) \surjects (B \backslash \GL_2)(F)$, and so give the Iwasawa decomposition $\GL_2(F) = B(F)\GL_2(\cO_F)$, which one refines to $\GL_2(F) =\ppp{Z(F)U(F)\abcd{\{\varpi_F^a\}_{a \in \bZ}}001} \GL_2(\cO_F)$. The advantage of the refinement is that the group encoding the nonuniqueness of the decomposition shrinks from $B(\cO_F) = B(F) \cap \GL_2(\cO_F)$ to $Z(\cO_F)U(\cO_F) = \{ \abcd z u 0 z\, \vert\, z \in \cO_F^\times, u \in \cO_F\}$. This group acts on  the primitive vectors $\left(\begin{smallmatrix}x\\y \end{smallmatrix}\right)$ with entries in $\cO_F/\fm_F^n$ by left multiplication: $\left(\abcd z u 0z, \left(\begin{smallmatrix}x\\y \end{smallmatrix}\right)\right) \mapsto \left(\begin{smallmatrix}zx + uy \\ zy \end{smallmatrix}\right)$. The orbits are indexed by both the ``valuation'' $0 \le \ell \le n$ of $y$ and, with the subsequent normalization $y = \varpi_F^\ell$, the  class $\ov{x}$ of $x$ in $\cO_F/(1 + \fm_F^{\min(\ell,\, n - \ell)})$. Since $K_1(n)$ is the stabilizer of $\left(\begin{smallmatrix}1 \\ 0 \end{smallmatrix}\right)$ for the similar transitive left multiplication action of $\GL_2(\cO_F)$, these orbits correspond to the double cosets $Z(\cO_F)U(\cO_F)\backslash \GL_2(\cO_F) /K_1(n)$. In conclusion, $Z(F)U(F)\backslash \GL_2(F)/K_1(n)$ is indexed by invariants $\ell$, $\ov{x}$, and $a$ as above, and it remains to note that for the element $g_{t,\, \ell,\, v}$ these invariants are $\ell$, $v\i$, and $t + 2\ell$, respectively: indeed, the matrix $\ppp{\begin{smallmatrix}v^{-1}&\\\varpi_F^\ell & v \end{smallmatrix}}$ in $\GL_2(\cO_F)$ sends $\left(\begin{smallmatrix}1 \\ 0 \end{smallmatrix}\right)$ to the primitive vector $\left(\begin{smallmatrix}v^{-1}\\\varpi_F^\ell \end{smallmatrix}\right)$ (so its $\ov{x}$ and $\ell$ invariants are $v^{-1}$ and $\ell$, respectively) and can be written in the Bruhat decomposition as
\[\ppp{\begin{smallmatrix}v^{-1}&\\\varpi_F^\ell & v \end{smallmatrix}} = \ppp{\begin{smallmatrix}-\varpi_F^{-\ell}&-v^{-1}\\&-\varpi_F^{\ell} \end{smallmatrix}}
\ppp{\begin{smallmatrix}
& 1\\ -1&
\end{smallmatrix}} \ppp{\begin{smallmatrix}
1& v\varpi_F^{-\ell} \\&1
\end{smallmatrix}} =
\ppp{\begin{smallmatrix}-\varpi_F^{\ell}&\\&-\varpi_F^{\ell} \end{smallmatrix}}\ppp{\begin{smallmatrix}\varpi_F^{-t-2\ell}&\\&1 \end{smallmatrix}}\ppp{\begin{smallmatrix}1&v^{-1}\varpi_F^{t+\ell} \\& 1 \end{smallmatrix}}g_{t,\,\ell,\,v},
\]
which gives the sufficient
$g_{t,\, \ell,\, v} \in Z(F)U(F)\ppp{\begin{smallmatrix}\varpi_F^{t+2\ell}&\\&1 \end{smallmatrix}}\ppp{\begin{smallmatrix}v^{-1}&\\\varpi_F^\ell & v \end{smallmatrix}}$.}

\[
\GL_2(F) =  \bigsqcup_{0 \le \ell \le n}\  \ \bigsqcup_{v \in \cO_F^\times/(1 + \fm_F^{\min(\ell,\, n - \ell)})}\ \ \bigsqcup_{t \in \bZ}\ \  Z(F)U(F) g_{t,\,\ell,\,v} K_1(n).
\]
This decomposition reduces us to studying the values $W_{\pi,\,\psi}(g_{t,\,\ell,\,v})$,
and the following Atkin--Lehner relation that results from  \cite{Sah16}*{Proposition 2.28}\footnote{The proof of this relation 
does not use the blanket assumption of \cite[Section 2]{Sah16} that $\pi$ be unitarizable.} and \eqref{eqn:eps-product}  halves the range of the $\ell$ that one needs to consider: if $\omega_\pi = \mathbf{1}$ and $c(\psi) = 0$, then, for $0\le \ell \le a(\pi)$, there is a $p$-power root of unity $\zeta$ with
\be \label{l:atkinlehner}
	W_{\pi,\,\psi}(g_{t,\,\ell,\,v}) = \pm   \zeta \ W_{\pi,\,\psi}(g_{t +2\ell - a(\pi),\,a(\pi)-\ell,\,-v}). 
\ee
\epp



As we now illustrate, this relation is useful for deducing a description of the $p$-adic valuations of the elements $W_{\pi,\,\psi}(g_{t,\,\ell,\,v})$ with $\ell \in \{0, a(\pi)\}$. 

\begin{proposition}\label{p:otherprimes}
For a finite extension $F/\bQ_p$, an irreducible, admissible, infinite-dimen\-sional representation $\pi$ of $\GL_2(F)$ with $a(\pi)\ge 1$ and $\omega_\pi = \mathbf{1}$, 
an additive character $\psi \colon F \ra \bC^\times$ with $c(\psi) = 0$, a $t \in \Z$, an $\ell \in \{0, a(\pi)\}$, and a $v \in \cO_F^\times$, there is a $p$-power root of unity $\zeta$ such that
	\[
	W_{\pi,\,\psi}(g_{t,\,\ell,\,v}) = \begin{cases}
		\pm\zeta \qF^{-(1+t + \ell)} &\text{if }a(\pi)=1\text{ and }t + \ell \ge -1 ,\\
		\pm\zeta &\text{if }a(\pi)>1\text{ and }t + \ell= -a(\pi),\\
		0 &\text{otherwise.}
	\end{cases}
	\]
\end{proposition}



\emph{Proof.}
Since \eqref{l:atkinlehner} swaps $W_{\pi,\,\psi}(g_{t,\,0,\,v})$ and $W_{\pi,\,\psi}(g_{t-a(\pi),\,a(\pi),\,-v})$, 
 we may assume that  $\ell = a(\pi)$. Then, in terms of the description in \cref{foot-dec}, the matrices $g_{t,\, a(\pi),\, v}$ and $g\ce \abcd{\varpi_F^{t + 2 a(\pi) }}{}{}{1} \abcd {v\i}  \ \ {v\i}$ have the same invariants, so $W_{\pi,\,\psi}(g_{t,\, a(\pi),\, v})$ and $W_{\pi,\,\psi}(g)$ agree up to a factor that is a value of $\psi$, that is, up to a $p$-power root of unity.
It then remains to recall from \cite{CS18}*{Lemma 2.10} that 
\[
\qqqqqqqq W_{\pi,\,\psi}(g) = W_{\pi,\,\psi}(\ppp{\begin{smallmatrix}
\varpi_F^r&\\&1
\end{smallmatrix}})=\begin{cases}
\pm\qF^{-r} &\text{if }~a(\pi) = 1,~r\geq 0,\\
1  &\text{if }~a(\pi) > 1,~ r=0,\\
0 &\text{otherwise.} \qqqqqqqq\qqq\ \, \, \QED
\end{cases}
\]

\bpp[The Fourier expansion of $W_{\pi,\,\psi}(g_{t,\,\ell,\,v})$] \label{pp:basic-identity}
In \S\ref{elements-g}, for fixed $t \in \bZ$ and
$\ell \ge 0$, the function
\[
\cO_F^\times \ni v\mapsto W_{\pi,\,\psi}(g_{t,\,\ell,\,v}) 
\qxq{descends to the quotient}
 \cO_F^\times/(1 + \fm_F^\ell),
\]
so, by Fourier inversion, there are constants $c_{t,\,\ell}(\chi)\in\C$ for $\chi\in\Xc_{\le \ell}$ (see \S\ref{s:localgauss}) such that
\begin{equation}\label{eq_hazel}
\tst W_{\pi,\,\psi}(g_{t,\,\ell,\,v})=\sum_{\chi\in\Xc_{\le \ell}}c_{t,\,\ell}(\chi)\,\chi(v) \qxq{for every} v \in \cO_F^\times.
\end{equation}

To make use of this \emph{local Fourier expansion}, it is key to explicate the Fourier coefficients $c_{t,\,\ell}(\chi) \in \bC$. This may be done in terms of $\eps$-factors of representations of $\GL_2 \times \GL_1$ by using the  \emph{basic identity} of \cite{Sah16}*{Proposition 2.23 and before Remark 2.22}:\footnote{The cited claims do not use the blanket assumption of \cite{Sah16}*{Section 2} that $\pi$ 
be unitarizable.} if 
$c(\psi) = 0$ and 
$\omega_\pi = \mathbf{1}$, then, for $0 \le \ell \le a(\pi)$ and $\chi\in\Xc_{\le \ell}$, 
\[
\tst \frac{\varepsilon(\frac 12,\, \chi\pi)}{L(s,\, \chi\pi)}\, \displaystyle \sum_{t\in\Z}\, \qF^{(t+a(\chi\pi))(\frac 12-s)}\,c_{t,\,\ell}(\chi)= \tst \frac{1}{L(1-s,\, \chi^{-1}\pi)}\,\displaystyle\sum_{r\geq 0}\, \,\,\qF^{-r(\frac 12-s)}\Ga_\psi(\varpi_F^{r-\ell},\chi^{-1})W_{\pi,\,\psi}\ppp{\abcd
{\varpi_F^r}\ \ 1
}
\]
as Laurent polynomials in $\qF^{s}$ with the Gauss sums $\Ga_\psi$ as in \S\ref{s:localgauss}. 
This method for accessing the numbers $c_{t,\,\ell}(\chi)$ was carried out in \cite{Ass19}*{Section 2}, and we  will cite the resulting formulas below. For a discussion of related unpublished approaches of Templier and Hu, see \cite{CS18}*{Remark 2.20}.
\epp




\bpp[Classification of ramified $\pi$ with $\omega_\pi = \mathbf{1}$] \label{s:repclassification} \label{classification}
Our analysis of the Fourier coefficients $c_{t,\,\ell}(\chi)$ 
will rest on the following well-known classification of the irreducible, admissible, infinite-dimensional, representations $\pi$ of $\GL_2(F)$ that are ramified (that is, $a(\pi)\geq 1$) and whose  central character is  trivial (that is, $\omega_\pi=\mathbf{1}$). We refer to \cite[Sections 2--3]{JL70} and \cite[Section 1.2]{Sch02} (or \cite{BH06}*{Section~9.11}) for its justification,  and when possible we also give formulas for $a(\pi)$, $L(s, \chi \pi)$ and $\eps(s, \chi \pi)$ with $\chi \in \Xc$.



\begin{enumerate}
	\item
	\emph{$\pi$ is supercuspidal.} In this case, $a(\pi) \ge 2$ and $L(s, \chi \pi) =1$ (see \cite{Cas73}*{before Lemma on page 303 and middle of page 304} and \cite[Section 24.5]{BH06}). 
	\begin{enumerate}
		\item[(1a)] \emph{$\pi$ is dihedral supercuspidal.} Such a $\pi$ is associated, via the Weil representation, to
a character $\xi \colon E^\times \ra \bC^\times$ of a quadratic extension $E/F$ such that $\xi$ does not factor through $\mathrm{Norm}_{E/F}$, see \cite{JL70}*{Section 4} or \cite{Bum97}*{Theorem 4.8.6}. Equivalently, under the local Langlands correspondence \cite[Sections 33.4 and 34.4]{BH06} such a $\pi$ corresponds to $\Ind_{W_{E}}^{W_F} \xi$ where $\xi$ becomes a character of the Weil group $W_{E}$ via class field theory.
By \cite{JL70}*{Theorem 4.7~(ii)}, for such a $\pi$ we have $\omega_\pi= \xi|_{F^\times}\chi_{E/F}$, where $\chi_{E/F}$ is the quadratic character associated to $E/F$. 
In particular, $\omega_\pi = \mathbf{1}$ forces
		\be \label{eqn:eps-supercuspidal-norm}
			\qqqqqq \xi|_{\im(\mathrm{Norm}\colon E^\times \ra F^\times)} = \mathbf{1}, \qxq{while, by assumption,} \xi|_{\Ker(\mathrm{Norm}\colon E^\times \ra F^\times)} \neq \mathbf{1},
			\ee 
	so that $\xi$ is of finite order.	By \cite{JL70}*{Theorem 4.7~(i), (iii) and page 8} the representation $\chi \pi$ is also dihedral supercuspidal, 
	associated to 
	$\xi(\chi\circ\mathrm{Norm}_{E/F})\colon E^\times \ra \bC^\times$, and\footnote{\label{foot-gamma}By \cite[Lemma 1.2]{JL70} and \eqref{e:gausseps} with \eqref{e:epsunramtwist}--\eqref{e:epsduality}, we have $\gamma = \eps(\frac12, \chi_{E/F})$, and so also $\gamma^2 = \chi_{E/F}(-1)$.} 
	\be \label{eqn:eps-supercuspidal-twist}
		\qqqqqq \eps(s, \chi \pi) = \gamma \eps(s, \xi(\chi\circ\mathrm{Norm}_{E/F}), \psi \circ \mathrm{Trace}_{E/F}) \qxq{for some} \gamma \in \{\pm 1, \pm i\}.
\ee		
		With $d_{E/F}$ being the valuation of the discriminant of $E/F$, by \cite{Sch02}*{Theorem 2.3.2},
		\be \label{e:cond1a}
		\qqqq a(\pi) = [\bF_{E} : \bF_F] a(\xi) + d_{E/F}. \ee


		
		
		\item[(1b)]
		\emph{$\pi$ is nondihedral supercuspidal.}
		For such a $\pi$, we have $\Char(\bF_F) = 2$ and $a(\pi)>2$ (see \cite[Proposition 3.1.4]{Del73a} and \cite{Tun78}*{Proposition 3.5}), but there seems to be no simple expression for $\eps(s, \chi \pi)$. 
		For $F=\Q_2$, we describe such $\pi$ in \Cref{lem:Type1b} below. 
	\end{enumerate}

\item
\emph{$\pi \simeq\mu\St$ is the twist of the Steinberg representation by an unramified character $\mu$ with $\mu^2=\mathbf{1}$.}
In this case, $a(\pi)=1$, and, by \cite{Bum97}*{Section 4.7, equation (7.10)} and \cite{JL70}*{Proposition~3.6}, we have
\[
\qqq L(s,\chi\pi)=\begin{cases}\frac{1}{1- \mu(\varpi_F) \qF^{- 1/2-s }} &\text{if } \chi=\mathbf{1}, \\ 1 &\text{otherwise}, \end{cases} \qq
	\eps(s, \chi \pi) =\begin{cases}-\mu(\varpi_F)\qF^{1/2- s} &\text{if } \chi =\mathbf{1}, \\ \eps(s, \chi)^2 &\text{otherwise.} \end{cases}
	\]
	
\item
\emph{$\pi \simeq\mu\St$ is the twist of the Steinberg representation by a ramified character $\mu$ with $\mu^2=\mathbf{1}$.} In this case, by \cite{Bum97}*{Section 4.7, equation (7.10)} and \cite{JL70}*{Proposition 3.6}, 
we have $a(\pi)=2 a(\mu)\ge 2$ and
\[\ba
\qqq\  L(s,\chi\pi) =\begin{cases}\frac{1}{1- (\chi\mu)(\varpi_F) \qF^{- 1/2-s }} \!\!&\text{if } a(\chi\mu)=0, \\ 1 \!\!&\text{otherwise,} \end{cases} \ \
\eps(s, \chi \pi) =\begin{cases} -(\chi\mu)(\varpi_F)\qF^{ 1/2 - s } \!\!\!&\text{if } a(\chi\mu)=0, \msk \\ \eps(s, \chi\mu)^2,
\!\!\!&\text{otherwise.} \end{cases}
\ea \]

	\item
\emph{$\pi \simeq \mu\absF{\,\cdot\,}^{\sigma} \boxplus \mu\absF{\,\cdot\,}^{-\sigma}$ with $\sigma \neq \pm \frac 12$ is a principal series where the character $\mu \in \Xc$ is ramified with $\mu^2 =\mathbf{1}$.}
In this case, by \cite{JL70}*{Proposition 3.5}, 
we have $a(\pi)=2a(\mu)$ and
\[
\qqq L(s,\chi\pi)=\begin{cases}\frac{1}{(1- \qF^{-\sigma - s})(1- \qF^{\sigma - s})} &\text{if } \chi = \mu, \\ 1 &\text{otherwise,} \end{cases}
\qq \eps(s, \chi \pi) =\begin{cases}1 &\text{if } \chi = \mu, \\ \eps(s , \chi\mu)^2 &\text{otherwise.} \end{cases}
\]

\item
\emph{$\pi \simeq \mu\absF{\,\cdot\,}^{\sigma} \boxplus \mu^{-1}\absF{\,\cdot\,}^{-\sigma}$ is a principal series where the character $\mu \in \Xc$ is ramified with $\mu^2 \ne \mathbf{1}$.}
In this case, by the same reasoning as in the previous case, $a(\pi)=2a(\mu)$ and
\[
\qqq L(s,\chi\pi)=\begin{cases}\frac{1}{1- \qF^{\pm \sigma}} &\text{if } \chi = \mu^{\pm 1}, \\ 
1 &\text{otherwise,} \end{cases}
	\qq \eps(s, \chi \pi) = \qF^{\sigma(a(\chi \mu^{-1}) - a(\chi \mu))} \eps(s, \chi \mu) \eps (s, \chi \mu^{-1}).\]

\end{enumerate}
\epp

We refer to these cases as $\pi$ being of Type 1a, 1b, 2, 3, 4, or 5 (this numbering is not standard). Type 2 will not concern us much because our focus is the case $a(\pi) \ge 2$, and Types 1a, 3, 4, 5 are in some sense similar, for instance, $\eps(s, \pi)$ in these cases is expressed  in terms of $\eps$-factors of characters. Type 1b is the most subtle one, but it benefits from the more precise classification recorded in \Cref{lem:Type1b} that uses the following lemma, which further explicates conductor exponents.

\blem\label{lem:conductors}
For a supercuspidal  representation $\pi$ of $\GL_2(\Q_2)$ with $a(\pi)\ge 2$ and $\omega_\pi = \mathbf{1}$ \up{Type {\upshape 1}}, any twist-minimal twist $\pi_0$ of $\pi$ satisfies
\[a(\pi_0) \begin{cases}= a(\pi) &\text{if }~a(\pi) \text{ is  odd  or if }~ a(\pi)=2,\\
\le  a(\pi)-1 &\text{if }~a(\pi) \x{ is even and }~a(\pi) \ge 4,\\
\in \{a(\pi)-2, \, a(\pi)-1\} &\text{if }~a(\pi) \x{ is even and }~a(\pi) \ge 8.
\end{cases} \]
\elem

\begin{proof}
A twist of a supercuspidal representation is supercuspidal, and hence has conductor exponent $\ge 2$ (compare with \S\ref{s:repclassification}), so the first case follows from \eqref{cond-twist}. The second case may be deduced from \cite{AL78}*{Theorem 4.4 and the remark after it} by globalization, but we give a direct argument.

Suppose, for the sake of contradiction, that $a(\pi)$ is even with $\pi$ twist-minimal and $a(\pi) \ge 4$. By \cite[
Proposition 3.5]{Tun78}, such a $\pi$ is dihedral, associated to some $\xi: E^\times \rightarrow \C^\times$ with $E/\bQ_2$ unramified quadratic. By \eqref{e:cond1a}, we have $a(\xi)=\frac{a(\pi)}2>1$, so, by \cite[Section 18.1, Proposition]{BH06}, for any $\chi \in \Xc_{\Q_2,\, a(\xi)}$ also  $a(\chi \circ \mathrm{Norm}_{E/\Q_2}) = a(\xi)$. In particular, both $\xi$ and $\chi \circ \mathrm{Norm}_{E/\Q_2}$ are nontrivial on the group
\be \lab{disp-gp}
(1 + 2^{a(\xi)-1}\cO_{E}) / (1 + 2^{a(\xi)}\cO_{E}) \simeq (\bZ/2\Z)^2.
\ee
However, $\chi \circ \mathrm{Norm}_{E/\Q_2}$ is trivial on its subgroup $(1 + 2^{a(\xi)-1}\Z_2) / (1 + 2^{a(\xi)}\Z_2) \simeq \Z/2\Z$, and so is $\xi$: indeed, \eqref{eqn:eps-supercuspidal-norm} gives $\xi|_{\im(\mathrm{Norm}\colon E^\times\, \ra\, \Q_2^\times)} = \mathbf{1}$, whereas 
\[
\mathrm{Norm}_{E/\Q_2}\colon 1 + 2^{a(\xi) - 1}\cO_{E} \surjects 1 + 2^{a(\xi) - 1}\Z_2
\]
(see \cite[Chapter V, Section 2, Proposition 3 a)]{Ser79}). It follows that $\xi$ and $\chi \circ \mathrm{Norm}_{E/\Q_2}$ agree on the group \eqref{disp-gp}, so that $a(\xi (\chi^{-1} \circ \mathrm{Norm}_{E/\Q_2}))< a(\xi)$, and hence, by \eqref{e:cond1a}, also $a(\chi^{-1}\pi) < a(\pi)$, a contradiction.

Finally, suppose that $a(\pi)$ is even with $a(\pi)\ge 8$ and write $\pi \simeq \chi \pi_0$, so that \eqref{cond-twist} and the just-established inequality $a(\pi_0) \le a(\pi) - 1$ give $a(\pi) = 2a(\chi)$. Since $\omega_\pi = \mathbf{1}$, the central character of $\pi_0$ is $\chi^{-2}$, to the effect that $a(\pi_0) \ge 2a(\chi^2) $ by \cite{Tun78}*{Proposition~
3.4}. Since $a(\chi) \ge 4$ and we are dealing with $\bQ_2$, we have $a(\chi^2) = a(\chi)-1$, and the desired $a(\pi_0) \ge a(\pi)-2$ follows.
\end{proof}

\brem \label{odd-supercusp-cond}
In contrast, for an odd prime $p$ and a finite extension $F/\bQ_p$, every supercuspidal representation $\pi$ of $\GL_2(F)$ with $\omega_\pi = \mathbf{1}$ is twist-minimal, see, for instance, \cite{HNS19}*{Lemma~2.1}.
\erem

\bprop\label{lem:Type1b}
Up to isomorphism, there are $16$ nondihedral supercuspidal \up{that is, Type 1b}
representations $\pi$ of $\GL_2(\Q_2)$ with $\omega_\pi = \mathbf{1}$. 
Letting $\Xc^{\rm{quad}}_{\Q_2}$ be as in \uS\uref{s:type1b}, such $\pi$ are listed~as 
$$
\{\beta\pi_3: \beta \in \Xc^{\rm{quad}}_{\Q_2}\} \bigsqcup \{\beta\pi_7: \beta \in \Xc^{\rm{quad}}_{\Q_2}\}
$$
with the following conductor exponents\ucolon
\[
\ba
a(\pi_3) &= a(\beta_0 \pi_3)=3, \qq a(\beta_2 \pi_3) = a(\beta_0\beta_2 \pi_3) = 4, \\
a(\beta_3 \pi_3) &= a(\beta_2\beta_3 \pi_3) = a(\beta_0\beta_3 \pi_3) = a(\beta_0\beta_2\beta_3 \pi_3) = 6, \\
a(\pi_7) = a(\beta_0\pi_7)= a(\beta_2 \pi_7) &= a(\beta_0\beta_2 \pi_7) = a(\beta_3 \pi_7) = a(\beta_2\beta_3 \pi_7) = a(\beta_0\beta_3 \pi_7) = a(\beta_0\beta_2\beta_3 \pi_7) = 7.
\ea
\]
In contrast, no \emph{dihedral}	 supercuspidal representation $\pi'$ of $\GL_2(\bQ_2)$ with $\omega_{\pi'} = \mathbf{1}$ has $a(\pi') \in \{3, 7\}$.
\eprop


\begin{proof}
Via the local Langlands correspondence \cite[Section 33.4]{BH06}, 
our supercuspidal $\pi$ 
corresponds to an irreducible, smooth 
representation $\sigma\colon W_{\bQ_2} \ra \GL_2(\bC)$, which has its associated projectivization $\overline{\sigma}\colon W_{\bQ_2} \ra \PGL_2(\bC)$. Since $\omega_\pi = \mathbf{1}$, we have $\det(\sigma) = \mathbf{1}$, so $\sigma(W_{\bQ_2})$ is a subgroup of $\SL_2(\C)$ that is  necessarily finite (see \cite{BH06}*{Section 28.6, Proposition}). Since $\pi$ is nondihedral, $\sigma$ is not induced from a subgroup.  The projective image $\ov{\sigma}(W_{\bQ_2})$ must be the symmetric group $S_4$:  the only other finite, solvable subgroups of $\PGL_2(\bC)$ are cyclic, dihedral, and $A_4$, and the first two cannot occur because $\sigma$ is irreducible and not induced from a quadratic extension (compare with \cite{Wei74}*{Section 13}),
whereas Weil proved in \cite[Sections 34--35]{Wei74} that $\ov{\sigma}(W_{\bQ_2}) \not \simeq A_4$ (more precisely, $\sigma(W_{\bQ_2}) \not\simeq A_4$ because $A_4$ has no irreducible, $2$-dimensional representation, and $\sigma(W_{\bQ_2})$ is not a central extension of $A_4$ by $\bZ/2\bZ$ because the ``Condition C with respect to $A_4$'' of \cite{Wei74}*{Section~21} fails for $\bQ_2$; see also \cite{BR99}*{Section 8}).

Up to conjugation, there is a unique embedding of $S_4$ into $\PGL_2(\C)$ (compare with \cite{Wei74}*{Section~14}), so we fix one such and, in the notation of \emph{op.~cit.},~let $\Delta_0 \surjects S_4$ be the central extension by $\{\pm 1\}$ obtained by the preimage in $\SL_2(\C)$. Since $S_4$ has no faithful, irreducible, $2$-dimensional representations, by conjugating we may assume that $\sigma(W_{\bQ_2}) = \Delta_0$. In particular, the $S_4$-extension $K/\bQ_2$ cut out by $\ov{\sigma}$ extends to a $\Delta_0$-extension $\wt{K}/\bQ_2$. Thus, by \cite[Section 24 (with Sections~16 and 21)]{Wei74} (``Condition C with respect to $\Delta_0$'' is equivalent to ``Condition C with respect to $\Delta_0'$''), this extension also extends to a $\Delta_0'$-extension $\wt{K}'/\bQ_2$ with $\Delta_0' := \GL_2(\bF_3)$ inside $\GL_2(\bC)$ (note that $\GL_2(\bF_3)/\{\pm 1\} \simeq S_4$). By \cite{Wei74}*{Section 36} and \cite{BR99}*{Section 8}, this means that $K$ is one of the two $S_4$-extensions of $\bQ_2$ that extend to $\GL_2(\bF_3)$-extensions of $\bQ_2$. In particular, since any two lifts of $\ov{\sigma}$ to a $\wt{\sigma}\colon W_{\bQ_2} \ra \GL_2(\bC)$ are twists by a character (compare with 
\cite[Section 1]{Koc77}), we have isolated two distinct families of twists of $2$-dimensional, irreducible, smooth representations of $W_{\bQ_2}$ that could contain $\sigma$.

By \cite{Cal78}*{Theorem 5}, there exist representations $\pi_3$ and $\pi_7$ of $\GL_2(\Q_2)$, each either supercuspidal or a twist of Steinberg, such that $\omega_{\pi_3} = \omega_{\pi_7} = \mathbf{1}$ and $a(\pi_3)=3$, $a(\pi_7)=7$. To conclude it then suffices to argue that these $\pi_c$ are nondihedral supercuspidal: indeed, they will be twist-minimal by \Cref{lem:conductors}, 
the representation $\pi$ will be of the form $\beta\pi_c$ with $\beta\in \fX^{\rm{quad}}_{\bQ_2}$, all the latter will be pairwise distinct by \cite{BH06}*{Section 51.5}, and the formulas for the $a(\beta\pi_c)$ will follow from \eqref{cond-twist}.

The formulas for the conductor exponents in \S\ref{s:repclassification} show that $\pi_c$ is not a twist of Steinberg. Thus, we assume that $\pi_c$ is dihedral supercuspidal, associated to a quadratic extension
$E/\bQ_2$ and a character $\xi \colon E^\times \ra \bC^\times$ subject to \eqref{eqn:eps-supercuspidal-norm}. By \cite{Tun78}*{Proposition 3.5}, the extension $E/\bQ_2$ is \emph{ramified}, so that $a(\xi) = c - d_{E/\Q_2} \in \{c-2, c-3\}$ (see \S\ref{s:repclassification} and \S\ref{s:type1b}). For $c = 3$, this is already a contradiction: indeed, since $\bF_{E} \cong \bF_2$, the inequality $a(\xi) \le 1$ gives $a(\xi) = 0$, which contradicts \eqref{eqn:eps-supercuspidal-norm}. For $c = 7$, if $d_{E/\Q_2} = 2$, equivalently, if $a(\xi)=5$, then, by \eqref{eqn:eps-supercuspidal-norm} and \cite{Ser79}*{Chapter~IV, Section 1, Proposition 4 and Chapter V, Section 3, Corollary 3}, we have $\xi|_{1+4 \Z_2} = \mathbf{1}$, so the inclusion $1 + \fm_E^{4}\cO_{E} \subseteq (1 + 4 \Z_2)(1 + \fm_E^{5}\cO_{E})$ contradicts $a(\xi) = 5$. In the remaining case $c = 7$ with $d_{E/\bQ_2} = 3$, we have $a(\xi) = 4$, so again $\xi|_{1 + 4\bZ_2} = \mathbf{1}$, which, since $\xi|_{\Q_2^\times} = \chi_{E/\Q_2}$ (see \S\ref{s:repclassification}), contradicts the conductor-discriminant formula $a(\chi_{E/\Q_2}) = d_{E/\bQ_2} = 3$.
\end{proof}

\begin{remark}
As we learned from Ralf Schmidt, the main assertion of Proposition \ref{lem:Type1b} is due to Nekljudova \cite{Nek75} who obtained it by analyzing the Hecke algebra (see also \cite{Nob78}). With the local Langlands correspondence, it could also be deduced from results in \cite{Zin79} or \cite{Hen79}. 
\end{remark}

To prepare for a $p$-adic study of the values of $W_{\pi,\,\psi}$, we begin by exhibiting a general integrality away from $p$ property of these values in \Cref{prop:Type1b}. Its argument rests on the following~lemma.

\begin{lemma}\label{l:softintegrality}
For a finite extension $E/\Q_p$, 
an $m > 0$, a Haar measure  $dx$ on the \emph{additive} group $E^m$ with $\int_{\cO_{E}^m} d x \in \overline{\Z}[\frac 1p]$, and a function $f\colon (\cO_{E}^\times)^m \ra \overline{\Z}$ that is right multiplication invariant by $(1 + \varpi_{E}^n\cO_{E})^m$ for some $n > 0$ \up{that is, $f(x) = f(xy)$ for $y \in (1 + \varpi_{E}^n\cO_{E})^m$}, we~have
\be\label{soft-1}
\tst \int_{(\cO_{E}^\times)^m} f(x)d x \in \overline{\Z}[\frac 1p];
\ee
for a Haar measure $d^\times x$ on the \emph{multiplicative} group $(E^{\times})^m$ with $\int_{(\cO_{E}^\times)^m} d^\times x \in \overline{\Z}[\frac 1p]$, instead 
\be \label{soft-2}
\tst \frac{1}{\zeta_{E}(1)^m}\int_{(\cO_{E}^\times)^m} f(x)d^\times x \in \overline{\Z}[\frac 1p].
\ee
\end{lemma}

\begin{proof}
Due to \eqref{zeta-values}, the first display implies the second one. For the former,
\[
\int_{(\cO_{E}^\times)^m} f(x)d x =\!\!\ \sum_{x_0 \in (\cO_{E}^\times)^m/(1 + \varpi_{E}^n\cO_{E})^m}\!\!\!\!\!\! f(x_0)  \vol((1 + \varpi_{E}^n\cO_{E})^m) = \frac{\int_{\cO_{E}^m} d x }{q_E^{mn}}\!\! \sum_{x_0 \in  (\cO_{E}^\times)^m/(1 + \varpi_{E}^n\cO_{E})^m}\!\!\!\!\!\! f(x_0),
\]
and it remains to note that  $f$ takes values in $\ov{\bZ}$.  
 \end{proof}


\begin{proposition}\label{prop:Type1b}
For a finite extension $F/\bQ_p$, an irreducible, admissible, infinite-dimen\-sional representation $\pi$ of $\GL_2(F)$ such that $a(\pi)\ge 1$ and $\omega_\pi = \mathbf{1}$, an additive character $\psi \colon F \ra \bC^\times$ with $c(\psi) = 0$, and a $g \in \GL_2(F)$, 
$$
\tst\ \  W_{\pi,\, \psi}(g) \in \overline{\Z}[\frac 1p]\ \ \ \x{if $\pi$ is}\ \ \begin{cases} \x{dihedral supercuspidal \up{Type {\upshape1a}} or a twist of $\St$ \up{Types {\upshape2}, {\upshape3}}, or}\\
\x{principal series $\chi\absF{\,\cdot\,}^{\sigma} \boxplus \chi^{-1}\absF{\,\cdot\,}^{-\sigma}$ \up{Types {\upshape4}, {\upshape5}} with $q_F^{\pm\sigma} \in \overline{\Z}[\frac 1p]$}.
\end{cases}
$$
In addition, if $\pi$ is nondihedral supercuspidal \up{Type {\upshape1b}} and $F = \bQ_2$, then we  have
 \be \label{e:type1bwhittaker}
\tst	W_{\pi,\,\psi}(g) \in\begin{cases}
		\frac{1}{2^{1/2}}  \overline{\Z}  &\text{if }a(\pi)=6,~\ell=3,~t\in \{-3, -4\},\msk \\
		\overline{\Z} &\text{otherwise.}
\end{cases}
	\ee
\end{proposition}


\begin{proof}
By \S\ref{elements-g}, we may assume that $g = g_{t,\,\ell,\,v}$ for a $t \in \Z$, a $0 \le \ell \le a(\pi)$, and a $v \in \cO_F^\times$. For the first assertion, by \Cref{p:otherprimes}, we may assume that $\pi$ is not of Type 2, and, to conclude, claim that $W_{\pi,\, \psi}(g_{t,\,\ell,\,v})$ is a $\overline{\Z}[\frac 1p]$-linear combination of products of quantities $\int_{(\cO_{E}^\times)^m} f(x)d x$ with $f$ and $dx$ as in \Cref{l:softintegrality} for a finite extension $E/F$. This will follow from formulas for $W_{\pi,\, \psi}(g_{t,\,\ell,\,v})$ derived by Assing in \cite{Ass19}*{Section 3}. For referring to them below, we recall from \eqref{eqn:eps-product} that 
\[
\tst \eps(\frac 12, \tilde{\pi}) = \eps(\frac 12, \pi)= \pm 1
\]
and from \eqref{MGE-eq} that $\eps(\frac 12, \chi, \psi) \in \overline{\Z}[\frac 1p]^\times$ for a  character $\chi \colon F^\times \ra \bC^\times$ of finite order.

Namely, \cite{Ass19}*{Lemma 3.1} gives the desired description for $\pi$ of Type 1a (with $E/F$  quadratic and $m=1$; by \eqref{eqn:eps-supercuspidal-twist}, the quantity $\gamma$ there lies in $\{ \pm 1, \pm i\}$). To similarly treat $\pi \simeq \mu \St$ of Type 3, we first twist by a finite order unramified character and use \eqref{lem:twist by unramified finite order char} to assume that $\mu(\varpi_F)=1$, and then apply \cite{Ass19}*{Lemma~3.3}\footnote{Even though the case $\ell = a(\chi) = 1$ is omitted from the cited statement, it is treated in the proof: as is observed in the beginning of the argument there, the subcase $t \neq -2$ reduces to \cite{Ass19}*{Lemma~2.1}, whereas the subcase $t = -2$ is addressed before the phrase ``If $l = 1 = a(\chi)$, we will leave this expression as it is.''} 
(now $E=F$ and $m\in \{1, 2\}$; in the case of \emph{loc.~cit.}~that involves Sali\'{e} sums, we use \eqref{soft-2} instead of \eqref{soft-1}). Finally, for $\pi$ of Type 4 or 5, we combine the assumption $q_F^{\pm \sigma} \in \overline{\Z}[\frac 1p]$ with \cite{Ass19}*{Lemma~3.6} (now $E=F$ and $m\in \{1, 2\}$).

For the remaining \eqref{e:type1bwhittaker}, we assume that $\pi$ is of Type 1b with $F = \bQ_2$ and use \eqref{l:atkinlehner} with \Cref{p:otherprimes} to reduce to $0 < \ell \le \frac{a(\pi)}{2}$. By the classification in \Cref{lem:Type1b}, we have $a(\pi) \le 7$, so the bound is $1 \le \ell \le 3$.  We will use the local Fourier expansion
 $$
 \tst W_{\pi,\,\psi}(g_{t,\,\ell,\,v})\overset{\eqref{eq_hazel}}{=}\sum_{\chi\in\Xc_{\le \ell}}c_{t,\,\ell}(\chi)\,\chi(v)
 $$
and the following formulas for the $c_{t,\,\ell}(\chi)$ derived in \cite{Ass19}*{Section 2.1} from the basic identity of \S\ref{pp:basic-identity}:
\[
c_{t,\,\ell}(\chi) = \begin{cases}
-\eps(\frac 12,\pi)&\text{if }~\ell=1,~t=-a(\pi),~\chi=\mathbf{1},\\
2^{1-\ell/2}\eps(\frac 12,\chi)\eps(\frac 12,\chi^{-1}\pi)&\text{if }
~t=-a(\chi\pi),~\chi\in\Xc_{\ell},\\
0&\text{otherwise}.
\end{cases}
\]
Since $1\le \ell \le 3$, the appearing $\chi$ are quadratic (see \S\ref{s:type1b}), so $\eps(\frac 12,\pi)$, $\eps(\frac 12,\chi)$, and $\eps(\frac 12,\chi^{-1}\pi)$ are all roots of unity (see \eqref{e:epsduality} and \eqref{eqn:eps-product}). Thus, since $2^{1-\ell/2} \in \ov{\bZ}$ for $\ell \le 2$, we reduce to $\ell = 3$, when $a(\pi) \in \{6, 7\}$ and, in the notation of \S\ref{s:type1b}, the only appearing $\chi$ are $\beta_3$ and $\beta_2\beta_3$. If $a(\pi) = 6$, then for these $\chi$, by \Cref{lem:Type1b}, we have $a(\chi\pi) \in \{3,4\}$, and the claim follows. In the remaining case $a(\pi) =7$, we likewise have $a(\chi\pi) =7$, so we only need to consider the value
$$
\tst W_{\pi,\,\psi}(g_{-7,\,3,\,v}) = \frac 1{2^{1/2}} (\eps(\frac 12,\beta_3)\eps(\frac 12,\beta_3\pi)\beta_3(v) + \eps(\frac 12,\beta_2\beta_3)\eps(\frac 12,\beta_2\beta_3\pi)\beta_2\beta_3(v)).
$$
\Cref{l:gausssums2} gives $\eps(\frac 12,\beta_3) = \pm 1$ and $\eps(\frac 12,\beta_2\beta_3) = \pm i$, and \eqref{eqn:eps-product} gives $\eps(\frac 12,\beta_3\pi) = \pm 1$ and $\eps(\frac 12,\beta_2\beta_3\pi) = \pm 1$. 
Thus, $W_{\pi,\,\psi}(g_{-7,\,3,\,v})$ lies in $\{\pm\frac{1 + i}{2^{1/2}}, \pm\frac{1 - i}{2^{1/2}}\}$, and so is a root of unity in $\ov{\bZ}$.
\end{proof}

A final preparation for \Cref{T1,T2} is the following vanishing result that draws heavily on \cite{CS18}, which studied the phenomenon of exceptional vanishing of the values of $W_{\pi,\, \psi}$.



\bprop\label{l:vanishing of Whittaker newform}
For a finite extension $F/\bQ_p$, an additive character $\psi \colon F \ra \bC^\times$ with $c(\psi) = 0$, an irreducible, admissible, infinite-dimensional representation $\pi$ of $\GL_2(F)$ with $a(\pi)\ge 2$ and $\omega_\pi = \mathbf{1}$, a twist-minimal twist $\pi_0$ of $\pi$, a $0\le \ell \le a(\pi)$, and a $v \in \cO_F^\times$, we have
\[
W_{\pi,\,\psi}(g_{t,\,\ell,\,v})=0 \qxq{if} \begin{cases}
t<- \max(a(\pi), 2\ell), \x{ or} \\
t > - \max(a(\pi), 2\ell),~\ell \neq \frac{a(\pi)}{2}, \x{ or} \\
t > -a(\pi_0),~\x{$\pi$ is supercuspidal \up{Type {\upshape 1}}}, \x{ or} \\
t \neq -\max(a(\pi), 2\ell),~p\x{ is odd,}~\x{$\pi$ is supercuspidal \up{Type {\upshape 1}}}.
 \end{cases}
\]
Moreover, in the case $F = \bQ_2$ we have the following additional vanishing for $\ell = \frac{a(\pi)}{2}$\ucolon  
\[
 W_{\pi,\,\psi}(g_{t,\,\frac{a(\pi)}{2},\,v})=0 \qxq{if} \begin{cases}t \le -a(\pi),~\x{$\pi$ is supercuspidal \up{Type {\upshape 1}} with $a(\pi_0) \le a(\pi) - 1$, or} \\
t \le -a(\pi) + 1,~\x{$\pi$ is supercuspidal \up{Type {\upshape 1}} with $a(\pi_0) \le a(\pi) - 2$, or} \\
t \le -a(\pi) + 1,~\x{$\pi$ is a ramified twist of $\St$ \up{Type {\upshape 3}},  or} \\
t \le -a(\pi) + 1,~\x{$\pi \simeq \mu\abs{\,\cdot\,}^{\sigma}_F \boxplus \mu\abs{\,\cdot\,}_F^{-\sigma}$ with $\sigma \neq \pm \frac 12$, $\mu^2 = \mathbf{1}$ \up{Type {\upshape 4}},  or} \\
t \le -a(\pi) + 2,~\x{$\pi \simeq \mu\abs{\,\cdot\,}^{\sigma}_F \boxplus \mu\i\abs{\,\cdot\,}_F^{-\sigma}$ with $\mu^2 \neq \mathbf{1}$ \up{Type {\upshape 5}}}.
\end{cases}
\]
\eprop
\begin{proof}
The additional vanishing statements for $\ell = \frac{a(\pi)}{2}$ follow from the rest and from \cite[Theorem~2.14]{CS18} (with \S\ref{s:type1b} and \eqref{cond-twist}; for instance, for Type 5, one uses that $a(\mu) \ge 4$, so that also $a(\mu^2) = a(\mu) - 1$). 

For the main statement, its last case follows from the rest: indeed, by \Cref{odd-supercusp-cond}, if $p$ is odd and $\pi$ is supercuspidal, then $a(\pi_0) = a(\pi)$. Moreover, its case $t<- \max(a(\pi), 2\ell)$ follows from \cite{Sah17}*{Proposition 2.10~(1)},\footnote{The proof does not use the assumption of \cite{Sah17}*{Section 2.2} that $\pi$ be unitarizable, 
compare with \cite{CS18}*{Proposition 2.11}.} so we assume that $t\ge - \max(a(\pi), 2\ell)$. In the remaining cases, we use the Atkin--Lehner relation \eqref{l:atkinlehner}, which replaces $t$ by $t + 2\ell - a(\pi)$ and $\ell$ by $a(\pi) - \ell$, to reduce to $0 \le \ell \le \frac{a(\pi)}{2}$, and we will conclude from \eqref{eq_hazel} by arguing that $c_{t,\,\ell}(\chi)=0$ for all $\chi\in\Xc_{\le \ell}$.

For this, we will use the basic identity reviewed in \S\ref{pp:basic-identity}. By inspecting \S \ref{s:repclassification}, in the remaining cases in question we find that $L(s,\chi\pi)=1$, and, by \cite[Lemma 2.10]{CS18},
\[
W_{\pi,\,\psi}\left(\sabcd{\varpi_F^r}{}{}{1}\right) = \begin{cases}1, &\x{if $r = 0$,} \\ 0, &\x{if $r > 0$.} \end{cases}
\]
In effect, the basic identity in the cases in question is the equality
\[ \tst
		\eps(\frac 12,\chi\pi)\sum_{t\in\Z}\qF^{(t+a(\chi\pi))(\frac 12-s)}c_{t,\,\ell}(\chi)
		=
			\Ga_\psi(\varpi_F^{-\ell},\chi^{-1})
\]
of Laurent polynomials in $q_F^s$. In the case when $\ell < \frac{a(\pi)}2$, by \eqref{cond-twist}, we have $a(\chi\pi) = a(\pi)$, so the $c_{t,\,\ell}(\chi)$ indeed vanish for $t \neq -a(\pi)$. In the remaining case when $\pi$ is supercuspidal, we have $a(\chi\pi)\ge a(\pi_0)$, and the $c_{t,\,\ell}(\chi)$ still vanish for $t > -a(\pi_0) \ge -a(\chi\pi)$, as desired.
\end{proof}

In the remaining case $a(\pi) \ge 2$, for clarity, we split the sought bounds on $\val_p(W_{\pi,\,\psi}(g_{t,\,\ell,\,v}))$ into the case of an odd $p$ (\Cref{T1}) and that of $F = \Q_2$ (\Cref{T2}). To avoid additional technical complications, we do not attempt to treat  the case of a general finite extension of $\bQ_2$.

\bthm \label{T1}
For a finite extension $F/\bQ_p$ with $p$ \emph{odd}, an irreducible, admissible, infinite-dimen\-sional representation $\pi$ of $\GL_2(F)$ with $a(\pi)\ge 2$ and $\omega_\pi = \mathbf{1}$, an additive character $\psi \colon\! F \ra \bC^\times$ with $c(\psi) = 0$, an isomorphism $\bC \simeq \ov{\bQ}_p$, 
a $t \in \Z$, a $0 \le \ell \le a(\pi)$, and a $v\in\cO_F^\times$, we have
\[
\tst \val_p(W_{\pi,\,\psi}(g_{t,\,\ell,\,v})) \geq \begin{cases}
0  &\text{if }~\ell \in \{0, a(\pi)\},\\
0  &\text{if }~\ell \in \{1, a(\pi) - 1\},~a(\pi)>2,\\
\tst [\bF_F : \bF_p]\left(1 - \frac{\min(\ell,\, a(\pi)-\ell)}{2}\right)  &\text{if }~\ell \notin \{0,1, \frac{a(\pi)}{2}, a(\pi)-1, a(\pi)\},\\
\tst  -[\bF_F : \bF_p] + \min\left(\frac{[\bF_F : \bF_p]}2, \frac12 +  \frac{1}{p-1}\right)  & \text{if }~\ell=1,~a(\pi)=2,  ~t=-2, \\
\tst [\bF_F : \bF_p]\left(1 - \frac{a(\pi)}{4}\right)  &\text{if }~\ell = \frac{a(\pi)}{2},~a(\pi)>2, ~t= - a(\pi),
\end{cases}
\]
and, for $\ell = \frac{a(\pi)}{2}$ and an even $a(\pi)$, the following additional bounds \up{see also Proposition \uref{l:vanishing of Whittaker newform}}\ucolon
\benumr
\m \label{1-i}
if $\pi$ is  supercuspidal \up{Type {\upshape 1}} with $a(\pi)=2$, then
\[\tst \qq \val_p(W_{\pi,\,\psi}(g_{t,\,1,\,v})) \ge
-[\FF:\Fp] + \frac12 + \frac{1}{p-1}\!;
\]

\m \label{1-ii}
if $\pi$ is a twist of Steinberg by a ramified quadratic character \up{Type {\upshape 3}}, then $a(\pi)=2$ and
\[
\tst \qq \val_p(W_{\pi,\,\psi}(g_{t,\,1,\,v})) \ge -\frac{t+4}{2}[\bF_F : \bF_p] + \min\left(-[\bF_F : \bF_p]\left(\frac{t +1}2\right), \, \frac12 + \frac{1}{p-1}\right)\!;
\]

\m \label{1-iii}
if $\pi$ is a principal series $\mu\abs{\,\cdot\,}_F^{\sigma} \boxplus \mu\abs{\,\cdot\,}_F^{-\sigma}$ with $\mu^2 = \mathbf{1}$ \up{Type {\upshape 4}},  then $a(\pi)=2$ and
\[
\tst \qqq \val_p(W_{\pi,\,\psi}(g_{t,\,1,\,v}))
\ge-[\bF_F : \bF_p]  - (t+2)|\val_p(\qF^{\sigma})| + \min\left(-[\bF_F : \bF_p]\left(\frac{t+1}2\right) , \, \frac12 + \frac{1}{p-1}\right)\!;
\]


\m\label{1-iv}
if $\pi$ is a principal series $\mu\abs{\,\cdot\,}^{\sigma}_F \boxplus \mu^{-1}\abs{\,\cdot\,}^{-\sigma}_F$ with $\mu^2 \neq \mathbf{1}$ \up{Type {\upshape 5}}, then
\[
\tst \qq\ \  \val_p(W_{\pi,\,\psi}(g_{t,\,\frac{a(\pi)}{2},\,v}))\ge
\begin{cases}
-\frac{[\bF_F : \bF_p](t+4)}{2} +\frac12 + \frac{1}{p-1}-(t+2)|\val_p(\qF^{\sigma})| &\x{if }~a(\pi)=2,\\
-\frac{[\bF_F : \bF_p]\max(t + a(\pi),\, a(\pi)/2-2)}{2}  -(t+a(\pi))|\val_p(\qF^{\sigma})| &\x{if }~a(\pi)>2.
\end{cases}
\]

\eenum
\ethm

\bthm \label{T2}
For an irreducible, admissible, infinite-dimensional, representation $\pi$ of $\GL_2(\Q_2)$ with $a(\pi)\ge 2$ and $\omega_\pi = \mathbf{1}$, an additive character $\psi \colon \Q_2 \ra \bC^\times$ with $c(\psi) = 0$, an isomorphism $\bC \simeq \ov{\bQ}_2$, 
a $t \in \Z$, a $0 \le \ell \le a(\pi)$,  and a $v\in \Z_2^\times$, we have
\be \label{e:p2mainlocal}
\tst \val_2(W_{\pi,\,\psi}(g_{t,\,\ell,\,v})) \geq \begin{cases}0  &\text{if }~\ell \in \{0,1,a(\pi)-1, a(\pi)\},\\
\tst 1-\frac{\min(\ell,\, a(\pi)-\ell)}{2}  &\text{if }~\ell \notin \{0,1, \frac{a(\pi)}{2}, a(\pi)-1, a(\pi)\},\\
0  &\text{if }~\ell \in \{3, a(\pi)-3\},~a(\pi)>6,
\end{cases}
\ee
and, for $\ell = \frac{a(\pi)}{2}$ and an even $a(\pi) > 2$, the following additional bounds \up{see also Proposition~\uref{l:vanishing of Whittaker newform}}\ucolon
\benumr
\m \label{case-i}
if $\pi$ is supercuspidal \up{Type {\upshape 1}}, then 
\[
\qqq \tst \val_2(W_{\pi,\,\psi}(g_{t,\,\frac{a(\pi)}{2},\,v})) \ge 1 -\frac{a(\pi)}{4} \ \ \x{and, for $a(\pi)\in\{6,8\}$,}\ \ \val_2(W_{\pi,\,\psi}(g_{-a(\pi)+1,\,\frac{a(\pi)}{2},\,v})) \ge 0;
\]


\m \label{case-iii}
if $\pi$ is a twist of Steinberg by a ramified quadratic character \up{Type {\upshape 3}}, then $a(\pi) \in \{4,6\}$,
\[
\tst \qq \val_2(W_{\pi,\,\psi}(g_{t,\,\frac{a(\pi)}{2},\,v})) = \begin{cases}
-(t+3) &\text{if }~t\ge -2,~a(\pi)=4, \\
 -(t+\frac72) &\text{if }~t\ge -2,~a(\pi)=6,\\
 -\frac12 &\text{if }~t= -4,~a(\pi)=6, \\
  \infty &\text{otherwise}; \\
\end{cases}
\]

\m \label{case-iv}
if $\pi$ is a principal series $\mu\abs{\,\cdot\,}^{\sigma}_{\bQ_2} \boxplus \mu\abs{\,\cdot\,}^{-\sigma}_{\bQ_2}$ with $\mu^2 = \mathbf{1}$ \up{Type {\upshape 4}}, then $a(\pi) \in \{4,6\}$,
\[
\tst \qqq \val_2(W_{\pi,\,\psi}(g_{t,\,\frac{a(\pi)}{2},\,v}))
\begin{cases}
\ge -\frac{t+4}2 - (t+2)|\val_2(2^\sigma)| &\text{if }~t\ge -2,~a(\pi)=4, \\
\ge -\frac{t+5}{2} - (t+2)|\val_2(2^\sigma)| &\text{if }~t\ge -2,~a(\pi)=6,\\
= -\frac 12 &\text{if }~t= -4,~a(\pi)=6,\\
=  \infty &\text{otherwise}; \\
\end{cases}
\]

\m \label{case-v}
if $\pi$ is a principal series $\mu\abs{\,\cdot\,}^{\sigma}_{\bQ_2} \boxplus \mu^{-1}\abs{\,\cdot\,}_{\bQ_2}^{-\sigma}$ with $\mu^2 \neq \mathbf{1}$ \up{Type {\upshape 5}}, then $a(\pi) \ge 8$, 
\[
\tst \qqq \val_2(W_{\pi,\,\psi}(g_{t,\,\frac{a(\pi)}{2},\,v}))\ge
\begin{cases}
\frac{1-t-a(\pi)}{2}-(t+a(\pi)-2)|\val_2(2^{\sigma})| &\text{if }~t\ge -\frac{a(\pi)}{2},\\
\frac{4 - a(\pi)}4-(t+a(\pi)-2)|\val_2(2^{\sigma})| &\text{if }~-a(\pi) + 2 < t< -\frac{a(\pi)}{2}, \\
\infty & \x{if }~t \le -a(\pi) + 2.
\end{cases}
\]
\eenum
\ethm



\bpp[]\emph{Proof of Theorems \uref{T1} and \uref{T2}.}
Even though we have separated the cases of an odd $p$ and of $p = 2$ with $F = \bQ_2$ into separate statements, 
we will prove them simultaneously.
\epp
For $\ell \in \{0, a(\pi)\}$, the assertion is that $\val_p(W_{\pi,\,\psi}(g_{t,\,\ell,\,v})) \geq 0$, which follows from \Cref{p:otherprimes}. Each of the assertions that involves $\ell > \frac{a(\pi)}{2}$ allows any $t \in \bZ$. Thus, we may use the Atkin--Lehner relation \eqref{l:atkinlehner} to switch $\ell$ and $a(\pi)-\ell$ if needed to assume from now on that 
\benuma
\item \label{ass-1}
$1 \le \ell \le \frac{a(\pi)}{2}$ and, by also using \Cref{l:vanishing of Whittaker newform}, that if $\ell < \frac{a(\pi)}{2}$, then $t = -a(\pi)$. 
\eenum
Moreover, $\pi$ is not of Type 2 because $a(\pi) \ge 2$ (see \S\ref{s:repclassification}). If $\pi$ is of Type 1b (so that $p = 2$), then the sought bounds follow from \Cref{lem:Type1b} and \eqref{e:type1bwhittaker}.
Thus, we assume from now on that
\begin{enumerate} \addtocounter{enumi}{1}
\item \label{ass-2}
$\pi$ is not of Type 1b or Type 2.
\end{enumerate}


Our basic strategy is as follows: by the local Fourier expansion \eqref{eq_hazel}, we have
\be\label{eq_hazel-recall}
\tst W_{\pi,\,\psi}(g_{t,\, \ell,\,v}) =\sum_{\chi\in\Xc_{\le \ell}}c_{t,\,\ell}(\chi)\,\chi(v), \qxq{so} \val_p( W_{\pi,\,\psi}(g_{t,\, \ell,\,v})) \ge \min_{\chi\in\Xc_{\le \ell}}(\val_p(c_{t,\,\ell}(\chi))),
\ee
and we will bound $\val_p(c_{t,\,\ell}(\chi))$ individually for each representation in the classification of \S\ref{s:repclassification} (in exceptional cases individual bounds will not suffice and we will consider the full sum). Below we omit our fixed $\psi$ from the notation when forming $\eps$-factors with respect to it.

\textbf{The case when $\pi$ is of Type 1a.} Such a $\pi$ is associated to a character $\xi \colon E^\times \ra \bC^\times$ for a quadratic extension $E/F$. By \cite[Section 2.1]{Ass19}, for $1 \le \ell \le \frac{a(\pi)}{2}$ and $\chi \in \fX_{\le \ell}$,
\begin{equation}\label{e:supercuspidal ctl(mu)}
c_{t,\,\ell}(\chi) = \begin{cases}
-\frac{1}{\qF-1}\eps(\frac 12,\pi)&\text{if }~t=-a(\pi),~\ell=1,~\chi=\mathbf{1},\\
\frac{1}{\qF-1}\qF^{1-\ell/2}\eps(\frac 12,\chi)\eps(\frac 12,\chi^{-1}\pi)&\text{if }~
t=-a(\chi\pi),~\chi\in\Xc_{\ell},\\
0&\text{otherwise.}
\end{cases}
\end{equation}
In particular, $c_{t,\,\ell}(\mathbf{1})=0$ unless $t = -a(\pi)$ and $\ell = 1$, in which case $\val_p(c_{-a(\pi),\,1}(\mathbf{1}))=0$ (see \eqref{eqn:eps-product}), 
and $c_{t,\,\ell}(\chi)=0$ for $\chi\in\Xc_{\le \ell}\setminus\{\mathbf{1}\}$ unless $\chi \in \Xc_\ell$. Since all the required bounds are nonpositive for Type 1a when $\ell=1$, this reduces us to $\chi \in \Xc_\ell$ with $t = -a(\chi\pi)$.

We begin with the case $a(\pi) = 2$, when $\ell=1$ and, since $\chi \in \Xc_1$, also $F \neq \Q_2$ (so that $p$ is odd) and $t = -a(\chi \pi)=-2$ (see \eqref{cond-twist}).
 By \S\ref{s:repclassification}, the representation $\chi\i\pi$ is dihedral supercuspidal associated to $\xi(\chi^{-1}\circ\mathrm{Norm}_{E/F}) \colon E^\times \ra \bC^\times$. By \cite[Proposition~
 3.5]{Tun78}, we may assume that $E/F$ is unramified, so that $a(\xi(\chi^{-1}\circ\mathrm{Norm}_{E/F})) = 1$ by \eqref{e:cond1a}.
Thus, by \eqref{eqn:eps-supercuspidal-twist} and \Cref{p:maingausseps}~\ref{MGE-i},
\[
\tst \val_p(\eps(\frac12, \chi^{-1} \pi)) = \val_p(\eps(\frac12, \xi(\chi^{-1}\circ\mathrm{Norm}_{E/F}), \psi \circ \mathrm{Trace}_{E/F})) = 
-[\FF:\Fp] + \frac{s(\xi\i(\chi\circ\mathrm{Norm}_{E/F}))}{p-1}.
\]
Consequently, \eqref{e:supercuspidal ctl(mu)} and \Cref{p:maingausseps}~\ref{MGE-i} give
\[
\tst \val_p(c_{-2,\,1}(\chi)) = -[\FF:\Fp] + \frac{s(\chi\i) + s(\xi\i(\chi\circ\mathrm{Norm}_{E/F}))}{p-1}.
\]
By \eqref{eqn:eps-supercuspidal-norm}, we have $\xi|_{\cO_F^\times} = \mathbf{1}$, so \eqref{e:schiprop1} and \eqref{e:schiprop2} give $p-1 \mid 2s(\chi\i) + s(\xi\i(\chi\circ\mathrm{Norm}_{E/F}))$. 
Since $s(\chi\i)$ and $s(\xi\i(\chi\circ\mathrm{Norm}_{E/F}))$ are positive, it follows that $s(\chi\i) + s(\xi\i(\chi\circ\mathrm{Norm}_{E/F})) \ge \frac{p-1}2 + 1$.
In conclusion, for $a(\pi) = 2$, 
we obtain the sufficient bound
\[
\tst \val_p(c_{-2,\,1}(\chi)) \ge -[\FF:\Fp] + \frac12 + \frac{1}{p-1}.
\]
We next turn to the case when $a(\pi)>2$ with $1\le \ell<\frac{a(\pi)}{2}$, and $\chi \in \fX_\ell$ with $t = -a(\chi \pi)$ as above. By \eqref{cond-twist}, we have $a(\pi)=a(\chi^{\pm 1} \pi)$, so that, by \eqref{e:cond1a}, also $a(\xi(\chi^{\pm 1}\circ\mathrm{Norm}_{E/F}))=a(\xi)$. In addition, $a(\xi) > 1$: indeed, otherwise, by \eqref{eqn:eps-supercuspidal-norm}, we would have $a(\xi) = 1$ and, since, by \eqref{e:cond1a},
\[
[\bF_{E} : \bF_F]a(\xi) + d_{E/F} = a(\pi) > 2,
\]
the quadratic extension $E/F$ would be ramified (so that $\bF_{E} \cong \bF_F$), we would have $p = 2$ because $d_{E/F} \le 1$ for odd $p$ (see \cite{Ser79}*{Chapter III, Section 6, Proposition 13}), and the simultaneous $\bF_{E} \cong \bF_2$ and $a(\xi) = 1$ would give a contradiction.
 Thus, \eqref{e:supercuspidal ctl(mu)} together with \Cref{p:maingausseps} and \eqref{eqn:eps-supercuspidal-twist} gives
\[
\val_p(c_{-a(\pi),\,\ell}(\chi)) \begin{cases}
	\ge \frac{1}{p-1}&\text{if }~\ell=1,\\
	=[\FF:\Fp](1 - \frac{\ell}{2})&\text{if }~\ell > 1.\\
	\end{cases}
\]
These bounds suffice in all cases with $a(\pi)>2$ and $\ell<\frac{a(\pi)}{2}$ except when $p = 2$ with $a(\pi) > 6$ and $\ell = 3$, when instead we seek to show that $\val_2(W_{\pi,\, \psi}(g_{-a(\pi),\,\ell,\, v})) \ge 0$ and  bounding each $\val_p(c_{t,\,\ell}(\chi))$ does not suffice. Instead, in the notation of \S\ref{s:type1b}, in this case \eqref{eq_hazel-recall} and \eqref{e:supercuspidal ctl(mu)} give 
\be \label{e:type1basicidspeciall3}
\tst W_{\pi,\,\psi}(g_{-a(\pi),\,3,\,v}) =\frac 1{2^{1/2}}\left( \eps(\frac12, \beta_3) \eps(\frac12, \beta_3\pi) \beta_3(v) + \eps(\frac12, \beta_2\beta_3) \eps(\frac12, \beta_2 \beta_3\pi) (\beta_2\beta_3)(v)\right).
\ee
Since $\beta_2^2 = \beta_3^2 = \mathbf{1}$, \Cref{l:gausssums2} and \eqref{eqn:eps-product} then give the sufficient $W_{\pi,\,\psi}(g_{-a(\pi),\,3,\,v}) \in \{\pm\frac { 1 + i}{2^{1/2}}, \pm\frac { 1 - i}{2^{1/2}}\}$. 

We turn to the remaining case when $a(\pi)>2$ with $\ell=\frac{a(\pi)}{2}$, and $\chi\in\Xc_{\ell}$ with $t = -a(\chi\pi)$ as above. If $a(\chi\i \pi) > 2$ (for instance, if $p$ is odd, see \Cref{odd-supercusp-cond}), then, as above, \eqref{e:cond1a} gives $a(\xi(\chi\i \circ \mathrm{Norm}_{E/F})) > 1$, to the effect that, by \eqref{eqn:eps-supercuspidal-twist}, \eqref{e:supercuspidal ctl(mu)}, and \Cref{p:maingausseps}~\ref{MGE-ii},
\be \label{1a-cases}
\tst \val_p(c_{t,\,\frac{a(\pi)}{2}}(\chi))
=[\FF:\Fp](1 - \frac{a(\pi)}{4}).
\ee
If, in contrast, $a(\chi\i \pi) = 2$, then $p = 2$, \Cref{lem:conductors} and \S\ref{s:type1b} give $\chi^2 = \mathbf{1}$ and so also $\omega_{\chi\i \pi} = \mathbf{1}$, and \eqref{1a-cases} follows from \eqref{eqn:eps-product}, \eqref{e:supercuspidal ctl(mu)}, and \Cref{p:maingausseps}~\ref{MGE-ii}. 

The equality \eqref{1a-cases} suffices for the desired bounds unless $p = 2$ and $a(\pi) \in \{6,8\}$, when we seek to show the additional bound $\val_2(W_{\pi,\, \psi}(g_{-a(\pi)+1,\,\frac{a(\pi)}{2},\, v})) \ge 0$.
In this final case, by \Cref{lem:conductors} and \eqref{cond-twist}, the minimal conductor twist $\pi_0$ of $\pi \simeq \chi_0 \pi_0$ satisfies  $a(\pi_0)\le a(\pi)-1$ and $a(\chi_0) = \frac{a(\pi)}{2}$. Moreover, we may assume that $a(\pi_0) = a(\pi) - 1$ because otherwise $W_{\pi,\, \psi}(g_{-a(\pi)+1,\,\frac{a(\pi)}2,\, v})=0$ by Proposition \ref{l:vanishing of Whittaker newform}. Then $E/\Q_2$ is ramified by \cite[Proposition~3.5]{Tun78} and, for any $\chi \in \fX_{\frac{a(\pi)}{2}}$,  
we have $a(\chi\chi_0)\le a(\chi_0) -1=\frac{a(\pi)}{2} - 1<\frac{a(\pi_0)}{2}$, so also $a(\chi\pi) =  a((\chi\chi_0)\pi_0)=a(\pi_0) = a(\pi)-1$ (see \eqref{cond-twist}). Consequently, by \eqref{eq_hazel-recall} and \eqref{e:supercuspidal ctl(mu)},
\be\label{e:type1basicidspecial}
\tst W_{\pi,\,\psi}(g_{-a(\pi)+1,\,\frac{a(\pi)}{2},\,v}) = 2^{1-\frac{a(\pi)}{4}}\sum_{\chi \in \Xc_{\Q_2,\,\frac{a(\pi)}2}} \eps(\frac12, \chi) \eps(\frac12, \chi^{-1}\pi) \chi(v).
\ee
If $a(\pi)=6$, then, as after \eqref{e:type1basicidspeciall3}, \Cref{l:gausssums2} gives the sufficient $W_{\pi,\,\psi}(g_{-5,\,3,\,v}) \in \{\pm\frac { 1 + i}{2^{1/2}}, \pm\frac { 1 - i}{2^{1/2}}\}$.
If $a(\pi)=8$, then, 
letting $\beta_4 \in \Xc_{\Q_2,\, 4}$ be nonquadratic  with $\beta_4(-1)=-1$, we have
\[
\Xc_{\Q_2,\,4} = \{\beta_4, \beta_2\beta_4 , \beta_4^{-1}, \beta_2\beta_4^{-1}\}, \qxq{with} \beta_2 \in \Xc_{\Q_2,\, 2}, \q \beta_2(-1) = -1 \qx{as  in \S\ref{s:type1b}.}
\]
In this notation, \eqref{eqn:eps-product} gives $\eps(\frac12, \beta_4^{-1}\pi)\eps(\frac12, \beta_4\pi) = 1$ and $\eps(\frac12, \beta_2\beta_4^{-1}\pi)\eps(\frac12, \beta_2\beta_4\pi) = 1 $, so, with
\[
\tst x \ce \eps(\frac12, \beta_4) \eps(\frac12, \beta_4^{-1}\pi) \beta_4(v) \qxq{and} x'\ce \eps(\frac12, \beta_2\beta_4) \eps(\frac12, \beta_2\beta_4^{-1}\pi) (\beta_2\beta_4)(v),
\]
by \eqref{e:epsduality} and \eqref{e:type1basicidspecial}, we have
\be \label{e:specialcasesuper1}
\tst W_{\pi,\,\psi}(g_{-7,\,4,\,v}) = \frac 12\left(x - x^{-1} + x' + x'^{-1} \right).
\ee
The characters $\beta_4\i$ and $\beta_2\beta_4\i$ agree on $1 + 4 \Z_2$, so they satisfy \Cref{l:vunit}~\ref{item: mult -> additive, even case} with the \emph{same} $u \in \bZ_2^\times$. Thus, \Cref{l:evaluate gauss sums with a(mu)>1}~\ref{EG-a} gives $\Ga_\psi(\frac 1{16}, \beta_4\i) = \pm \Ga_\psi(\frac 1{16}, \beta_2\beta_4\i)$, so that, by \eqref{e:gausseps}, also
\be \label{just-sign}
\tst \eps(\frac12, \beta_4) = \pm  \eps(\frac12, \beta_2\beta_4), \qx{where, by \Cref{p:maingausseps}~\ref{MGE-ii}, both sides are roots of unity.}
\ee
By \S\ref{s:repclassification}, the representations $\beta_4^{-1} \pi$ and $\beta_2\beta_4^{-1} \pi$ of conductor exponent $7$ (see before \eqref{e:type1basicidspecial}) are dihedral supercuspidal associated to $\xi(\beta_4\i \circ \mathrm{Norm}_{E/\bQ_2})$ and $\xi(\beta_2\beta_4\i \circ \mathrm{Norm}_{E/\bQ_2})$, respectively. Thus, since $E/\bQ_2$ is ramified quadratic, and hence $d_{E/\bQ_2} \in \{ 2, 3\}$, we obtain from \eqref{e:cond1a} that
\[
a(\xi(\beta_4\i \circ \mathrm{Norm}_{E/\bQ_2})) = a(\xi(\beta_2\beta_4\i \circ \mathrm{Norm}_{E/\bQ_2})) \in \{4,5\}.
\]
Since these two characters agree on $1+ \varpi_{E}^2\cO_{E} = 1+ 2\cO_{E}$, we conclude  as in \eqref{just-sign}, but now also using \eqref{eqn:eps-supercuspidal-twist} (with \eqref{e:epspsi}) and the odd conductor exponent cases of \Cref{l:vunit,l:evaluate gauss sums with a(mu)>1}, that  
\[
\tst \eps(\frac12, \beta_4^{-1}\pi)  = \pm  \eps(\frac12, \beta_2\beta_4^{-1} \pi), \qxq{where both sides are roots of unity.}
\]
Thus, $x$ and $x'$ are roots of unity, $x = \pm x'$, and \eqref{e:specialcasesuper1} gives
\[
W_{\pi,\,\psi}(g_{-7,\,4,\,v})\in \{x, -x^{-1}\}, \qxq{so also the sought} \val_2(W_{\pi,\,\psi}(g_{-7,\,4,\,v})) \ge 0.
\]
\textbf{The case when $\pi$ is of Type 3.}
Such a $\pi$ is $\mu\St$ for a ramified character $\mu$ with $\mu^2=\mathbf{1}$, and $a(\pi) = 2a(\mu)$.
We twist by the unramified quadratic character 
if needed to assume that $\mu(\varpi_F)=1$: by \eqref{cond-twist} and \eqref{lem:twist by unramified finite order char}, this changes neither $a(\pi)$ nor $\val_p(W_{\pi,\,\psi}(g_{t,\,\ell,\,v}))$.
By \cite{Ass19}*{Lemma 2.1} and \eqref{zeta-values}, for $1 \le \ell \le \frac{a(\pi)}{2}$ and $\chi \in \fX_{\le \ell}$, 
\[
c_{t,\,\ell}(\chi) = \begin{cases}
\eps(\frac 12, \chi\i \mu)^2 \Ga_\psi(\varpi_F^{-\ell}, \chi\i) &\text{if }~\chi \neq \mu,~t = -2a(\mu\chi), \\
\frac 1{q_F}\Ga_\psi(\varpi_F^{-\ell}, \mu\i) &\text{if }~\chi = \mu,~t = -2, \\
-\frac {q_F^2 - 1}{q_F^{3 + t}}\Ga_\psi(\varpi_F^{-\ell}, \mu\i) &\text{if }~\chi = \mu,~t \ge -1, \\
0 & \text{otherwise}.
\end{cases}
\]
By then using the formula \eqref{e:gausseps} for $\Ga_\psi(\varpi_F^{-\ell}, \chi\i)$ together with \eqref{e:epsduality}, we obtain
\be\ba \lab{Type-3-c}
c_{t,\,\ell}(\chi) = \begin{cases}
\frac{1}{q_F-1}q_F^{1-\ell/2}\eps(\frac 12,\chi^{-1}\mu)^2\eps(\frac12, \chi) &\text{if }~ \chi\notin \{\mathbf{1}, \mu\},~ t=-2a(\chi\mu), ~ \ell = a(\chi),\\
-\frac{1}{q_F-1}\mu(-1)&\text{if }~\chi = \mathbf{1},~ t=-2a(\mu), ~\ell=1, \\
\frac{1}{q_F-1}\qF^{-\ell/2}\eps(\frac12, \mu) &\text{if }~\chi =\mu,~ t=-2, ~ \ell = a(\mu),\\
-(\qF + 1)\qF^{-(t+2 + \ell/2)}\eps(\frac12, \mu)&\text{if }~\chi=\mu,~ t\geq -1, ~ \ell = a(\mu),\\
0 & \text{otherwise}.
\end{cases}
\ea\ee
We begin with the case of an odd $p$, when necessarily $a(\mu)=1$, so that $a(\pi)=2$, 
and $\ell = 1$. Since $\mu^2 = 1$, from \eqref{e:schiquadratic} and \eqref{e:schiprop1} we obtain $\frac{p-1}2 \mid  s(\chi^{-1}\mu) + s(\chi)$, so, for $\chi\notin \{\mathbf{1}, \mu\}$, also $2s(\chi^{-1}\mu) +   s(\chi) \ge \frac{p-1}{2} + 1$. Since $\mu^2 = 1$, \Cref{p:maingausseps} and \eqref{Type-3-c} then give the sufficient
\be \label{e:type3arg}
\val_p(c_{t,\,\ell}(\chi)) \ge \begin{cases}
-[\FF:\Fp] + \frac12 + \frac{1}{p - 1} &\text{if }~ \chi\notin \{\mathbf{1}, \mu\},~t=-2,\\
0 &\text{if }~\chi = \mathbf{1},~t=-2, \\
-[\FF:\Fp](t + \frac 52) &\text{if }~\chi=\mu,~t\ge-2, \\
\infty &\x{otherwise.}
\end{cases}
\ee
For the remaining $F=\Q_2$, in the notation of \S\ref{s:type1b}, we have $\mu \in \{\beta_2, \beta_3, \beta_2\beta_3\}$, so 
$a(\pi) = 4$ if $\mu=\beta_2$, and $a(\pi)=6$ if $\mu \in \{\beta_3, \beta_2\beta_3\}$.
It then suffices to use \eqref{eq_hazel-recall}, the values \eqref{Type-3-c}, and \Cref{l:gausssums2} to compute the only possible nonzero $W_{\pi,\,\psi}(g_{t,\,\ell,\,v})$ for $1 \le \ell \le a(\mu)$:
\[
\ba
W_{\pi,\,\psi}(g_{t,\,1,\,v}) &=\tst  -\mu(-1)  \in \{ \pm 1\} \quad\text{if }~t=-a(\pi),\\
W_{\pi,\,\psi}(g_{t,\,2,\,v}) &= \begin{cases}
\frac 12 \eps(\frac 12, \beta_2) \beta_2(v) \in\{
\pm \frac i2\} &\text{if }~\mu=\beta_2,~t=-2,\\
-\frac 3{2^{t + 3}} \eps(\frac 12, \beta_2) \beta_2(v) \in \{
\pm \frac {3i}{2^{t+3}}\} &\text{if }~\mu=\beta_2,~t\ge -1,\\
\eps(\frac 12, \beta_2 \mu)^2 \eps(\frac 12, \beta_2) \beta_2(v) \in  \{ \pm i\} &\text{if }~\mu\in \{\beta_{3}, \beta_2 \beta_3\},~t=-6,
\end{cases}
\\
W_{\pi,\,\psi}(g_{t,\,3,\,v}) &= \begin{cases}
\frac 1{2^{1/2}} \eps(\frac 12, \beta_2)^2 \eps(\frac 12, \beta_2\mu) (\beta_2\mu)(v) \in \{ \pm \frac{1}{2^{1/2}}, \pm \frac{i}{2^{1/2}}\} &\text{if }~\mu\in \{\beta_3, \beta_2\beta_{3}\},~t=-4,\\
\frac 1{2^{3/2}} \eps(\frac 12, \mu) \mu(v) \in  \{ \pm \frac{1}{2^{3/2}}, \pm \frac{i}{2^{3/2}}\} &\text{if }~\mu\in \{\beta_3, \beta_2\beta_{3}\},~t =-2,\\
-\frac {3}{2^{t + 7/2}} \eps(\frac 12, \mu) \mu(v) \in \{ \pm \frac{3}{2^{t + 7/2}}, \pm \frac{3i}{2^{t+ 7/2}} \} &\text{if }~\mu\in \{\beta_3, \beta_2\beta_{3}\},~t\ge -1.
\end{cases}
\ea
\]
\textbf{The case when $\pi$ is of Type 4.} 
Such a $\pi$ is $\mu\absF{\,\cdot\,}^{\sigma} \boxplus \mu\absF{\,\cdot\,}^{-\sigma}$ for $\sigma \neq \pm \frac 12$ and a ramified $\mu \in \fX$ with $\mu^2 = \mathbf{1}$, and $a(\pi)=2a(\mu)$. 
By \cite[Lemma 2.2]{Ass19} and \eqref{zeta-values}, for $1 \le \ell \le \frac{a(\pi)}{2}$ and $\chi \in \fX_{\le \ell}$,
\[
c_{t,\,\ell}(\chi) =
\begin{cases}
\eps(\frac 12, \chi\i\mu \abs{\,\cdot\,}_F^{-\sigma})\eps(\frac 12, \chi\i\mu \abs{\,\cdot\,}_F^{\sigma}) \Ga_\psi(\varpi_F^{-\ell}, \chi\i) &\text{if }~\chi\neq  \mu,~t=-2a(\chi\mu),\\
\frac{1}{\qF} \Ga_\psi(\varpi_F^{-\ell}, \mu) &\text{if }~\chi = \mu,~t=-2, \\
-\frac{q_F - 1}{\qF^{3/2}} \Ga_\psi(\varpi_F^{-\ell}, \mu)(\qF^{-\sigma}+\qF^{\sigma}) &\text{if }~\chi=\mu ,~t=-1,\\
-\frac{q_F - 1}{\qF^{2 + t/2}} \Ga_\psi(\varpi_F^{-\ell}, \mu) \ppp{
\frac{1}{\qF^{\sigma(t+2)}} +\qF^{\sigma(t+2)}  -\sum_{m=0}^t \frac{\qF - 1}{\qF^{\sigma(2m -t)}}
}
 &\text{if }~\chi = \mu,~t\geq 0,\\
0 & \text{otherwise.}
\end{cases}
\]
By then using the formula \eqref{e:gausseps} for $\Ga_\psi(\varpi_F^{-\ell}, \chi\i)$ and the formulas \eqref{e:epsunramtwist}--\eqref{e:epsduality}, we obtain
\[
c_{t,\,\ell}(\chi) =
\begin{cases}
\frac{1}{\qF-1}q_F^{1-\ell/2}\eps(\frac 12,\chi^{-1}\mu)^{2} \eps(\frac12, \chi)  &\text{if }~\chi\notin \{\mathbf{1}, \mu\},~t=-2a(\chi\mu),~\ell=a(\chi),\\
-\frac{1}{\qF-1}\mu(-1) &\text{if }~\chi = \mathbf{1},~ t=-2a(\mu), ~ \ell=1, \\
\frac{1}{\qF-1}\qF^{-\ell/2}  \eps(\frac12, \mu)  &\text{if }~\chi= \mu ,~t=-2,~\ell=a(\mu),\\
-q_F^{-(\ell + 1)/2}\eps(\frac12, \mu)(\qF^{-\sigma}+\qF^{\sigma}) &\text{if }~\chi=\mu ,~t=-1,~\ell=a(\mu),\\
\frac {-\eps(\frac12,\, \mu)}{\qF^{(t+\ell + 2)/2}} \ppp{
\frac{1}{\qF^{\sigma(t+2)}} +\qF^{\sigma(t+2)}  -\sum_{m=0}^t \frac{\qF - 1}{\qF^{\sigma(2m -t)}}
}
 &\text{if }~\chi = \mu,~t\geq 0,~\ell=a(\mu),\\
0 & \text{otherwise.}
\end{cases}
\]
If $p \neq 2$, then $a(\mu)=1$, so $a(\pi)=2$ and $\ell=1$, and, similarly to \eqref{e:type3arg}, we get the sufficient
\[
\val_p(c_{t,\,\ell}(\chi)) \ge \begin{cases}
-[\FF:\Fp] + \frac12 + \frac{1}{p - 1}  &\text{if }~\chi\notin \{\mathbf{1}, \mu\},~t = -2,\\
0 &\text{if }~\chi = \mathbf{1},~t = -2, \\
 -\frac{t + 3}{2} [\bF_F : \Fp] - (t + 2)\abs{\val_p(q_F^\sigma)} &\text{if }~\chi = \mu,~t\geq -2,\\
 \infty &\x{otherwise.}
\end{cases}
\]
In the remaining case $F=\Q_2$, similarly to Type 3, in the notation of \S\ref{s:type1b}, we have $\mu \in \{\beta_2, \beta_3, \beta_2\beta_3\}$, and we combine the above formulas for the $c_{t,\, \ell}(\chi)$ with \eqref{eq_hazel-recall} and \Cref{l:gausssums2} to find the following sufficient formulas for the only possible nonzero $W_{\pi,\,\psi}(g_{t,\,\ell,\,v})$ in the range in question:
\[
\ba
W_{\pi,\,\psi}(g_{t,\,1,\,v}) &\in \{ \pm 1\} \qx{if } ~t=-a(\pi),\\
W_{\pi,\,\psi}(g_{t,\,2,\,v}) &\in \begin{cases}  \{\pm \frac i2\}
&\text{if }~\mu=\beta_2,~t=-2,\\
\{\pm \frac{i}{2^{3/2}}(2^{-\sigma}+2^\sigma)\} &\text{if }~\mu=\beta_2,~t= -1,\\
\left\{\pm \frac{i}{2^{2+t/2}}\left(\frac 1{2^{\sigma(t+2)}}+2^{\sigma(t+2)}-\tst \sum_{m=0}^t \frac 1{2^{\sigma(2m - t)}}\right)\right\}&\text{if }~\mu=\beta_2,~t\ge 0,\\
\{\pm i\} &\text{if }~\mu\in \{\beta_3, \beta_2\beta_3\},~t=-6,
\end{cases}
\\
W_{\pi,\,\psi}(g_{t,\,3,\,v}) &\in \begin{cases}\{ \pm \frac 1{2^{1/2}}, \pm \frac i{2^{1/2}} \}
&\text{if }~\mu\in \{\beta_3, \beta_2\beta_3\},~t=-4,\\
\{ \pm \frac 1{2^{3/2}}, \pm \frac i{2^{3/2}} \} &\text{if }~\mu\in \{\beta_3, \beta_2\beta_3\},~t =-2,\\
\{ \pm\frac 14 (\frac 1 {2^{\sigma}}+2^{\sigma}), \pm\frac i4 (\frac 1 {2^{\sigma}}+2^{\sigma})\} &\text{if }~\mu\in \{\beta_3, \beta_2\beta_3\},~t= -1,\\
 \frac{1}{2^{(t+5)/2}}(\frac 1{2^{\sigma(t+2)}} + 2^{\sigma(t+2)} - \sum_{m=0}^t \frac{1}{2^{\sigma(2m - t)}})\cdot \{ \pm 1, \pm i \}&\text{if }~\mu\in \{\beta_3, \beta_2\beta_3\},~t\ge 0.
\end{cases}
\ea
\]
\textbf{The case when $\pi$ is of Type 5.} 
Such a $\pi$ is $\mu\absF{\,\cdot\,}^{\sigma} \boxplus \mu^{-1}\absF{\,\cdot\,}^{-\sigma}$ for a ramified $\mu \in \fX$ with $\mu^2 \neq \mathbf{1}$, and $a(\pi)=2 a(\mu)$. 
By \cite[Lemma 2.2]{Ass19},\footnote{We corrected a slight mistake in \cite[Lemma 2.2]{Ass19} (see also \cite{Ass19e}): when $\chi_1|_{\cO^\times} \neq \chi_2|_{\cO^\times}$, in the case ``if $a(\mu\chi_j) \neq a(\mu\chi_i) = 0$ for $\{ j, i\} = \{1, 2\}$ and $t \ge -a(\mu\chi_j)$'' of the formula for $c_{t,l}(\mu)$ one should instead have  ``$\zeta_F(1)^{-2}q^{-\frac{t}{2}}\chi_i(\varpi^{t + a(\mu\chi_j)})\chi_j(\varpi^{-a(\mu\chi_j)})G(\varpi^{-a(\mu\chi_j)}, \mu\chi_j)G(\varpi^{-l}, \mu\i)$.''} \eqref{zeta-values}, and \eqref{e:epsunramtwist}, for $1 \le \ell \le \frac{a(\pi)}{2}$ and $\chi \in \Xc_{\le \ell}$, 
\begin{equation*}
c_{t,\,\ell}(\chi) =
\begin{cases}
	\frac{\eps(\frac 12,\, \chi^{-1}\mu^{-1})\epsilon(\frac 12,\, \chi^{-1}\mu)}{\qF^{\sigma(a(\chi\mu^{-1}) - a(\chi\mu))}}\Ga_\psi(\varpi^{-\ell},\chi^{-1}) &\text{if }~\chi\neq \{ \mu^{\pm1}\},~t=-a(\chi\mu)-a(\chi\mu^{-1}),\msk\\
	-\qF^{-\frac{1}2 \pm\sigma (a(\mu^{2})-1)} \eps(\frac 12,\mu^{\mp 2})\Ga_\psi(\varpi^{-\ell},\mu^{\mp1}) &\text{if }~\chi= \mu^{\pm1} ,~t=-a(\mu^2)-1, \msk \\
	\frac{(\qF-1)^2}{\qF^{2+\frac{t}{2}\mp\sigma (t+2a(\mu^2))}}  \Ga_\psi(\varpi^{-a(\mu^2)},\mu^{\pm2})\Ga_\psi(\varpi^{-\ell},\mu^{\mp1}) & \text{if }~\chi=\mu^{\pm 1} ,~t\ge -a(\mu^2), \msk\\
	0 \qq \text{otherwise.}
\end{cases}
\end{equation*}
By then using the formula \eqref{e:gausseps} for the appearing Gauss sums as well as \eqref{e:epsduality}, we obtain
\[
c_{t,\,\ell}(\chi) =
\begin{cases}
-\frac{\mu(-1)}{q_F-1}  &\text{if }~\chi = \mathbf{1},~ t=-2a(\mu),~ \ell=1, \msk \\
\frac{\eps(\frac 12,\,\chi^{-1}\mu^{-1})\eps(\frac 12,\,\chi^{-1}\mu)\eps(\frac12,\, \chi)}{(q_F-1)\qF^{\ell/2 - 1 + \sigma(a(\chi\mu^{-1}) - a(\chi\mu))}} & \text{if }~\chi\notin \{\mathbf{1}, \mu^{\pm1}\},~t=-a(\chi\mu)-a(\chi\mu^{-1}),~\ell=a(\chi),\msk\\
-\frac {\eps(\frac 12,\,\mu^{\mp 2})\eps(\frac12,\, \mu^{\pm 1})} {(q_F-1)\qF^{(\ell - 1)/2 \pm\sigma (1 - a(\mu^{2}))}}  &\text{if }~\chi= \mu^{\pm1} ,~t=-a(\mu^2)-1,~\ell=a(\mu), \msk \\
\frac{\eps(\frac 12,\, \mu^{\mp 2})\eps(\frac12,\, \mu^{\pm 1})}{\qF^{(t+\ell+a(\mu^{2}))/2 \mp\sigma(t+2 a(\mu^{ 2}))}}   & \text{if }~\chi=\mu^{\pm 1} ,~t\ge -a(\mu^2) ,~\ell=a(\mu), \msk\\
0  &\text{otherwise.}
\end{cases}
\]
We begin with $a(\pi)=2$, when $\ell=1$ and $a(\mu)=1$, so $p$ is odd and, since $\mu^2 \ne \mathbf{1}$, also $a(\mu^2)=1$. By \eqref{e:schiprop1}, both $s(\chi^{-1}\mu^{-1}) + s(\chi^{-1}\mu) + 2s(\chi)$ and $s(\mu^{\mp2}) + 2 s(\mu^{\pm1})$ are divisible by $p - 1$, so
\[
\tst s(\chi^{-1}\mu^{-1}) + s(\chi^{-1}\mu) + s(\chi) \ge  \frac{p-1}{2} + 1 \qxq{and}  s(\mu^{\mp2}) +  s(\mu^{\pm1}) \ge \frac{p-1}{2} + 1.
\]
These inequalities, the formulas for the $c_{t,\,\ell}(\chi)$, and \Cref{p:maingausseps}~\ref{MGE-i} imply the sufficient bound
\[
\tst \val_p(c_{t,\,1}(\chi)) \ge
-\frac{4+t}{2}[\FF:\Fp]+ \frac12+\frac{1}{p-1}-(t+2)|\val_p(\qF^\sigma)|  \qxq{for} \chi\in\Xc_{\le 1}.
\]
In the case $a(\pi)>2$, that is, $a(\mu)>1$, we begin with $1\le \ell <a(\mu)=\frac{a(\pi)}{2}$, so that $t=-a(\pi)$ by \ref{ass-1}.  In this case, the formulas above for the $c_{t,\,\ell}(\chi)$ and \Cref{p:maingausseps} give
\[
\val_p(c_{-a(\pi),\,\ell}(\chi)) = \begin{cases}
\frac{s(\chi\i)}{p - 1} &\text{if }~\ell=1,\\
 [\FF:\Fp](1 - \frac{\ell}{2})&\text{if }~1<\ell<\frac{a(\pi)}{2},
\end{cases}
\]
which reduces us further to the last setting of \eqref{e:p2mainlocal}, in which, in addition, $F=\Q_2$ and $\ell=3$. In the notation of \S\ref{s:type1b}, the formulas above  for the $c_{t,\,\ell}(\chi)$ and \eqref{eq_hazel-recall} then give
\[
\tst W_{\pi,\,\psi}(g_{-a(\pi),\,3,\,v}) =\frac{1}{2^{1/2}}\sum_{\beta \in \{\beta_3,\, \beta_2\beta_3\} } \ppp{\eps(\frac12, \beta) \eps(\frac12, \beta\mu) \eps(\frac12, \beta\mu^{-1}) \beta(v)}.
\]
Since $\eps(\frac12, \beta\mu) \eps(\frac12, \beta\mu^{-1}) = (\beta\mu)(-1) \in \{ \pm 1\}$ by \eqref{e:epsduality}, \Cref{l:gausssums2} gives the sufficient
\[
\tst W_{\pi,\,\psi}(g_{-a(\pi),\,3,\,v}) \in \{ \pm \frac{1 + i}{2^{1/2}}, \pm \frac{1 - i}{2^{1/2}} \}.
\]
The remaining case is $a(\pi)>2$, that is, $a(\mu)>1$, with $\ell=\frac{a(\pi)}{2}= a(\mu)$, in which the above formulas for the $c_{t,\,\ell}(\chi)$ allow us to restrict to $\chi$ with $a(\chi)=a(\mu)$. For odd $p$,  since $a(\mu) > 1$, we have $a(\mu^2) = a(\mu)$, so if also $a(\chi\mu^{\pm 1})<a(\mu)$, then $a(\chi\mu^{\mp 1})= a(\chi\mu^{\pm 1}\cdot \mu^{\mp 2}) = a(\mu)$. Thus, for odd $p$, the above formulas for the $c_{t,\,\ell}(\chi)$ combine with \Cref{p:maingausseps} to give the sufficient bounds
\[
\val_p(c_{t,\,a(\mu)}(\chi)) \ge
\begin{cases}
-[\FF:\Fp](\frac{a(\mu)}{2} - 1)-\left|(a(\chi\mu^{-1})-a(\chi\mu))\val_p(\qF^{\sigma})\right| \\   \mkern248mu\text{if }~a(\chi\mu), a(\chi\mu^{-1})>1,~t=-a(\chi\mu)-a(\chi\mu^{-1}), \msk\\
\frac{[\FF:\Fp](1 - a(\mu))}{2}+\frac{s(\chi\mu^{\pm 1})}{p-1}-\left|(a(\mu)-1)\val_p(\qF^{\sigma})\right| \mkern8mu \text{if }~a(\chi\mu^{\pm 1})=1,~t=-a(\mu)-1, \\
-[\FF:\Fp](\frac{t}{2}+a(\mu)) -\left|(t+ 2a(\mu))\val_p(\qF^\sigma)\right|
 \mkern54mu \text{if }~\chi = \mu^{\pm 1},~t\geq -a(\mu) - 1. 
\end{cases}
\]
We are left with $F=\Q_2$, when $\mu^2\neq \mathbf 1$ gives $a(\mu)\ge 4$ (see \S\ref{s:type1b}), so $a(\pi) \ge 8$ and $a(\mu^2)=a(\mu)-1$. If $\chi \notin \{\mu^{\pm1}\}$, then, since $a(\chi) = a(\mu)$, exactly one of $a(\chi\mu)$ and $a(\chi\mu^{-1})$ equals $a(\mu)-1$, and the other one lies in $[2, a(\mu)-2]$ (compare with \cite[Lemma 2.2]{CS18}). Thus, for such $\chi$ we have $- a(\chi\mu)-a(\chi\mu^{-1}) \le -a(\mu)-1$ and, furthermore, $\left|a(\chi\mu^{-1})-a(\chi\mu)\right| = 2a(\mu)-2- a(\chi\mu)-a(\chi\mu^{-1})$. Thus, the formulas above for the $c_{t,\,\ell}(\chi)$ and
\Cref{p:maingausseps} give the sufficient final bounds
\[
\ \ \val_2(c_{t,\,a(\mu)}(\chi)) \ge
\begin{cases}
1 - \frac{a(\mu)}{2}-(t + 2a(\mu)-2)|\val_2(2^{\sigma})|  &\text{if }~\chi \neq \mu^{\pm 1},~t=-a(\chi\mu)-a(\chi\mu^{-1}), \\
\frac{1 - a(\mu)}{2}-(a(\mu) - 2)|\val_2(2^\sigma)|
 &\text{if }~\chi = \mu^{\pm 1},~t = -a(\mu),\\
 \frac{1 - t}{2} - a(\mu)-(t+ 2a(\mu) - 2)|\val_2(2^\sigma)|
 &\text{if }~\chi = \mu^{\pm 1},~t\geq -a(\mu)+1. \mkern40mu\QED
\end{cases}
\]

\section{$p$-adic valuations of Fourier coefficients at cusps}
\label{s:global p-adic valuations}

We turn to global consequences of the local analysis of the preceding section, more precisely, to \Cref{main-estimates} that $p$-adically bounds the Fourier expansions at cusps  of holomorphic newforms on $\Gamma_0(N)$. For this, we begin by reviewing notions that concern cusps and Fourier expansions. 

\bpp[Cusps] \label{pp:cusps}
The group $\SL_2(\mathbb{R})$ acts by M\"{o}bius transformations on the extended upper half-plane
\[
\mathfrak{H}^* \ce \mathfrak{H}  \cup \bP^1(\bQ) \qxq{with} \mathfrak{H} \ce \{ z \in \bC \colon \im(z) > 0\}
\]
 and, for an $N \ge 1$, the set of \emph{cusps} of $\Gamma_0(N)$ is the orbit space
\[
\cusps(\Gamma_0(N)) \ce \left(\Gamma_0(N) \cap \SL_2(\Z) \right) \bs \bP^1(\bQ).
\]
Since $\SL_2(\Z)$ acts transitively on $\bP^1(\bQ)$ and the stabilizer of $\infty \in \bP^1(\bQ)$ is $\{\pm
  \abcd 1 *\ 1\}$, we have 
\[
\cusps(\Gamma_0(N)) \cong (\Gamma_0(N) \cap \SL_2(\Z))\backslash
\SL_2(\Z)/ \{\pm \abcd    1 * \ 1\},
\]
and the latter is the global analogue of the local double coset set $ZU\bs \GL_2(F) / K_1(n)$ of \S\ref{elements-g}. Via the complex uniformization of $X_0(N)$, that is, via the identification of Riemann surfaces
\be \label{uniformization}
X_0(N)(\bC) \cong (\Gamma_0(N) \cap \SL_2(\bZ))\backslash\fH^*
\ee
(see \cite{Roh97}*{Section 1.10, Proposition 7}), the cusps are the complement of the elliptic curve locus of $X_0(N)_\bC$.


Concretely, each cusp $\fc$ of $\Gamma_0(N)$ is represented by an $\frac mL \in \bQ \subset \bP^1(\bQ)$ with $\gcd(m,N)=1$ and a uniquely determined \emph{denominator} $L\mid N$ of $\fc$ (compare with \cite{DS05}*{Proposition 3.8.3}). For $\fc = \abcd abcd \infty$, we have $L = \gcd(c, N)$.
The cusp $\infty$ is the unique one of denominator $N$ and there are $\phi(\gcd(L, \frac NL))$ cusps of denominator $L$ (see \emph{loc.~cit.}). The \emph{width} of a cusp $\fc$ is the smallest $w(\fc) \in \bZ_{> 0}$ such that $\gamma \abcd 1{w(\fc)}01 \gamma\i \in \Gamma_0(N)$ for any fixed $\gamma \in \SL_2(\bZ)$ with $\fc = \gamma\infty$, explicitly, 
\[
\tst w(\fc) = \frac{N}{\gcd(L^2,\,N)}.
\]
\epp

\bpp[Fourier expansions] \label{pp:q-expansions}
For a function $f\colon \fH \ra \bC$, a $k \in \bZ_{> 0}$, and a $\gamma = \abcd abcd \in \GL_2^+(\R)$,
\[
\tst \xq{the function} f|_k \gamma \colon \fH \ra \bC \qxq{is defined by} (f|_k \gamma)(z) \ce \det(\gamma)^{\frac{k}2}\frac{1}{(cz + d)^k}f(\frac{az+b}{cz+d}).
\]
If the ideal $\{h \in \Z: f = f|_k{\abcd 1 h \ 1} \} \subset \Z$ is nonzero, generated by a unique $w \in \bZ_{> 0}$, then $f$ descends along the map $\fH \surjects \bC^\times$ given by $z \mapsto e^{2\pi i z/w}$ to a function $f_0 \colon \bC^\times \ra \bC$. If then $f_0$ extends to a holomorphic function at $0$, then $f$ is \emph{holomorphic at $\infty$} and
we obtain its \emph{Fourier expansion at $\infty$}:
\be \label{eqn:FE-genl}
\tst f(z) = \sum_{n\geq 0} a_f(n)e^{\frac {2 \pi i n z}w}.
\ee
We say that such an $f$ is \emph{cuspidal at $\infty$} if $a_f(0) = 0$.

For a subgroup $\Gamma_1(N) \subset\Gamma \subset \GL_2(\wh{\bZ})$ and a $k \in \bZ_{> 0}$, a \emph{modular form} (resp.,~a \emph{cuspform}) of weight $k$ on $\Gamma$  is a holomorphic function $f\colon \fH \ra \bC$ such that both $f|_k \gamma = f$ for $\gamma \in \Gamma \cap \SL_2(\Z)$ and $f|_k \gamma'$ is holomorphic (resp.,~cuspidal) at $\infty$ for $\gamma' \in \SL_2(\bZ)$.
A cuspform $f$ on $\Gamma$ is \emph{normalized} if $a_f(1) = 1$. For instance, for $\Gamma = \Gamma_0(N)$, choosing $\gamma = \abcd {-1} \ \ {-1}$ gives $f(z) = (-1)^k f(z)$, so $k$ is even or $f = 0$.

For every modular form $f$ of weight $k$ on $\Gamma_0(N)$  and every cusp $\fc = \gamma \infty$ with $\gamma 
\in \SL_2(\bZ)$, we have $(f|_k \gamma)|_k (\begin{smallmatrix} 1 & {w(\fc)} \\  & 1\end{smallmatrix}) = f|_k \gamma$, so \eqref{eqn:FE-genl} gives the \emph{Fourier expansion of $f$ at $\fc$}:
\[
\tst (f|_k\gamma)(z) = \sum_{n\geq 0} a_f(n;\gamma)e^{\frac{2 \pi i n z}{w(\fc)}},
\]
which depends not only on $\fc$ but also on $\gamma$---explicitly, for any $\gamma' \in \SL_2(\bZ)$ with $\fc = \gamma'\infty$,
\[
a_f(n ; \gamma)= e^{\frac{2 \pi i nt}{w(\fc)}}a_f(n ; \gamma') \qxq{for some} t \in \bZ \qxq{that depends on} \gamma'^{-1} \gamma.
\]
In particular, for any isomorphism $\ov{\bQ}_p \simeq \bC$ and the resulting $p$-adic valuation $\val_p\colon \bC \ra \bQ \cup \{ \infty\}$,
\begin{equation} \label{def-val-cusp}
\tst \val_p(f|_\fc) \ce \inf_{n \ge 0}(\val_p(a_f(n; \gamma))) \qx{depends only on $f$ and $\fc$, and not on $\gamma$.
}
\end{equation}
\epp


\bpp[The representation $\pi_f$] \label{pif}
For a normalized newform $f$ on $\Gamma_1(N)$ (see \cite{Li75}*{page~294}),\footnote{Here and throughout the paper, a `newform' is implicitly assumed to be a (holomorphic) cuspform.} 
the Fourier coefficients $a_f(n)$ are algebraic integers that generate a number field $K_f$ (see, for instance, \cite{DI95}*{Corollary 12.4.5}). In particular, for a normalized newform $f$ on $\Gamma_0(N)$ and every prime $p$, we have $\val_p(f|_\infty) = 0$. For such an $f$, the Fourier coefficients $a_f(n; \gamma)$ at any cusp $\fc = \gamma \infty$ of denominator $L$  lie in $K_f(\zeta_{N/L})$ (see \cite[Theorem 7.6]{BN19}, which even exhibits the possibly smaller number field generated by the $a_f(n; \gamma)$), and to study them $p$-adically we will use the adelic viewpoint.

Namely, for a newform $f$ on $\Gamma_1(N)$, we let $\pi_f$ be the cuspidal, irreducible, admissible, automorphic $\GL_2(\bA_\bQ)$-representation spanned by the $\GL_2(\bA_\bQ)$-translates of the adelic newform associated to $f$ (see \cite{Gel75}*{Theorem 5.19}). 
 In the resulting factorization (compare with \cite{Fla79}*{Theorem 3})
\[
\tst \pi_f \cong \pi_{f,\, \infty} \tensor  \bigotimes'_{p < \infty }\pi_{f,\, p}
\]
each $\pi_{f,\, p}$ is an irreducible, admissible, infinite-dimensional 
representation of $\GL_2(\bQ_p)$ of conductor exponent $\val_p(N)$. If $f$ is on $\Gamma_0(N)$, then $\omega_{\pi_{f,\, p}} = \mathbf{1}$, and if also $\val_p(N) \ge 2$, then $\pi_{f,\, p}$ is of Type 1, 3, 4, or 5 in the classification of \uS\uref{s:repclassification}. In the last two cases, we have the following refinement.
\epp

\begin{lemma}\label{l:classi} \label{prop:properties of a_p}
For a prime $p$ and a newform $f$ of weight $k$ on $\Gamma_0(N)$ with $\val_p(N)\ge 2$,  
if the $\GL_2(\Q_p)$-representation $\pi_{f,\, p}$ is of Type {\upshape4} or {\upshape5}, that is, if
\[
\tst \pi_{f,\, p}\simeq \mu\abs{\,\cdot\,}_{\Q_p}^{\sigma} \boxplus \mu^{-1}\abs{\,\cdot\,}_{\Q_p}^{-\sigma} \qxq{for a ramified}  \mu \in \fX_{\Q_p} \qxq{such that} \sigma \neq \pm \frac 12 \qxq{when} \mu^2 = \mathbf1,
\]
then $\sigma \in i\R$ and $p^{ \pm {\sigma}+\frac{k-1}2} \in \overline{\Z}$, so that $\abs{\val_p(p^\sigma)} \le \frac{k-1}{2}$.
\end{lemma}

\begin{proof}
By the Ramanujan--Petersson conjecture at all finite places (see, e.g., \cite{Bla06}*{Theorem~1 and Remark on page~46}), the characters $\mu\abs{\,\cdot\,}_{\Q_p}^{\sigma}$ and $\mu^{-1}\abs{\,\cdot\,}_{\Q_p}^{-\sigma}$ are unitary, so $\sigma \in i\R$. By complex conjugation, it then remains to show that $p^{-\sigma+\frac{k-1}2} \in \ov{\bZ}$. For this, we first globalize $\mu$ to a finite order character $\wt{\mu} \colon  \A^\times_\bQ/\Q^\times \ra \bC^\times$ (compare with \cite{AT09}*{Chapter X, Section 2, Theorem 5}), set $\wt{\pi} \ce \wt{\mu}\pi_f$, and let $\wt{f}$ be the normalized newform of weight $k$ on $\Gamma_1(\wt{N})$ for which $\pi_{\wt{f}} \simeq \wt{\pi}$ (see \cite{Gel75}*{Theorem~5.19}), so that  $a_{\wt{f}}(p) \in \ov{\bZ}$ (see \S\ref{pif}).

If $\pi_{f,\, p}$ is of Type 4, then $\pi_{\wt{f},\, p} \simeq \abs{\,\cdot\,}_{\Q_p}^{\sigma} \boxplus \abs{\,\cdot\,}_{\Q_p}^{-\sigma}$ with $\sigma \neq \pm \frac 12$, so \cite{CS18}*{equation before (30)} gives
\[
\tst a_{\wt{f}}(p) = p^{\frac k2}W_{\pi_{\wt{f},\, p},\, \psi_p}(\abcd p \ \ 1) \overset{\x{\cite{PSS14}*{equation (121)}}}{=} p^{\frac{k - 1}2} (p^{\sigma} + p^{-\sigma}),
\]
where $\psi_p\colon \bQ_p \ra \bC^\times$ is an additive character  with $c(\psi_p) = 0$ and $W_{\pi_{\wt{f},\, p},\, \psi_p}$ is the normalized Whittaker newform of $\pi_{\wt{f},\, p}$ (see \S\ref{pp:Whittaker}).  Checking prime by prime, we obtain the sought $p^{-\sigma+\frac{k-1}2} \in \ov{\bZ}$.

If $\pi_{f,\,p}$ is of Type 5, then  $\pi_{\wt{f},\, p} \simeq \mu^2\abs{\,\cdot\,}_{\Q_p}^{\sigma} \boxplus \abs{\,\cdot\,}_{\Q_p}^{-\sigma}$ with $\mu^2 \neq \mathbf 1$, so \cite{CS18}*{equation (30)} gives
\[
\tst a_{\wt{f}}(p) = p^{\frac k2}W_{\pi_{\wt{f},\, p},\, \psi_p}\ppp{\abcd{p}{}{}{1}} \prod_{q\mid \wt{N},\, q\ne p}W_{\pi_{\wt{f},\, q},\, \psi_q}\ppp{\abcd{p}{}{}{1}}
\]
with $\psi_q$ and $W_{\pi_{\wt{f},\, q},\, \psi_q}$ as before. Since $\abcd{p}{}{}{1}=\abcd p00p\abcd{1}{}{}{p^{-1}}$, the factors for $q \neq p$ are all roots of unity (see \S\ref{pp:Whittaker}), so \cite{PSS14}*{equation (121)} now directly implies the sought $p^{-\sigma+\frac{k-1}2} \in \ov{\bZ}$.
\end{proof}

The following key lemma uses the adelic point of view to link the global $p$-adic valuation $\val_p(f|_\fc)$  to the local $p$-adic valuations $\val_p(W_{\pi_{f,\,p},\, \psi_p}(g_{t,\,\ell,\,v}))$ that were bounded in \Cref{T1,T2}.

\blem\label{prop:localgloballink}
For a prime $p$, a normalized newform  $f$ of weight $k$ on $\Gamma_0(N)$, and a cusp $\fc$ in $X_0(N)(\bC)$ of denominator $L$,
\be\label{e:pnmidN}
\xq{if} p\nmid N, \qxq{then} \val_p(f|_\fc) \ge 0.
\ee
If, in contrast, $p \mid N$, then, setting $\pi 
\ce \pi_{f,\, p}$ for brevity \up{see \uS\uref{pif}}, for any additive character $\psi \colon \Q_p \ra \bC^\times$ with $c(\psi)=0$, with the notation of \uS\uS\uref{pp:Whittaker}--\uref{elements-g}, we have  
\[
\tst \val_p(f|_\fc) \ge - \frac k2 \val_p\ppp{\frac{N}{\gcd(L^2,\, N)}} + \min_{\tau \in \Z_{\ge 0}, \, v \in \Z_p^\times }\left( \frac{k\tau}{2}+ \val_p(W_{\pi,\, \psi}(g_{\tau - \max(\val_p(N), \, 2\val_p(L)),\, \val_p(L),\, v}))\right)\!.
\]
\elem

\begin{proof}
We included \eqref{e:pnmidN} because it follows from the argument below, though \cite{DI95}*{Remark~12.3.5} gives it, too. We fix additive characters $\psi_q \colon \Q_q \ra \bC^\times$ with $c(\psi_q)=0$ for each prime $q \mid N$ such that $\psi_p = \psi$ in the case $p \mid N$, we fix a $\gamma=\sabcd mb{L}d\in\SL_2(\Z)$ with $\fc = \gamma \infty$, and we consider a variable Fourier coefficient $a_f(r;\gamma)$. 
By \cite{CS18}*{Proposition 3.3}, 
there are $v_q \in \Z_q^\times$ (that depend on $r$) such that
\[
\tst a_f(r;\gamma) = a_f(r_0)e^{\frac{2\pi ird}{w(\fc)L}} \prod_{q\mid N} q^{\frac k2 \ppp{\val_q(r) - \val_q\ppp{\frac{N}{\gcd(L^2, N)}}}} W_{\pi_{f,\, q},\, \psi_q}(g_{\val_q(r) - \max(\val_q(N), \, 2\val_q(L)),\, \val_q(L),\, v_q}),
\]
where $r_0 \ce \prod_{q \nmid N} q^{\val_q(r)}$. Since $a_f(r_0) \in \ov{\bZ}$ (see \S\ref{pif}) and $W_{\pi_{f,\, q},\, \psi_q}$ takes values in $\ov{\bZ}[\frac 1q]$ (see \Cref{prop:Type1b} with \Cref{prop:properties of a_p}), it remains to take $p$-adic valuations and let $r$ vary.
\end{proof}

We are ready to bound the $p$-adic valuations of Fourier expansions of newforms at cusps. 
\bthm \label{main-estimates} \label{t:main}
For a prime $p$, a  cuspform  $f$ that is a $\ov{\bZ}$-linear combination of normalized newforms of weight $k$ on $\Gamma_0(N)$, 
 a cusp $\fc \in X_0(N)(\bC)$ of denominator $L$, and an isomorphism $\bC \simeq \ov{\bQ}_p$, 
\[
\val_p(f|_\fc) \geq\tst  - \frac k2 \val_p\ppp{\frac{N}{\gcd(L^2,\, N)}} + \begin{cases}
0  &\text{if }~\val_p(\gcd(L, \frac NL)) = 0, \\
0  &\text{if }~\val_p(\gcd(L, \frac NL)) = 1,~\val_p(N) >2, \\
- \frac12 & \text{if }~\val_p(L) = \frac{1}{2}\val_p(N) = 1,  \\
1 -   \frac 12 \val_p(\gcd(L, \frac NL))  & \text{otherwise,}
\end{cases}
\]
as well as the following stronger bounds in the case $p = 2$\ucolon
\[
\tst \val_2(f|_\fc)\geq \tst- \frac k2 \val_2\ppp{\frac{N}{\gcd(L^2,\, N)}} + \begin{cases}
 0 & \text{if }~\val_2(L) = \frac{1}2\val_2(N) =1,\\
\frac{k}2 & \text{if }~\val_2(L) = \frac{1}2\val_2(N) \in \{2, 3, 4\}, \\ \frac{k}2 + 1 - \frac{1}{4}\val_2(N) &\text{if }~\val_2(L) = \frac{1}2\val_2(N) > 4, \\
 0  & \text{if }~\val_2(\gcd(L, \frac NL)) = 3,~\val_2(N) > 6.
\end{cases}
\]
\ethm

\begin{proof}
We lose no generality by assuming that $f$ is a normalized newform of weight $k$ on $\Gamma_0(N)$, so we set $\pi \ce \pi_{f,\, p}$ (see \S\ref{pif}) and fix an additive character $\psi \colon \bQ_p \ra \bC^\times$ with $c(\psi) = 0$.

 The case $\val_p(N) =0$ follows from \eqref{e:pnmidN}. 
In the case $\val_p(N) =1$, we have $\val_p(L)\in \{0,1\}$ and $a(\pi)=1$, and \Cref{prop:localgloballink} reduces us to showing that
\[
\tst \frac{k\tau}{2}+ \val_p(W_{\pi,\, \psi}(g_{\tau - \max(1, \, 2\val_p(L)),\, \val_p(L),\, v})) \ge 0
\]
for every $\tau \in  \bZ_{\ge 0}$ and $v \in \bZ_p^\times$, which follows from the first case of Proposition \ref{p:otherprimes}.

In the remaining case $\val_p(N) \ge 2$, by \S\ref{pif}, the representation $\pi$ is of Type 1, 3, 4, or 5 with $a(\pi) = \val_p(N)$, and 
 \Cref{prop:localgloballink} reduces us to showing that for $\tau \in  \bZ_{\ge 0}$ and $v \in \bZ_p^\times$ the quantity
\be \lab{key-quantity}
\tst \frac{k\tau}{2}+ \val_p(W_{\pi,\, \psi}(g_{\tau - \max(\val_p(N), \, 2\val_p(L)),\, \val_p(L),\, v}))
\ee
is at least the summand split into different cases in the desired inequalities. When $\val_p(L) \neq \frac{\val_p(N)}{2}$, this is immediate from \Cref{T1,T2}, so we assume from now on that $\val_p(L) = \frac{\val_p(N)}{2}$.

For $\pi$ of Type 3, if $p$ is odd, then \Cref{T1}~\ref{1-ii} shows that $\val_p(N) = 2$ and gives the conclusion (after plugging in the bounds from \Cref{T1}~\ref{1-ii}, the expression \eqref{key-quantity} becomes linear in $\tau$, so its extrema are at the endpoints of the range for $\tau$), and if $p = 2$, then \Cref{T2}~\ref{case-iii} (with \S\ref{s:type1b}) shows that $\val_2(N) \in \{4, 6\}$ and  gives the conclusion. 
For $\pi \simeq \mu \abs{ \cdot}_{\bQ_p}^\sigma \oplus \mu\i \abs{ \cdot}_{\bQ_p}^{-\sigma}$ of Type 4 or 5, \Cref{l:classi} shows that $|\val_p(p^\sigma)| \le \frac{k-1}{2}$, so \Cref{T1,T2} likewise give the conclusion. 

In the remaining case when $\pi$ is of Type 1, for odd $p$, by \Cref{l:vanishing of Whittaker newform}, we may restrict to $\tau = 0$, and then conclude by \Cref{T1}. In contrast, for $p = 2$, we combine \Cref{lem:conductors} and \Cref{l:vanishing of Whittaker newform} to reduce either to $a(\pi) = 2$ with $\tau = 0$ or to $\tau > 0$,
and then use \Cref{T2}.
\epf

We explicate the weight $2$ case of \Cref{main-estimates} because it is the most relevant one for our goals.

\begin{corollary}\label{cor:globalbdfcelliptic}
For  a prime $p$, a $\ov{\bZ}$-linear combination $f$ of  normalized newforms of weight $2$ on $\Gamma_0(N)$, a cusp $\fc \in X_0(N)(\bC)$ of denominator  $L$, and an isomorphism $\bC \simeq \ov{\bQ}_p$, 
\[
\tst
\val_p(f|_\fc)\geq -\val_p(\frac NL)+\begin{cases} 0 &\text{if }~\val_p(L) \in \{0,\val_p(N)\}, \\
\max(\frac12,\frac{1}{p-1}) &\text{if }~\val_p(L) = 1,~\val_p(N)=2,  \\
 1 &\text{if }~\val_p(L) \in \{1, \val_p(N)-1\},~\val_p(N)>2,   \\
 1 + \frac{1}2\val_2(N) &\text{if }~p=2,~  \val_2(L) = \frac{1}2\val_2(N) \in \{2,3,4\}, \\
 2 + \frac{1}4\val_2(N) &\text{if }~p=2,~  \val_2(L) = \frac{1}2\val_2(N) >4,\\
 3&\text{if }~p=2,~  \val_2(L) \in\{3,\val_2(N)-3\},~ \val_2(N)>6,\\
  1+ \frac{1}{2}\val_p(\gcd(L, \frac NL)) &\text{otherwise.} \QED
\end{cases}
\]
\end{corollary}

\begin{example}\label{exmp:Fourier expansions}
In \Cref{table1,table2}, for newforms $f$ associated to elliptic curves of conductor $N$, we used the SageMath algorithm\footnote{Available at \url{https://github.com/michaelneururer/products-of-eisenstein-series}. A faster and more general pari/gp algorithm  for algebraically computing Fourier expansions at cusps is based on \cite{Coh19}, but we did not use it because it is heuristic: to convert the numerically approximated Fourier coefficients to algebraic numbers, it uses a heuristic application of the LLL-algorithm. Our denominator bounds could help make this algorithm rigorous.
} described in \cite{DN18}*{Section 6} to compute the valuations $\val_p(f|_{\frac{1}{p^\ell}})$ for  $0< \ell \le \frac{1}{2}\val_p(N)$ (the restriction to this range is natural due to the Atkin--Lehner involutions). The resulting examples illustrate the sharpness of  \Cref{cor:globalbdfcelliptic}.
\numberwithin{table}{subsection}
	\renewcommand\thefigure{\ref*{exmp:Fourier expansions}.\arabic{figure}}
	\begin{table}[H]
			\begin{tabular}{c|c|c|c|c|c|c}
				Newform $f$& Level & Label & $\val_2(f|_{\frac 12})$ & $\val_2(f|_{\frac 14})$ & $\val_2(f|_{\frac 18})$ & $\val_2(f|_{\frac{1}{16}})$
				\\ &&&&&& \\[-1em]
				\hline &&&&&& \\[-1em]
				$q - 2q^{3} - q^{5} + 2q^{7} + q^{9} + O(q^{10})$ & $2^2 \cdot 5$ & 20a & $0$ & $ $  &  & \\

				$q - q^{3} - 2q^{5} + q^{9} + O(q^{10})$ & $2^3 \cdot 3$ & 24a & $-1$ & $ $ & $ $ & \\

				$q + q^{3} - 2q^{5} + q^{9} + O(q^{10})$ & $2^4 \cdot 3$ & 48a & $-2$ & $1$ & $ $ & $ $\\
				
				$q - 2q^{5} - 3q^{9} + O(q^{10})$ & $2^5$ & 32a & $-3$ & $-1$ & $ $ & $ $\\
				$q + 2q^5 - 3q^9 + O(q^{10})
				$ & $2^6$ & 64a & $-4$ & $-2$ & $1$ & $ $ \\
				$q - 2q^{3} + 2q^{5} + 4q^{7} + q^{9} + O(q^{10})
				$ & $2^7$ & 128b & $-5$ & $-3$ & $-1$ & $ $ \\
				$q + 4q^5 - 3q^9 + O(q^{10})$ & $2^8$ & 256c & $-6$ & $-4$ & $-2$ & $1$
			\end{tabular}
			\msk
			\caption{\label{table1}$p$-adic valuations of Fourier expansions for $p=2$ and small levels}
	\end{table}
\vspace{-24pt}
	\begin{table}[h]
		\begin{tabular}{c|c|c|c|c|c}
			Newform $f$& $p$ & Level & Label & $\val_p(f|_{\frac 1p})$ & $\val_p(f|_{\frac{1}{p^2}})$
			\\ &&&& \\[-1em]
			\hline &&&& \\[-1em]
			$q + q^{2} - q^{4} - q^{5} - 3q^{8} + O(q^{10})
			$ & $3$ &$p^2\cdot 5$ & 45a & $-\frac12$ & $ $
			\\ &&&& \\[-1em]
			$q - 2q^{4} - q^{7} + O(q^{10})$ &$3$ & $p^3$ & 27a & $-1$ & $ $
			\\ &&&& \\[-1em]
			$q + q^{2} + q^{4} + 3q^{5} - 4q^{7} + q^{8} + O(q^{10})$ &$3$ & $2\cdot p^4$ & 162d & $-2$ & $0$
			\\ &&&& \\[-1em]
			$q - 2q^{4} + 5q^{7} + O(q^{10})$ &$3$ & $p^5$ & 243b & $-3$ & $-1$
			\\ &&&& \\[-1em]
			$q + q^{2} + q^{3} - q^{4} + q^{6} - 3q^{8} + q^{9} + O(q^{10})$ & $5$ & $3\cdot p^2$ & 75b & $-\frac12$ & $ $
			\\ &&&& \\[-1em]
			$q - q^{2} + 2q^{3} + q^{4} - 2q^{6} - q^{8} + q^{9} + O(q^{10})$ & $7$ & $2\cdot p^2$ & 98a & $-\frac12$ & $ $
			\\ &&&& \\[-1em]
			$q + 2q^{2} - q^{3} + 2q^{4} + q^{5} - 2q^{6} + 2q^{7} - 2q^{9} + O(q^{10})$
			 & $11$ & $p^2$ & 121d & $-\frac12$ & $ $
		\end{tabular}
		\msk
		\caption{\label{table2}$p$-adic valuations of Fourier expansions for $3\le p\le 11$ and small levels}
	\end{table}
\end{example}

\vspace{-24pt}

\section{The differential determined by a newform lies in the $\bZ$-lattice $H^0(X_0(N),
\Omega)$} \label{canonical-lattice}


Any  cuspform $f$ of weight $2$ on $\Gamma_0(N)$ that has a rational Fourier expansion determines a differential form $\omega_f$ on $X_0(N)_\bQ$. The goal of this section is to use the results of \S\ref{s:global p-adic valuations} to show in \Cref{wf-integral} that, in particular, if such an $f$ is a normalized newform (that then corresponds to an isogeny class of elliptic curves over $\bQ$), then $\omega_f$ is integral in the sense that it lies in the $\bZ$-lattice
\[
H^0(X_0(N), \Omega) \subset H^0(X_0(N)_\bQ, \Omega^1),
\]
where $\Omega$ is the relative dualizing sheaf. For arguing this, it is convenient to work with the regular stack $\sX_0(N)$ that has both a modular interpretation and line bundles of modular forms instead of the possibly singular scheme $X_0(N)$ whose scheme-theoretic points lack a clear modular description. Thus, we begin by reviewing the definition of the ``relative dualizing'' sheaf in the stacky case. Some material of this section overlaps with the appendix of the unpublished manuscript \cite{Manin-Stevens}.


\bpp[``Relative dualizing sheaves'' of Deligne--Mumford stacks] \lab{rel-dual}
Let $X \ra S$ be a flat, locally finitely presented morphism of schemes with Cohen--Macaulay fibers. By \cite{SP}*{Lemma~\href{https://stacks.math.columbia.edu/tag/02NM}{02NM}}, the scheme $X$ is a disjoint union of clopen subschemes whose relative dimension over $S$ is constant. Thus, the theory of Grothendieck duality, specifically \cite{Con00}*{bottom halves of pages 157 and 214}, supplies relative dualizing $\sO_X$-module $\Omega_{X/S}$ that is quasi-coherent, locally finitely presented, $S$-flat, and of formation compatible with base change in $S$. For instance, if $X \ra S$ is smooth, then $\Omega_{X/S}$ is simply the top exterior power of the vector bundle $\Omega^1_{X/S}$.  The formation of $\Omega_{X/S}$ is compatible with \'{e}tale localization on $X$: for every \'{e}tale $S$-morphism $f\colon X' \ra X$ one has a canonical isomorphism
\be \lab{et-omega}
\iota_f \colon f^*(\Omega_{X/S}) \isomto \Omega_{X'/S}
\ee
supplied by \cite{Con00}*{Theorem 4.3.3 and bottom half of page 214}. Moreover, if $f' \colon X'' \ra X'$ is a further \'{e}tale $S$-morphism, then \cite{Con00}*{equation (4.3.7) and bottom half of page 214} supply the following compatibility:
\be \lab{iso-comp}
\iota_{f \circ f'} = \iota_{f'} \circ ((f')^*(\iota_{f})) \colon (f')^*(f^*(\Omega_{X/S})) \isomto \Omega_{X''/S}.
\ee
Let now $\sX \ra S$ be a flat and locally of finite presentation morphism of Deligne--Mumford stacks with Cohen--Macaulay fibers.
By working \'{e}tale locally on $S$, the compatibilities \eqref{iso-comp} ensure\footnote{See \cite{LMB00}*{Lemme 12.2.1} for a discussion of analogous compatibilities and their relevance for glueing.} that the $\sO_X$-modules $\Omega_{X/S}$ for \'{e}tale morphisms $X \ra \sX$ from a scheme $X$ glue to a quasi-coherent, locally finitely presented, $S$-flat $\sO_{\sX}$-module $\Omega_{\sX/S}$, the \emph{``relative dualizing sheaf''} of $\sX \ra S$, whose formation is compatible with base change in $S$ (see \cite{Con00}*{Theorem 4.4.4 and bottom half of page~214} for the base change aspect). If $\sX \ra S$ is smooth, then $\Omega_{\sX/S}$ is the top exterior power of $\Omega^1_{\sX/S}$.

The quasi-coherent $\sO_{\sX}$-module $\Omega_{\sX/S}$ has full support and is $S$-fiberwise Cohen--Macaulay: indeed, this reduces to the case when $S$ is the spectrum of a field and $\sX$ is a scheme, and in this case, by \cite{Har66}*{Remark on page~291}, the stalks of $\Omega_{\sX/S}$ are dualizing modules for the corresponding stalks of $\sO_{\sX}$ and hence, by \cite{SP}*{Lemma~\href{https://stacks.math.columbia.edu/tag/0AWS}{0AWS}}, are Cohen--Macaulay of full support. Similarly, by \cite{SP}*{Lemma~\href{https://stacks.math.columbia.edu/tag/0DW9}{0DW9}}, the module $\Omega_{\sX/S}$ is a line bundle if and only if the $S$-fibers of $\sX$ are Gorenstein.

We draw attention to the case when $\sX \ra S$ is proper and $\sX$ is not a scheme, in which we \emph{do not} claim any dualizing properties of the $\sO_\sX$-module $\Omega_{\sX/S}$ constructed above.
\epp

\bpp[The case of modular curves] \lab{modular-case}
For us, the key case is when $S = \Spec \bZ$ and $\sX$ is either the modular stack $\sX_\Gamma$ or its coarse space $X_\Gamma$ for an open subgroup $\Gamma \subset \GL_2(\wh{\bZ})$ (see \S\ref{conv}). The resulting $\sX \ra S$ is flat, of finite presentation, 
with Cohen--Macaulay fibers (the latter by the normality of $\sX$ and \cite{EGAIV2}*{Corollaire 6.3.5~(i)}), so the discussion of \S\ref{rel-dual} applies. Normality of $\sX$ and \cite{EGAIV2}*{Corollaire 6.12.6~(i)} ensure that $\sX^\reg$ is the complement of finitely many closed points of $\sX$, and hence contains $\sX_\bQ$ and is $\bZ$-fiberwise dense in $\sX$. Since $\sX^\reg$ is also $\bZ$-fiberwise Gorenstein (see \cite{Liu02}*{Chapter 6, Example 3.18}), the coherent, $\bZ$-flat, Cohen--Macaulay $\sO_\sX$-module $\Omega_{\sX^\reg/\bZ}$ of \S\ref{rel-dual} is a line bundle. 
In addition, $\Omega_{\sX/\bZ}$ agrees with the line bundle $\Omega^1_{\sU/\bZ}$ over any $\bZ$-smooth open $\sU \subset \sX$, for instance, over $\sX_{\bZ[\frac 1N]} \subset \sX$ for an $N \ge 1$ with $\Gamma(N) \subset \Gamma$ (see \cite{DR73}*{Chapitre IV, Th\'{e}or\`{e}me~6.7} and \cite{modular-description}*{Proposition 6.4~(a)}).
\epp

The key advantages of $\Omega_{\sX/\bZ}$ over the $\sO_\sX$-module $\Omega^1_{\sX/\bZ}$ are its aforementioned pleasant properties at the nonsmooth points. 
The following comparison relates $\Omega_{\sX_\Gamma/\bZ}$ to the more concrete~$\Omega_{X_\Gamma/\bZ}$.

\bprop\label{app-main}
For an open subgroup $\Gamma \subset \GL_2(\wh{\bZ})$, an $N \ge 1$ with $\Gamma(N) \subset \Gamma$, and the coarse space morphism $ \sX_\Gamma \xra{\pi} X_\Gamma$, we have an isomorphism of line bundles
\be \lab{omega-comp}
 \Omega^1_{(X_\Gamma)_{\bZ[\frac{1}{N}]}/\bZ[\frac{1}{N}]} \isomto (\pi_{\bZ[\frac{1}{N}]})_*(\Omega^1_{(\sX_\Gamma)_{\bZ[\frac{1}{N}]}/\bZ[\frac{1}{N}]}) 
\ee
and for any open $U \subset X_\Gamma$ such that $\sU \ce \pi\i(U) \xra{\pi} U$ is \'{e}tale over a $\bZ$-fiberwise dense open of~$U$, 
\[
H^0(U, \Omega) \subset H^0(U_\bQ, \Omega^1) \qq \text{is identified by \eqref{omega-comp} with} \qq H^0(\sU, \Omega) \subset H^0(\sU_\bQ, \Omega^1).
\]
\eprop

\bpf
The second assertion implies the first: indeed, for every open 
$U \subset (X_\Gamma)_{\bZ[\frac{1}{n}]}$, the map $\pi\i(U) \ra U$ is \'{e}tale over the complement of $j = 0$ and $j = 1728$ (see \cite{modular-description}*{last paragraph of the proof of Proposition 6.4}). For the same reason, away from $j = 0$ and $j = 1728$ the pullback map 
\be \label{omega-comp-Q}
\Omega^1_{(X_\Gamma)_\bQ/\bQ} \ra (\pi_\bQ)_*( \Omega^1_{(\sX_\Gamma)_\bQ/\bQ})
\ee
is an isomorphism: there it is the
$\Omega^1_{(X_\Gamma)_\bQ/\bQ}$-twist of the coarse space isomorphism $\sO_{X_\Gamma} \isomto \pi_*(\sO_{\sX_\Gamma})$. To conclude that \eqref{omega-comp-Q} is an isomorphism, we claim that so is its base change to the completion $\wh{\sO}_{(X_\Gamma)_\bQ,\, x}^\sh$ of the strict Henselization of $(X_\Gamma)_\bQ$ at any $x \in X_\Gamma(\ov{\bQ})$. We have 
\[
\wh{\sO}_{(X_\Gamma)_\bQ,\, x}^\sh \simeq \ov{\bQ}\llb t\rrb \qq \text{under which} \qq (\Omega^1_{(X_\Gamma)_\bQ/\bQ})_{\wh{\sO}_{(X_\Gamma)_\bQ,\, x}^\sh} \simeq \ov{\bQ}\llb t\rrb \cdot dt,
\]
and also, using the identification $X_\Gamma(\ov{\bQ}) \cong \sX_\Gamma(\ov{\bQ})$ to view $x$ in $\sX_\Gamma(\ov{\bQ})$,
\[
\wh{\sO}_{(\sX_\Gamma)_\bQ,\, x}^\sh \simeq \ov{\bQ}\llb \tau\rrb \qq \text{under which} \qq (\Omega^1_{(\sX_\Gamma)_\bQ/\bQ})_{\wh{\sO}_{(\sX_\Gamma)_\bQ,\, x}^\sh} \simeq \ov{\bQ}\llb \tau\rrb \cdot d\tau.
\]
Taking into account the action of the automorphism group of $x \in \sX_\Gamma(\ov{\bQ})$, we have, compatibly,
\[
\wh{\sO}_{(X_\Gamma)_\bQ,\, x}^\sh \cong (\wh{\sO}_{(\sX_\Gamma)_\bQ,\, x}^\sh)^G \q \text{and} \q ((\pi_\bQ)_*( \Omega^1_{(\sX_\Gamma)_\bQ/\bQ}))_{\wh{\sO}_{(X_\Gamma)_\bQ,\, x}^\sh} \cong ((\Omega^1_{(\sX_\Gamma)_\bQ/\bQ})_{\wh{\sO}_{(\sX_\Gamma)_\bQ,\, x}^\sh})^G
\]
for some finite group $G$ acting faithfully on $\wh{\sO}_{(\sX_\Gamma)_\bQ,\, x}^\sh$ (see \cite{DR73}*{Chapitre I, Section (8.2.1)} or \cite{Ols06}*{Theorem 2.12}). Since the ramification of $\pi_\bQ$ is tame, the faithfulness of the action implies by Galois theory that $G \simeq \mu_{\#G}(\ov{\bQ})$
 with, at the cost of changing the uniformizer $\tau$ above, $t = \tau^{\#G}$ and $\zeta \in \mu_{\#G}(\ov{\bQ})$ acts by $\tau \mapsto \zeta \cdot \tau$ (see \cite{Ser79}*{Chapter IV, Section 2, Proposition 8}). The desired $\ov{\bQ}\llb t\rrb \cdot dt \isomto (\ov{\bQ}\llb \tau\rrb \cdot d\tau)^G$ follows.


To conclude the sought identification $H^0(U, \Omega) \cong H^0(\sU, \Omega)$, we let $U' \subset U$ with preimage $\sU' \subset \sU$ be a $\bZ$-fiberwise dense open over which $\pi$ is \'{e}tale. The $\sO_{X_\Gamma}$-module $\Omega_{X_\Gamma/\bZ}$ has depth $2$ at the points in $U \setminus (U' \cup U_\bQ)$ (see \S\ref{modular-case}), and similarly for $\Omega_{\sX_\Gamma/\bZ}$, so, by \cite{EGAIV2}*{Th\'{e}or\`{e}me~5.10.5}, we have
\[ \ba
H^0(U, \Omega) = H^0(U', \Omega) \cap H^0(U_\bQ, \Omega^1) \qq &\text{inside} \q H^0(U'_\bQ, \Omega^1), \\
 H^0(\sU, \Omega) = H^0(\sU', \Omega) \cap H^0(\sU_\bQ, \Omega^1) \qq &\text{inside} \q H^0(\sU'_\bQ, \Omega^1).
\ea \]
Therefore, the isomorphism \eqref{omega-comp-Q} reduces us to the case when $U = U'$. Similarly, neither $H^0(U, \Omega)$ nor $H^0(\sU, \Omega)$ changes if we remove finitely many closed points from $U$, so we assume further that $U$ and $\sU$ are regular, so that $\Omega_{U/\bZ}$ and $\Omega_{\sU/\bZ}$ are line bundles (see \S\ref{modular-case}). Then $(\pi|_{\sU})^* (\Omega_{U/\bZ}) \cong \Omega_{\sU/\bZ}$ by the \'{e}taleness of $\sU \ra U$ (see \eqref{et-omega}), to the effect that 
there is a pullback map 
\[
 \Omega_{U/\bZ} \ra (\pi|_\sU)_*(\Omega_{\sU/\bZ}) \qxq{that is the $\Omega_{U/\bZ}$-twist of the isomorphism} \sO_U \isomto (\pi|_\sU)_*(\sO_\sU),
\]
and hence is an isomorphism. The sought identification follows by taking global sections.
\epf

We conclude that the $\bZ$-lattice determined in the $\bQ$-space of cuspforms $H^0(X_0(N)_\bQ, \Omega^1)$ by the relative dualizing sheaf $\Omega$ on the stack $\sX_0(N)$  agrees with its coarse space counterpart as follows.

\bcor \label{X0N-lattice-comp}\lab{AM-X0n}
For an $N \ge 1$ and the map $\sX_0(N) \xra{\pi} X_0(N)$, we have an $\sO_{X_0(N)}$-module isomorphism $\Omega_{X_0(N)/\bZ} \isomto \pi_*(\Omega_{\sX_0(N)/\bZ})$ that over $\bQ$ is the pullback of K\"{a}hler differentials. In particular,
\be \label{X0N-lattice}
H^0(X_0(N), \Omega) = H^0(\sX_0(N), \Omega) \qxq{inside} H^0(X_0(N)_\bQ, \Omega^1) \cong H^0(\sX_0(N)_\bQ, \Omega^1).
\ee
\ecor

\bpf
The map $\pi$ is \'{e}tale (even a $\bZ/2\bZ$-gerbe) over a $\bZ$-fiberwise dense open of $X_0(N)$, for instance, over the complement of $j = 0$ and $j = 1728$, 
see \cite{modular-description}*{proof of Theorem~6.7}. Thus, in the case $\Gamma = \Gamma_0(N)$, \Cref{app-main} applies to every open $U \subset X_0(N)$ and gives the claim.
\epf

Due to the abstract nature of $\Omega$, the lattice $H^0(\sX_0(N), \Omega)$ is \emph{a priori} inexplicit. To remedy this, in particular, to relate this lattice to the integrality properties of Fourier expansions studied in \S\ref{s:global p-adic valuations}, we will use an integral version of the Kodaira--Spencer isomorphism presented in \Cref{integral-KS}.

\bpp[The line bundle $\omega$] \lab{omega}
The cotangent space at the identity section of the universal generalized elliptic curve gives a line bundle $\omega$ on $\sX(1)$, which pulls back to a line bundle $\omega$ on $\sX_\Gamma$ for every open subgroup $\Gamma \subset \GL_2(\wh{\bZ})$. 
We write `$\cusps$' for the reduced complement of the elliptic curve locus of $\sX_\Gamma$, so that  `$\cusps$' restricts to a Weil divisor on the regular locus $\sX_\Gamma^\reg$, which contains $(\sX_\Gamma)_\bQ$ and is $\bZ$-fiberwise dense in $\sX_\Gamma$ (see \S\ref{modular-case}). 
By \cite{Del71}*{Section 2}, for every $k \in \bZ_{> 0}$ and every $\Gamma \subset \Gamma_1(N)$, the space $H^0((\sX_\Gamma)_\bC, \omega^{\tensor k})$ (resp.,~$H^0((\sX_\Gamma)_\bC, \omega^{\tensor k}(-\cusps))$) is canonically identified  with the $\bC$-vector space of  modular forms (resp.,~cuspforms) of weight $k$ on $\Gamma$ reviewed in \S\ref{pp:q-expansions}, so $H^0(\sX_\Gamma, \omega^{\tensor k})$ (resp.,~$H^0(\sX_\Gamma, \omega^{\tensor k}(-\cusps))$ if $\sX_\Gamma$ is regular) is a $\bZ$-lattice in this $\bC$-vector space.

Thanks to this algebraic description, one enlarges the scope of the definitions: in the rest of this article, by a \emph{modular form} (resp.,~\emph{cuspform}) of weight $k$ on $\Gamma$ over a scheme $S$ we mean an element of $H^0((\sX_\Gamma)_S, \omega^{\tensor k})$ (resp.,~$H^0((\sX_\Gamma)_S, \omega^{\tensor k}(-\cusps))$; we will use the latter only when $\sX_\Gamma$ is regular).
\epp

\bprop \label{integral-KS}
For an open subgroup $\Gamma \subset \GL_2(\wh{\bZ})$, 
letting $y$ range over the generic points of the $\bF_p$-fibers of $\sX_\Gamma$ for the 
set of primes $p$ that divide every \up{equivalently, the smallest} $N \ge 1$ with $\Gamma(N) \subset \Gamma$, and letting $d_y$ denote the valuation of the different ideal of the extension $\sO_{\sX_\Gamma,\, \ov{y}}^\sh/\sO_{\sX(1),\, \ov{y}}^\sh$ of discrete valuation rings \up{see \uS\uref{conv} for the notation}, we have
\[
\tst \Omega_{\sX^\reg_\Gamma/\bZ} \cong \omega_{\sX^\reg_\Gamma}^{\otimes 2}(-\cusps + \sum_y d_y  \ov{\{ y \}}). 
\]
\eprop



\bpf
It is indeed equivalent to consider the smallest $N$ with $\Gamma(N) \subset \Gamma$: if $p \mid N$ but $p \nmid M$ for some $\Gamma(M) \subset \Gamma$, then, for $N' \ce \frac{N}{p^{\val_p(N)}}$, every element of $\Gamma(N')$ is congruent modulo $p^{\val_p(N)}$ to an element of $\Gamma(M)$, so $\Gamma(N') \subset \Gamma(MN')\Gamma(N) \subset \Gamma$, contradicting the minimality of $N$.

For the main assertion, since 
both sides are 
line bundles (see \S\ref{rel-dual}) and $\sX_\Gamma$ is normal, 
by \cite{EGAIV2}*{Th\'{e}or\`{e}me 5.10.5}, it suffices to exhibit the desired isomorphism over the slightly smaller open $\sU \subset \sX^\reg_\Gamma$ that is the preimage of the open of $\sX(1)$ obtained by removing the images of the singular points of $\sX_\Gamma$. We will bootstrap the claim from its case for $\sX(1)$ supplied by \cite{Kat73}*{Section~(A1.3.17)}:
\be \lab{Katz-input}
\Omega^1_{\sX(1)/\bZ} \cong \omega^{\tensor 2}_{\sX(1)}(-\cusps).
\ee
By working \'{e}tale locally on $\sX(1)$ and using \cite{Con00}*{Theorem 4.3.3, equation (4.3.7), bottom of page~206}, we get 
\be \lab{duality-at-its-best}
\Omega_{\sU/\bZ} \cong \Omega_{\sU/\sX(1)} \otimes_{\sO_{\sU}} \pi^* \Omega_{\sX(1)/\bZ},
\ee
where $\pi\colon \sU \ra \sX(1)$ is the forgetful map. Since $\pi$ is finite locally free over $\pi(\sU)$, by 
\cite{Con00}*{bottom half of page~31 and pages~137--139, especially, compatibility (VAR6) on page~139}, 
the $\sO_{\sU}$-module $\Omega_{\sU/\sX(1)}$ reviewed in \S\ref{rel-dual}, is identified with $\sH om_{\sO_{\pi(\sU)}}(\pi_*(\sO_{\sU}), \sO_{\pi(\sU)})$. 
Thus, since $\pi$ is generically \'{e}tale, 
the element
\[
\mathrm{trace} \in \Hom_{\sO_{\pi(\sU)}}(\pi_*(\sO_{\sU}), \sO_{\pi(\sU)}) \cong \Gamma(\sU, \Omega_{\sU/\sX(1)}),
\]
via the correspondence \cite{SP}*{Lemma \href{https://stacks.math.columbia.edu/tag/01X0}{01X0}} (with \cite{SP}*{Lemma \href{https://stacks.math.columbia.edu/tag/0AG0}{0AG0}}), gives rise to the identification
\be \lab{omega-kick}
\tst \Omega_{\sU/\sX(1)}\cong \sO_{\sU}(\sum_{x \in \abs{\sX_\Gamma}^{(1)}} d_x  \ov{\{x\}}),
\ee
where the sum is over the height $1$ points $x$ of $\sX_\Gamma$ and $d_x$ is the order of vanishing of `$\mathrm{trace}$' at $\sO_{\sX_\Gamma,\, \ov{x}}^\sh$. 
By considering the fractional multiples of `$\mathrm{trace}$' that still map $\sO_{\sX_\Gamma,\, \ov{x}}^\sh$ into $\sO_{\sX(1),\, \ov{x}}^\sh$, we see that $d_x$ is the valuation of the different ideal of $\sO_{\sX_\Gamma,\, \ov{x}}^\sh/\sO_{\sX(1),\, \ov{x}}^\sh$ (see \cite{Ser79}*{Chapter III, Section 3}). Thus, $d_x = 0$ whenever this extension is \'{e}tale, so each $x$ that contributes to the sum either lies on the cusps of $(\sX_\Gamma)_\bQ$ or is the generic point of an irreducible component of an $\bF_p$-fiber of $\sX_\Gamma \ra \Spec \bZ$ such that $p \mid N$ for every $\Gamma(N) \subset \Gamma$ (see \cite{DR73}*{Chapitre IV, D\'{e}finition 3.2}). At the former, ramification is tame and $d_x = e_x - 1$, where $e_x$ is the ramification index of $\sO_{\sX_\Gamma,\, \ov{x}}^\sh/\sO_{\sX(1),\, \ov{x}}^\sh$ (see \cite{Ser79}*{Chapter III, Section 6, Proposition 13}). Thus, since 
\[
\tst \pi^*(\omega^{\tensor 2}_{\sX(1)}(-\cusps))\cong \omega_{\sU}^{\otimes 2}(-\sum_{x\in\cusps}e_x \overline{\{x\}}),
\]
by \eqref{Katz-input}--\eqref{omega-kick}  we obtain the desired
\[
\tst \Omega_{\sU/\bZ} \cong \omega_{\sU}^{\otimes 2}(-\cusps + \sum_y d_y  \ov{\{ y \}}). \qedhere
\]
\epf

\begin{variant} \label{coarse-KS}
For an open subgroup $\Gamma \subset \GL_2(\wh{\bZ})$ and the forgetful map $\pi \colon X_\Gamma \ra X(1)$, 
letting $y$ range over the height $1$ points of $X_\Gamma$ and letting $d'_y$ denote the valuation of the different ideal of the extension $\sO_{X_\Gamma,\, y}/\sO_{X(1),\, \pi(y)}$ of discrete valuation rings,
we have
\[
\tst \Omega_{X^\reg_\Gamma/\bZ} \cong (\pi^*\Omega^1_{X(1)/\bZ})|_{X_\Gamma^\reg}(\sum_{y \in X_\Gamma^{(1)}} d'_y  \ov{\{ y \}}). 
\]
\end{variant}

\bpf
The proof is the same (but simpler) as that of \Cref{integral-KS}. 
Namely, $X(1) \cong \bP^1_\bZ$ is $\bZ$-smooth, so $\Omega_{X(1)/\bZ} \cong \Omega^1_{X(1)/\bZ}$, and, similarly to there, one may restrict to the preimage $U \subset X_\Gamma$ of $X(1) \setminus \pi(X_\Gamma \setminus X_\Gamma^\reg)$ and then conclude by using the analogues of \eqref{duality-at-its-best} and \eqref{omega-kick}.
\epf


For general $\Gamma$, it is tricky to directly compute the $d_y$ that appear in the integral Kodaira--Spencer formula of \Cref{integral-KS} because the extension $\sO_{\sX_\Gamma,\, \ov{y}}^\sh/\sO_{\sX(1),\, \ov{y}}^\sh$ involves imperfect residue fields and may be wildly ramified. For $\Gamma_0(N)$, we will compute the $d_y$  in \Cref{compute-dy}, and for this  we first argue that only the level at $p$  matters and then describe $\sX_0(p^{\val_p(N)})$ along the cusps.


\blem \label{indep-of-more}
For open subgroups $\Gamma, \Gamma' \subset \GL_2(\wh{\bZ})$ with $\Gamma(N) \subset \Gamma$ and $\Gamma(N') \subset \Gamma'$, a generic point $y_{\Gamma \cap \Gamma'}$ of the $\bF_p$-fiber of $\sX_{\Gamma \cap \Gamma'}$ with $p \nmid N'$, and its image $y_\Gamma$ in $\sX_\Gamma$, in Proposition \uref{integral-KS} we have
\[
d_{y_{\Gamma \cap \Gamma'}} = d_{y_\Gamma}.
\]
\elem

\bpf
By \cite{DR73}*{Chapitre IV, Construction 3.8, Proposition 3.9}, the stack $\sX_{\Gamma \cap \Gamma'}$ agrees with the normalization\footnote{Note that for \cite{DR73}*{Chapitre IV, \'{e}quation (3.9.1)} to hold, one needs to take the normalization of its left side, see \cite{modular-description}*{Example 4.5.3}.} of $\sX_\Gamma \times_{\sX(1)} \sX_{\Gamma'}$. Thus, since the assumption $p \nmid N'$ ensures that the map $\sX_{\Gamma'} \ra \sX(1)$ is \'{e}tale at the image of $y_{\Gamma \cap \Gamma'}$ (see \cite{DR73}*{Chapitre IV, D\'{e}finition 3.2 onwards}), the map $\sX_{\Gamma \cap \Gamma'} \ra \sX_\Gamma$ is \'{e}tale at $y_{\Gamma \cap \Gamma'}$. In particular, letting $\ov{y}$ be a geometric point above $y_{\Gamma \cap \Gamma'}$, we have $\sO_{\sX_{\Gamma \cap \Gamma'},\, \ov{y}}^\sh \isomto \sO_{\sX_\Gamma,\, \ov{y}}^\sh$, so that 
$d_{y_{\Gamma \cap \Gamma'}} = d_{y_\Gamma}$, as desired.
\epf

\bpp[The components of $\sX_0(N)_{\bF_p}$] \label{X0N-comps}
We recall from \cite{KM85}*{Theorem 13.4.7} that the irreducible components of $\sX_0(N)_{\bF_p}$ correspond to pairs $(a, b)$ of integers $a, b \ge 0$ with $a + b = \val_p(N)$ in such a way that on the $(a, b)$-component the $p$-primary part of the cyclic subgroup that is part of the modular interpretation of $\sX_0(N)$ is generically an extension of an \'{e}tale group of order $p^b$ by the $a$-fold relative Frobenius kernel.  The ramification index $e_{(a,\, b)}$ of the strict Henselization of $\sX_0(N)$ at the generic point of the $(a, b)$-component of $\sX_0(N)_{\bF_p}$ was determined in \cite{KM85}*{Section (13.5.6)}:
\be \label{e-formula}
e_{(a,\, b)} = \phi(p^{\min(a,\, b)}).
\ee
If $p \mid N$, then the forgetful map $\sX_0(N)_{\bF_p} \ra \sX_0(\frac Np)_{\bF_p}$ sends each $(a, b)$-component with $b > 0$ to the $(a, b - 1)$-component and the $(a, 0)$-component to the $(a - 1, 0)$-component.
\epp


\blem \label{describe-cusps}
For a prime $p$ and an $n \ge 0$, the base change of the forgetful map $\sX_0(p^n) \ra \sX(1)$ along the map $\Spec(\bZ\llb q \rrb) \ra \sX(1)$ given by the Tate generalized elliptic curve over $\bZ\llb q \rrb$ is
\begin{equation*}
\tst \sX_0(p^n) \times_{\sX(1)} \bZ\llb q \rrb \cong  \bigsqcup_{\substack{a + b = n \\ a \ge b \ge 0}} \Spec(\bZ[\zeta_{p^b}]\llb q \rrb) \sqcup \bigsqcup_{\substack{a + b = n \\ 0 \le a < b}} \Spec((\bZ[\zeta_{p^a}]\llb q \rrb)[X]/(X^{p^{b - a}} - \zeta_{p^a}q)),
\end{equation*}
where, without explicating the $\bZ\llb q\rrb$-algebra structure, the last term is $\bZ[\zeta_{p^a}]\llb X \rrb$. After base change to $\bF_p$, the term indexed by $(a, b)$ in this decomposition maps to the $(a, b)$-component of $\sX_0(p^n)_{\bF_p}$. 
\elem

\bpf
By \cite{DR73}*{Chapitre VII, Corollaire 2.2}, the finite, flat $\bZ\llb q\rrb$-scheme $\sX_0(p^n) \times_{\sX(1)} \bZ\llb q \rrb$ is the normalization of $\bZ\llb q\rrb$ in the finite $\bZ\llp q \rrp$-scheme $\sX_0(p^n) \times_{\sX(1)} \bZ\llp q \rrp$. The latter parametrizes cyclic (in the sense of Drinfeld) subgroups of order $p^n$ of the Tate elliptic curve over $\bZ\llp q \rrp$, so, by \cite{KM85}*{Theorem 13.6.6}, it is
\[
\tst 
\Spec(\bZ\llp q \rrp) \sqcup \Spec(\bZ\llp q^{\frac 1{p^n}} \rrp) \sqcup \bigsqcup_{\substack{a + b = n \\ a,\, b > 0}} \Spec\ppp{\bZ\llp q \rrp[X]/(\Phi_p(\frac{X^{p^{b - 1}}}{q^{p^{a - 1}}}))}
\]
where $\Phi_p(Z) \ce Z^{p - 1} + \dotsc + Z + 1$ is the $p$-th cyclotomic polynomial. More explicitly, if $a \ge b \ge 1$, then $X/q^{p^{a - b}}$ is a $p^b$-th root of unity in the source of the surjection
\[
\tst \bZ\llp q \rrp[X]/(\Phi_p(\frac{X^{p^{b - 1}}}{q^{p^{a - 1}}})) \ra   \bZ[\zeta_{p^b}]\llp q \rrp \qxq{given by} X \mapsto \zeta_{p^b}\, q^{p^{a - b}}
\]
that must also be injective because its source and target are free $\bZ\llp q \rrp$-modules of rank $p^{b - 1}(p - 1)$. Similarly, if $1 \le a \le b$, then $X^{p^{ b- a}}/q$ is a $p^a$-th root of unity in the source of the isomorphism
\[
\tst \bZ\llp q \rrp[X]/(\Phi_p(\frac{X^{p^{b - 1}}}{q^{p^{a - 1}}})) \xra{\sim}   (\bZ[\zeta_{p^a}]\llp q \rrp)[X]/(X^{p^{b - a}} - \zeta_{p^a}q).
\]
To conclude the claimed description of $\sX_0(p^n) \times_{\sX(1)} \bZ\llb q \rrb$, it remains to note that both $\bZ[\zeta_{p^b}]\llb q \rrb$ for $a \ge b$ and $(\bZ[\zeta_{p^a}]\llb q \rrb)[X]/(X^{p^{b - a}} - \zeta_{p^a}q) \cong \bZ[\zeta_{p^a}]\llb X \rrb$ for $a \le b$ are normal (even regular). The claim about the $(a, b)$-component 
follows from \cite{KM85}*{Proposition 13.6.2 and the proof of Theorem~13.6.6}.
\epf

Before proceeding to the promised formula for the $d_y$ in \Cref{compute-dy}, we record the following consequence of \Cref{describe-cusps} that relates the present section to the analytic considerations of~\S\ref{s:global p-adic valuations}.


\blem \label{reduce-correctly}
For $L \mid N$ and a prime $p$, every cusp of $\sX_0(N)_\bC$ of denominator $L$ \up{see \uS\uref{pp:cusps} and use $\sX_0(N)(\bC) = X_0(N)(\bC)$} reduces to the $(\val_p(L), \val_p(\frac NL))$-component of $\sX_0(N)_{\bF_p}$ \up{see \uS\uref{X0N-comps}}.
\elem

\bpf
Points of $\sX_0(N)$ and of its coarse space $X_0(N)$ valued in algebraically closed fields agree and every cusp is a $\ov{\bQ}$-point, 
so the statement makes sense. Moreover, the complex uniformizations \eqref{uniformization} are compatible with forgetting some of the level, so we may assume that $N = p^n$. 
For $L = N$, the only cusp of $X_0(N)$ of denominator $L$ is $\infty$ and its punctured analytic neighborhood parametrizes pairs $(\bC^\times/q^\bZ, \langle e^{2\pi i/N} \rangle)$ with $q = e^{2\pi i z}$ and $\im z \gg 0$ (see \cite{Roh97}*{Section 1.10, Proposition 7}). Thus, by the algebraic theory of the Tate curve with its canonical subgroup $\mu_N$ (see \cite{DR73}*{Chapitre~VII, Section 1, especially, \'{e}quation (1.12.3)}), this cusp factors through the $(n, 0)$-term of the right side decomposition of \Cref{describe-cusps}, and hence reduces to the $(n, 0)$-component. For the other cusps, we induct on $n$, so we suppose that $n > 0$ and consider a cusp $\fc$ of denominator $p^\ell$ with $\ell \le n - 1$. By induction, the image of $\fc$ reduces to the $(\ell, n - \ell - 1)$-component of $\sX_0(p^{n - 1})_{\bF_p}$. Thus, if $n - \ell - 1 > 0$, then $\fc$ must reduce to the $(\ell, n - \ell)$-component of $\sX_0(p^{n})_{\bF_p}$ (see \S\ref{X0N-comps}). To bootstrap the remaining $\phi(p^{\min(n - 1,\,  1)})$ cusps with $\ell = n - 1$ (see \S\ref{pp:cusps}), it now remains to note that, by \Cref{describe-cusps}, there are precisely $\phi(p^{\min(n - 1,\,  1)})$ cusps that reduce to the $(n - 1, 1)$-component of $\sX_0(p^n)_{\bF_p}$.
\epf

\bprop \label{compute-dy}
For an $N \ge 1$, a prime $p$, 
 the generic point $y$ of the $(a, b)$-component of $\sX_0(N)_{\bF_p}$ \up{see \uS\uref{X0N-comps}}, and the valuation $d_{(a, b)}$ of the different of the extension $\sO_{\sX_0(N),\, \ov{y}}^\sh/\sO_{\sX(1),\, \ov{y}}^\sh$, 
\be \lab{d-formula}\ba
d_{(a,\,b)} = \begin{cases}b  & \x{if $a = 0$,} \\ p^{\min(a,\, b) - 1}(pb - b - 1) &\x{if $a, b \ge 1$,}  \\ 0 &\x{if $b = 0$.}\end{cases}
\ea\ee
\eprop

\bpf
By \Cref{indep-of-more} and \S\ref{X0N-comps}, we may forget level away from $p$ to assume that $N = p^n$.  As in the proof of \Cref{integral-KS}, the different of a finite, generically separable extension $R'/R$ of discrete valuation rings is the annihilator of the $R'$-module $\Hom_R(R', R)/(\mathrm{trace}_{R'/R})$. The formation of this annihilator commutes with flat base change in $R$ (after which $R$ and $R'$ may cease being discrete valuation rings). We will apply this to  $\sO^\sh_{\sX_0(N),\, \ov{y}}/\sO^\sh_{\sX(1),\, \ov{y}}$, the valuation $d_{(a,\, b)}$ of whose different we wish to compute. Namely, by \cite{DR73}*{Chapitre VII, Th\'{e}or\`{e}me 2.1}, the map $\Spec(\bZ\llb q \rrb) \ra \sX(1)$ given by the Tate generalized elliptic curve over $\bZ\llb q \rrb$ realizes its source as an \'{e}tale double cover of the formal completion of $\sX(1)$ along the cusps, and the flat base change map we will use is the resulting $\sO^\sh_{\sX(1),\, \ov{y}} \ra \bZ\llb q \rrb^\sh_{(p)}$, where the latter strict Henselization is at the generic point of the $\bF_p$-fiber of $\bZ\llb q \rrb$. In this notation, by \Cref{describe-cusps}, the resulting base change of $\sO^\sh_{\sX_0(N),\, \ov{y}}$ is
\[
\bZ[\zeta_{p^b}]\llb q \rrb_{(p)}^\sh \qxq{if} a \ge b, \qxq{and} ((\bZ[\zeta_{p^a}]\llb q \rrb)[X]/(X^{p^{b - a}} - \zeta_{p^a}q))_{(p)}^\sh \qxq{if} a \le b.
\]
These are discrete valuation rings, and the extension $\bZ[\zeta_{p^b}]\llb q \rrb_{(p)}^\sh/\bZ\llb q \rrb^\sh_{(p)}$ is a flat base change of $\bZ[\zeta_{p^b}]_{(p)}^\sh/\bZ_{(p)}^\sh$. Thus, the $a \ge b$ case of \eqref{d-formula} follows from the ramification theory of cyclotomic fields \cite{Was97}*{Proposition 2.1}. To similarly treat the $a \le b$ case, we will use subextension
\be \label{eqn:subexts}
\bZ\llb q \rrb^\sh_{(p)} \subset \bZ[\zeta_{p^a}]\llb q \rrb_{(p)}^\sh \subset ((\bZ[\zeta_{p^a}]\llb q \rrb)[X]/(X^{p^{b - a}} - \zeta_{p^a}q))_{(p)}^\sh
\ee
and the tower formula for the different \cite{Ser79}*{Chapter III, Section 4, Proposition 8} (that, notably, does not require  residue field extensions to be separable---an assumption not met here). Namely, letting $\wt{d}_{(a,\, b)}$ be the valuation of the different of the top extension, \cite{Was97}*{Proposition 2.1} now gives
\[
d_{(a,\,b)} = \wt{d}_{(a,\,b)} + \begin{cases}p^{a - 1}(pa - a - 1) &\x{if $a \ge 1$,}  \\ 0 & \x{if $a = 0$.}\end{cases}
\]
To compute $\wt{d}_{(a,\,b)}$, we note that the top subextension in \eqref{eqn:subexts} is of degree $p^{b - a}$, does not change the uniformizer $1 - \zeta_{p^a}$, induces a purely inseparable residue field extension of degree $p^{b - a}$, and, as a module, is generated by powers of $X$. Thus, since $X, X^2, \dotsc, X^{p^{b - a} - 1}$ have trace $0$ in this extension, we conclude that $d = (b - a) \phi(p^a)$. The desired formula in the remaining case $a \le b$ follows.
\epf

With the integral version of the Kodaira--Spencer isomorphism (\Cref{integral-KS}) and the explicit formulas for the $d_y$ (\Cref{compute-dy}) in hand, we are ready to characterize the $\bZ$-lattice $H^0(\sX_0(N), \Omega)$ in terms of the $p$-adic properties of Fourier expansions at all cusps in \Cref{explicit-crit}.

\blem \label{indep-of-c}
For a prime $p$, an 
$f \in H^0(\sX_0(N)_{\ov{\bQ}_p}, \omega^{\tensor k})$ with $k \ge 1$, a cusp $\fc \in X_0(N)(\ov{\bQ}_p)$ of denominator $L$, and an isomorphism $\iota\colon \ov{\bQ}_p \simeq \bC$, the valuation $v \ce \val_p(\iota(f)|_{\iota(\fc)})$ defined as in \eqref{def-val-cusp} \up{see also \uS\uref{omega}} after pullback\footnote{The only role of the auxiliary level is to ensure that $\sX(N\wt N)_\bC$ is a scheme and hence admits a complex uniformization analogous to the one discussed in \eqref{uniformization}.} to a cusp $\wt{\fc} \in X(N\wt{N})(\bC)$ above $\fc$ for a sufficiently divisible $\wt{N}$ depends only on $f$ and $\val_p(L)$ 
\up{and not on $\fc$, $\iota$, $\wt N $, or $\wt{\fc}$}\ucolon letting $\sU \subset \sX_0(N)_{\bZ_p}$ denote the open complement of those irreducible components of $\sX_0(N)_{\bF_p}$ that do not meet the reduction of $\fc$,
\be \label{char-v}
v \qxq{is the largest rational number such that}  p^{-v}f \in H^0(\sU_{\ov{\bZ}_p}, \omega^{\tensor k}).
\ee
\elem

\bpf
By \Cref{reduce-correctly}, the irreducible component of $\sX_0(N)_{\bF_p}$ that contains the reduction of $\fc$ depends only on $\val_p(L)$, so the same holds for $\sU$  and it suffices to establish \eqref{char-v}. Moreover, by scaling $f$, we may assume that $v = 0$.
 By the normality of $\sX_0(N)$, the forgetful map
\[
\pi \colon \sX(N\wt N) \ra \sX_0(N) \qxq{satisfies} \sO_{\sX_0(N)} \isomto (\pi_*(\sO_{\sX(N\wt N)}))^{\Gamma_0(N)/\Gamma(N\wt N)}
\]
and this persists after flat base change, such as to $\ov{\bZ}_p$. Thus, $\Gamma_0(N)/\Gamma(N\wt{N})$ acts transitively on the cusps $\wt{\fc} \in X(N\wt N)(\bC)$ above $\fc$ and, letting  
$\wt{\sU} \subset \sX(N\wt N)_{\bZ_p}$ be the complement of those irreducible components of $\sX(N\wt N)_{\bF_p}$ that do not meet the reduction of a fixed $\wt{\fc}$, we reduce to showing that
\be \label{eqn:no-vprime}
\xq{no} v' \in \bQ_{> 0} \qxq{satisfies} p^{-v'}f|_{\sX(N\wt N)_{\ov \bQ_p}} \in H^0(\wt{\sU}_{\ov{\bZ}_p}, \omega^{\tensor k}).
\ee
In addition, limit arguments eliminate the artificial non-Noetherian aspects: they allow us to replace $\ov{\bQ}_p$ and $\ov \bZ_p$ by a variable sufficiently large finite extension $F/\bQ_p$ and its ring of integers $\cO_F$.

For sufficiently divisible $\wt{N}$, the stack $\sX(N\wt{N})$ is a scheme (already $15 \mid \wt N$ suffices, see \cite{KM85}*{Corollary 2.7.2}) and, by \cite{KM85}*{Theorem 10.9.1}, the formal completion of $\sX(N\wt{N})_{\cO_F}$ along the closure of $\wt{\fc}$ is $\cO_F\llb q^{\frac 1{N\wt{N}}}\rrb$. Under a trivialization of the pullback of $\omega^{\tensor k}$ to this formal completion, the pullback of $f$ is described by its $q$-expansion, which is an element of $F\llb q^{\frac 1{N\wt{N}}}\rrb$ that, via $\iota$, agrees with the analytic Fourier expansion of $f$ at $\wt{\fc}$ constructed as in \S\ref{pp:q-expansions} (see \cite{DR73}*{Chapitre VII, Section~4.8}). Consequently, $\varpi_F^a f$ with $a \in \bZ$ extends to a section of $\omega^{\tensor k}$ over a neighborhood of the closure of $\wt \fc$ in $\sX(N\wt{N})_{\cO_F}$ if and only if $\frac a{e_F} \ge 0$, where $e_F$ is the absolute ramification index of $F$. The complement in $\wt{\sU}_{\cO_F}$ of the union of such a neighborhood with $\sX(N\wt N)_F$ is of codimension $\ge 2$, so, since $\sX(N\wt{N})_{\cO_F}$ is Cohen--Macaulay, \cite{EGAIV2}*{Th\'{e}or\`{e}me 5.10.5} ensures that $\varpi_F^a f$ extends to a neighborhood of the closure of $\wt{\fc}$ in $\sX(N\wt N)_{\cO_F}$ if and only if $\varpi_F^a f \in H^0(\wt{\sU}_{\cO_F}, \omega^{\tensor k})$. As $F$ grows, this achieves the promised \eqref{eqn:no-vprime}.
\epf

\bprop \label{explicit-crit}
For a prime $p$ and a cuspform $f \in H^0(\sX_0(N)_{\ov \bQ_p}, \omega^{\tensor 2}(-\cusps))$, 
the differential $\omega_f \in H^0(X_0(N)_{\ov{\bQ}_p}, \Omega^1)$ lies in the $\ov \bZ_{p}$-lattice $H^0(X_0(N)_{\ov\bZ_{p}}, \Omega) \cong H^0(\sX_0(N)_{\ov\bZ_{p}}, \Omega)$ \up{see \eqref{X0N-lattice}} if and only if for every $0 \le \ell \le \val_p(N)$ and some \up{equivalently, any} cusp $\fc \in X_0(N)(\ov{\bQ}_p)$ whose denominator $L$ satisfies $\ell = \val_p(L)$ and some \up{equivalently, any} isomorphism $\iota\colon \ov{\bQ}_p \simeq \bC$,
we have
\be \label{test-ineq}
\val_p(\iota(f)|_{\iota(\fc)}) \ge \begin{cases}
-\val_p(N)  & \x{if }~\val_p(L) = 0, \\
-\val_p(\frac NL) + \frac{1}{p  - 1} &\x{if }~0 < \val_p(L) < \val_p(N), \\
0 &\x{if }~\val_p(L) = \val_p(N).\end{cases}
\ee
For such an $f$ defined over a number field $K$ with ring of integers $\cO_K$, 
we have $\omega_f \in H^0(X_0(N)_{\cO_K}, \Omega)$ 
if and only if \eqref{test-ineq} holds for all primes $p$ and all embeddings $K \hra \ov{\bQ}_p$. 
\eprop

\bpf
The last assertion follows from the rest because any finite free $\cO_K$-module $M$ (such as $H^0(X_0(N)_{\cO_K}, \Omega) \cong H^0(X_0(N), \Omega) \tensor_{\bZ} \cO_K$, see \S\ref{rel-dual}) agrees with the set of $m \in M \tensor_{\cO_K} K$ whose image in $M \tensor_{\cO_K} \ov \bQ_p$ lies in $M \tensor_{\cO_K} \ov \bZ_p$ for every prime $p$ and every embedding $K \hra \ov{\bQ}_p$.
For \eqref{test-ineq} itself, we begin by recalling the integral Kodaira--Spencer isomorphism of Propositions \ref{integral-KS} and \ref{compute-dy}: letting $y$ range over the generic points of the irreducible components of $\sX_0(N)_{\bF_p}$, with $d_y$ as there,
\[
\tst  \Omega_{\sX_0(N)_{\bZ_p}/\bZ_p} \cong \omega^{\tensor 2}(-\cusps + \sum_{y} d_y\ov{\{y\}}).
\]
Consequently, the characterization of $\val_p(f|_\fc)$ given in \Cref{indep-of-c} together with \cite{EGAIV2}*{Th\'{e}or\`{e}me 5.10.5} (applied as in the preceding proof) show that $\omega_f \in H^0(\sX_0(N)_{\ov\bZ_{p}}, \Omega)$ if and only if for every $y$ and some cusp $\fc$ that reduces modulo $p$ on $\ov{\{ y\}}$, we have $d_y/e_y \ge -\val_p(\iota(f)|_{\iota(\fc)})$ where $e_y$ is the absolute ramification index of the discrete valuation ring $\sO^\sh_{\sX_0(N),\, y}$. By \Cref{reduce-correctly}, a cusp $\fc$ of denominator $L$ reduces to the $(\val_p(L), \val_p(\frac NL))$-component of $\sX_0(N)_{\bF_p}$ for which, by \eqref{e-formula}, the corresponding $e_y$ is $\phi(p^{\min(\val_p(L),\, \val_p(\frac NL))})$. To arrive at
\eqref{test-ineq}, it then remains to use \eqref{d-formula}.
\epf

We are ready for our main integrality result for normalized newforms. 

\bthm \label{wf-integral}
For a number field $K$ and an $f \in H^0(\sX_0(N)_K, \omega^{\tensor 2}(-\cusps))$ whose base change along some $K \hra \bC$ is a $\ov{\bZ}$-linear combination of normalized newforms on $\Gamma_0(N)$ \up{see \uS\uref{omega}},
\[
 \omega_f \in H^0(X_0(N)_{\cO_K}, \Omega) \cong H^0(\sX_0(N)_{\cO_K}, \Omega) \qxq{inside} H^0(X_0(N)_K, \Omega^1) \cong H^0(\sX_0(N)_K, \Omega^1)
\]
\up{identification by flat base change and \eqref{X0N-lattice}}, and, more generally, 
for any $\Gamma_1(N) \subset \Gamma \subset \Gamma_0(N)$,
\[
\omega_f \in H^0((X_\Gamma)_{\cO_K}, \Omega) \subset H^0((X_\Gamma)_K, \Omega^1) \qxq{and} \omega_f \in H^0((\sX_\Gamma)_{\cO_K}, \Omega) \subset H^0((\sX_\Gamma)_K, \Omega^1).
\]
\ethm

\bpf
A Galois conjugate of a newform is still a newform (see \cite{DI95}*{Corollary 12.4.5}), so the assumption on $f$ does not depend on the choice of an embedding $K \hra \bC$. For the first assertion, by \Cref{explicit-crit}, we need to check that for every prime $p$, every embedding $\gL\colon K \hra \ov{\bQ}_p$, every $0 \le \ell \le \val_p(N)$, some cusp $
\fc \in X_0(N)(\bC)$ whose denominator $L$ satisfies $\val_p(L) = \ell$, and some isomorphism $\iota \colon \ov{\bQ}_p \simeq \bC$, the valuation $\val_p(\iota(\lambda(f))|_{\fc})$ satisfies the bound \eqref{test-ineq}. This, however, follows from \Cref{cor:globalbdfcelliptic}.


To deduce that $\omega_f \in H^0((X_\Gamma)_{\cO_K}, \Omega)$ for an arbitrary $\Gamma$, 
since $\Omega_{(X_\Gamma)_{\cO_K}/\cO_K}$ is a Cohen--Macaulay $\sO_{(X_\Gamma)_{\cO_K}}$-module of full support (see \S\ref{rel-dual}), by \cite{EGAIV2}*{Th\'{e}or\`{e}me 5.10.5}, it suffices to show the containment $\omega_f \in H^0((X_\Gamma^\reg)_{\cO_K}, \Omega)$. Thus, \Cref{coarse-KS} and the settled case $\Gamma = \Gamma_0(N)$ reduce us to showing that for every height $1$ point $y \in X_\Gamma$ with images $y' \in X_0(N)$ and $y'' \in X(1)$, the extensions
\[
\sO_{X(1),\, y''} \subset \sO_{X_0(N),\, y'} \subset \sO_{X_\Gamma,\, y} \qxq{of discrete valuation rings satisfy} d_{y/y''} \ge e_{y/y'}d_{y'/y''}
\]
where $d_*$ (resp.,~$e_*$) is the valuation of the different (resp.,~the ramification index) of the indicated subextension. This inequality is immediate from the tower formula for the different \cite{Ser79}*{Chapter III, Section 4, Proposition 8}. To likewise deduce that also $\omega_f \in H^0((\sX_\Gamma)_{\cO_K}, \Omega)$, one uses \Cref{integral-KS} instead.
\epf

\brem \label{wf-primitive}
For a \emph{normalized} cuspform $f$ of weight $2$ on $\Gamma_0(N)$, if $\omega_f$ lies in $H^0(X_0(N), \Omega)$, then it is a primitive \up{that is, not divisible by any $m > 1$} element of this $\bZ$-lattice. In fact, 
then it is primitive even in the $\bZ$-lattice $H^0(X_\Gamma^\sm, \Omega^1)$ for every $\Gamma_1(N) \subset \Gamma \subset \Gamma_0(N)$. Indeed, the finite maps $X_1(N) \ra X_\Gamma \ra X_0(N)$ are flat away from finitely many closed points (see \cite{EGAIV2}*{Proposition~6.1.5}), so they restrict to maps $X_1(N)^\sm \ra X_\Gamma^\sm \ra X_0(N)^\sm$ away from these points. By \cite{EGAIV2}*{Th\'{e}or\`{e}me 5.10.5}, removing finitely many closed points has no effect on $H^0((-)^\sm, \Omega^1)$, so we obtain the inclusions
\be \label{wf-primitive-eq}
H^0(X_0(N)^\sm, \Omega^1) \subset H^0(X_\Gamma^\sm, \Omega^1) \subset H^0(X_1(N)^\sm, \Omega^1),
\ee
which 
reduce primitivity 
to the case $\Gamma = \Gamma_1(N)$ settled as in \cite{Ste89}*{proof of Theorem 1.6} via $q$-expansions.
\erem


\section{Rational singularities of $X_0(N)$} \lab{sec:rat-sing}

For studying the Manin constant, the $\bZ$-lattice $H^0(\cJ_0(N), \Omega^1)$ given by the global differentials on the N\'{e}ron model $\cJ_0(N)$ of the modular Jacobian $J_0(N) \ce \Pic^0_{X_0(N)_\bQ/\bQ}$ is more convenient than the \emph{a priori} larger $H^0(X_0(N), \Omega)$  because it is functorial with respect to both a modular parametrization $J_0(N) \surjects E$ and its dual $E \ra J_0(N)$. Thanks to this functoriality, the Manin conjecture implies that the differential $\omega_f$ associated to the normalized newform $f$ determined by $E$ should lie in $H^0(\cJ_0(N), \Omega^1)$, and we show this unconditionally in \Cref{Neron-integral} whenever $X_0(N)$ has rational singularities. 
We show in \Cref{rat-sing-main} that this assumption holds in a vast number of cases.

\bpp[Rational singularities] \label{pp:rat-sing}
We recall from \cite{Lip69}*{Definition 1.1} that a Noetherian, normal, $2$-dimensional, local domain $R$ has \emph{rational singularities} if $H^1(Z, \sO_Z) = 0$ for some proper, birational morphism $Z \ra \Spec(R)$ with $Z$ regular. In this case, by \cite{Lip69}*{Proposition 1.2}, we have $H^1(Z, \sO_Z) = 0$ for every proper, birational $Z \ra \Spec(R)$ with $Z$ merely normal,  and any such $Z$ also has rational singularities.
\epp


The following result summarizes the relevance of rational singularities for our purposes.

\bprop \label{prop:Raynaud-review}
For an excellent discrete valuation ring $R$ with fraction field $K$ and residue field $k$, a normal, proper, flat relative curve $X$ over $R$ such that $X_{\ov{K}}$ is irreducible and $X^\sm \cap X_{k} \neq \emptyset$, the Jacobian $J \ce \Pic^0_{X_K/K}$, and its N\'{e}ron model $\cJ$ over $R$, the map $\Pic^0_{X/R} \ra \cJ^0$ is an isomorphism if and only if the~inclusion
\be \label{lattice-crit}
H^0(\cJ, \Omega^1) \hra H^0(X, \Omega) \qxq{is an equality inside} H^0(J, \Omega^1) \cong H^0(X_K, \Omega^1),
\ee
which happens if and only if $X$ has rational singularities\uscolon more generally, letting $\pi \colon Z \surjects X$ be a 
proper, birational morphism with $Z$ regular,  $H^0(X, \Omega)/H^0(\cJ, \Omega^1) \simeq H^0(X, R^1\pi_*(\sO_Z))$.
\eprop

\bpf
We have $R \isomto H^0(X, \sO_X)$ because this finite morphism of normal domains (see \cite{SP}*{Lemma~\href{https://stacks.math.columbia.edu/tag/0358}{0358}}) is, by checking over $\ov{K}$, an isomorphism. Thus, since $X^\sm \cap X_k \neq \emptyset$, by \cite{Ray70c}*{Th\'{e}or\`{e}me 8.2.1}, the map $X \ra \Spec R $ is cohomologically flat and $\Pic^0_{X/R}$ is a separated, smooth $R$-group scheme (see also \cite{BLR90}*{Section 8.4, Proposition 2}). In particular, the N\'{e}ron property supplies the map $\Pic^0_{X/R} \ra \cJ$.
Moreover,  
the deformation-theoretic \cite{BLR90}*{Section 8.4, Theorem~1} gives the identification $H^1(X, \sO_{X}) \cong \Lie(\Pic^0_{X/R})$ of finite free $R$-modules. Consequently, by the Grothendieck--Serre duality (see \cite{Con00}*{Theorem 5.1.2}),
\begin{equation} \label{GS-id}
\tst H^0(\Pic^0_{X/R}, \Omega^1) = \Hom_{R}(\Lie(\Pic^0_{X/R}), R) = H^0(X, \Omega)  \qxq{in} H^0(J, \Omega^1)\cong H^0(X_K, \Omega^1).
\end{equation}
Thus, there is the claimed inclusion $H^0(\cJ, \Omega^1) \hra H^0(X, \Omega)$, which, since all the global differentials on $\cJ$ are translation invariant (see \cite{BLR90}*{Section 4.2, Propositions 1 and 2}), is an equality if $\Pic^0_{X/R} \cong \cJ^0$. Conversely, if the inclusion is an equality, then the separated morphism $\Pic^0_{X/R} \ra \cJ^0$ is an isomorphism on Lie algebras, that is, it is \'{e}tale (see \cite{EGAIV4}*{Corollaire 17.11.2}), and hence, by checking the triviality of its kernel over $K$ (see \cite{EGAIV4}*{Th\'{e}or\`{e}me 18.5.11 c)}), even an isomorphism.

By Lipman's \cite{SP}*{Theorem \href{https://stacks.math.columbia.edu/tag/0BGP}{0BGP}}, a desingularization $\pi \colon Z \surjects X$ exists (ensuring this is the only role of the excellence of $R$). Moreover, by the above and the proof of \cite{BLR90}*{Section 9.7, Theorem~1}, the map $\pi^*\colon H^1(X, \sO_X) \ra H^1(Z, \sO_Z)$ is identified with the map $\Lie(\Pic^0_{X/R}) \hra \Lie(\cJ)$. By forming duals, the finite length cokernel of the latter is isomorphic to $H^0(X, \Omega)/H^0(\cJ, \Omega^1)$. On the other hand, Grothendieck's theorem on formal functions \cite{EGAIII1}*{Corollaire 4.1.7} shows that $H^2(X, \sO_X) = 0$. The above and the spectral sequence $H^i(X, R^j\pi_*(\sO_Z)) \Ra H^{i + j}(Z, \sO_Z)$ then give the claimed
\[
H^0(X, \Omega)/H^0(\cJ, \Omega^1) \simeq H^1(Z, \sO_Z)/\pi^*(H^1(X, \sO_X) ) \cong H^0(X, R^1\pi_*(\sO_Z)).
\]
Since $R^1\pi_*(\sO_Z)$ is supported at the singular points of $X$ and vanishes if and only if $X$ has rational singularities (see \S\ref{pp:rat-sing}), the latter happens if and only if \eqref{lattice-crit} holds.
\epf

\beg\label{XG-case}
\Cref{prop:Raynaud-review} applies to $R = \bZ_{(p)}$ and $X = (X_\Gamma)_{\bZ_{(p)}}$ for every prime $p$ and every $\Gamma_1(N) \subset \Gamma \subset \Gamma_0(N)$. Indeed, $X_1(N)^\sm \cap X_1(N)_{\bF_p} \neq \emptyset$ by \cite{KM85}*{Section (13.5.6)}, so, since, by \cite{EGAIV2}*{Proposition 6.1.5}, the finite map $X_1(N) \surjects X_\Gamma$ is flat away from finitely many points, also $X_\Gamma^\sm \cap (X_\Gamma)_{\bF_{p}} \neq \emptyset$. More generally, it also applies to any $(X_{\Gamma\, \cap\, \wt{H}})_{\bZ_{(p)}}$ with $\Gamma$ as before and $\Gamma_{\mathrm{diag}}(M) \subset \wt{H} \subset \GL_2(\wh{\bZ})$ the preimages of subgroups 
\[
\tst \{ \abcd {x_1} \ \ {x_2}  \vert\, x_i \in (\bZ/M\bZ)^\times \} \subset H \subset \GL_2(\bZ/M\bZ)
\]
for some $M$ coprime to $N$: indeed, the identity $\tst \abcd 0  1  M  0  \abcd abcd \abcd 0  1  M  0 \i = \abcd {d}{  \frac{c}{M}}  {Mb}  a $ gives 
\[
\tst \abcd 0  1  M  0  \Gamma_0(M^2) \abcd 0  1  M  0 \i  = \Gamma_{\mathrm{diag}}(M),
\]
so, by \cite{DR73}*{Chapitre IV, Proposition 3.19 (see also \'{e}quation (3.14.1))}, we obtain an isomorphism 
\[
X_{\Gamma\, \cap\, \Gamma_0(M^2)} \simeq X_{\Gamma\, \cap\, \Gamma_{\mathrm{diag}}(M)},
\]
to the effect that we may now instead use the resulting finite flat map
\be \label{exceptional-map-ann}
X_{\Gamma\, \cap\, \Gamma_0(M^2)} \simeq X_{\Gamma\, \cap\, \Gamma_{\mathrm{diag}}(M)} \surjects X_{\Gamma\, \cap\, \wt{H}} \qxq{to conclude that} X_{\Gamma\, \cap\, \wt{H}}^\sm \cap (X_{\Gamma\, \cap\, \wt{H}})_{\bF_p} \neq \emptyset.
\ee
\eeg

By \Cref{prop:Raynaud-review}, controlling the lattice $H^0(\cJ_0(N), \Omega^1)$ relevant for the Manin constant hinges on  positively answering the pertinent cases of the following question considered by Raynaud \cite{Ray91}.


\begin{question} \label{rat-sing-q}
Does $X_0(N)$ have rational singularities for every $N \ge 1$?
\end{question}

We know of no $N$ for which the answer is negative, in fact, we exhibit a positive one for a large class of $N$ in \Cref{rat-sing-main}, which subsumes \cite{Ray91}*{Th\'{e}or\`{e}me 2}. The new cases in \Cref{rat-sing-main} will come by bootstrapping from \Cref{low-genus}, whose proof uses the following lemma.

\blem \label{const-in-Z}
For $\Gamma_1(N) \subset \Gamma \subset \Gamma' \subset \Gamma_0(N)$, the Jacobians $J_\Gamma$ and $J_{\Gamma'}$ of  $(X_{\Gamma})_\bQ$ and $(X_{\Gamma'})_\bQ$, and isogenous newform elliptic curve quotients\footnote{We say that a surjection of abelian varieties $\pi\colon J_\Gamma \surjects E$ is a \emph{newform quotient} of $J_\Gamma$ if $J_\Gamma/(\Ker(\pi)^0)$ is associated to a newform on $\Gamma$ via the Eichler--Shimura construction (compare, for instance, with \cite{Roh97}*{Section 3.7} or \cite{DS05}*{Definition 6.6.3}). We call such an $E$ a \emph{newform elliptic curve quotient} if, in addition, $E$ is an elliptic curve. } $\pi \colon J_\Gamma \surjects E$ and $\pi' \colon  J_{\Gamma'} \surjects E'$, if $\Ker(\pi)$ and $\Ker(\pi')$ are connected, then there is an isogeny $e \colon E \ra E'$ such that the Manin constants $c_\pi$ and $c_{\pi'}$ satisfy 
\[
c_{\pi'} = c_\pi \cdot \#\Coker(\Lie \cE \xra{\Lie e} \Lie \cE') \qx{where $\cE$ and $\cE'$ are the N\'{e}ron models of $E$ and $E'$}.
\]
Moreover, $c_\pi \in \bZ$ for any newform elliptic curve quotient $\pi\colon J_\Gamma \surjects E$ \up{regardless of $\Ker(\pi)$}.
\elem

\bpf
Everything was settled in \cite{Manin-semistable}*{Lemma 2.12} except for the assertion that $c_\pi \in \bZ$ in the case when $\Ker(\pi)$ is nonconnected. To reduce the latter to the case when $\Ker(\pi)$ is connected, it suffices to consider the factorization $J_\Gamma \surjects J_\Gamma/(\Ker(\pi)^0) \surjects E$ of $\pi$.
\epf

\bprop \label{low-genus}
For the following $\Gamma \subset \GL_2(\wh{\bZ})$, the modular curve $X_\Gamma$ has rational singularities\ucolon
\benumr
\m \label{LG-i}
any $\Gamma_1(N) \subset \Gamma \subset \Gamma_0(N)$ such that $(X_\Gamma)_\bQ$ has genus $\le 1$\uscolon

\m\label{LG-iii}
$\Gamma = \Gamma_0(9)\, \cap\, \wt{C}_3$ with $\wt{C}_3 \subset \GL_2(\wh{\bZ})$ the preimage of the cyclic subgroup $C_3 \subset \GL_2(\bZ/2\bZ) \simeq~\!S_3$.
\eenum
\eprop

\bpf
We will use \Cref{prop:Raynaud-review}, which applies thanks to \Cref{XG-case} (note that $\Gamma_{\mathrm{diag}}(2) = \Gamma(2)$), so we let $\cJ$ be the N\'{e}ron model over $\bZ$ of the Jacobian of $(X_\Gamma)_\bQ$. In particular, we may assume that the genus of $(X_{\Gamma})_\bQ$ is positive: indeed, in the genus $0$ case the spaces in \eqref{lattice-crit} vanish. Then the genus of $(X_{\Gamma})_\bQ$ is $1$: indeed, for \ref{LG-iii}, the genus of $X_0(36)_\bQ$ is $1$, so, due to the surjection
\be \label{exceptional-map}
X_0(36) \xra{\eqref{exceptional-map-ann}} X_{\Gamma_0(9)\, \cap\, \wt{C}_3},
\ee
that of $(X_{\Gamma_0(9)\, \cap\, \wt{C}_3})_\bQ$ is $\le 1$ (in fact, it is $1$, but we do not need to sidestep into showing this).

In \ref{LG-i}, the map $(X_\Gamma)_\bQ \ra X_0(N)_\bQ$ is then an isogeny of elliptic curves over $\bQ$ (see \cite{Sch09}*{Corollary~1.2~(i)}), so that $N < 50$ (compare with \Cref{X0N-genus1} below). By, for instance, \Cref{const-in-Z} and Cremona's \cite{ARS06}*{Theorem 5.2}, the Manin conjecture holds for the optimal parametrization of the elliptic curve $(X_\Gamma)_\bQ$ by the modular curve $(X_\Gamma)_\bQ$: the differential $\omega_f$ associated to the unique normalized newform on $\Gamma_0(N)$ lies in $H^0(\cJ, \Omega^1)$. However, by \Cref{wf-integral} and \Cref{wf-primitive}, this $\omega_f$ is also a primitive element of the lattice $H^0(X_\Gamma^\sm, \Omega^1)$. Since $H^0(\cJ, \Omega^1) \subset H^0(X_\Gamma, \Omega) \subset H^0(X_\Gamma^\sm, \Omega^1)$ (see \eqref{lattice-crit}), these $\bZ$-modules are then all generated by $\omega_f$, so \Cref{prop:Raynaud-review} gives \ref{LG-i}.

In \ref{LG-iii}, we have reduced to the $\bQ$-fiber of the map \eqref{exceptional-map} being an isogeny of elliptic curves of degree $3$ (compare with \cite{DS05}*{bottom of page 66}). Thus, by \cite{LMFDB}*{elliptic curve \href{http://www.lmfdb.org/EllipticCurve/Q/36a1/}{36a1}}, it
must be the unique degree $3$ isogeny with source $X_0(36)_\bQ$. By \cite{LMFDB}*{elliptic curve \href{http://www.lmfdb.org/EllipticCurve/Q/36a3/}{36a3}}, the Manin constant of the resulting \emph{nonoptimal} modular parametrization of the elliptic curve $(X_{\Gamma_0(9)\, \cap\, \wt{C}_3})_\bQ$ is $1$, so the pullback of the N\'{e}ron differential $\omega_\cJ$ is the differential $\omega_f$ associated to the unique normalized newform on $\Gamma_0(36)$. In particular, by \Cref{wf-integral} and \Cref{wf-primitive}, this pullback is a primitive element of $H^0(X_0(36)^\sm, \Omega^1)$ and, to conclude in the same way as for \ref{LG-i}, we use the inclusions
\[
H^0(\cJ, \Omega^1) \overset{\eqref{lattice-crit}}{\subset} H^0(X_{\Gamma_0(9)\, \cap\, \wt{C}_3}, \Omega) \subset H^0(X_{\Gamma_0(9)\, \cap\, \wt{C}_3}^\sm, \Omega^1) \subset H^0(X_0(36)^\sm, \Omega^1),
\]
the last one of which is obtained as \eqref{wf-primitive-eq} by using the map $X_0(36) \ra X_{\Gamma_0(9)\, \cap\, \wt{C}_3}$.
\epf

\beg \label{X0N-genus1}
The $\bZ$-curve $X_0(N)$ has rational singularities for $N = 1, \dotsc, 21, 24, 25, 27, 32, 36, 49$: these are the $N$ for which $X_0(N)_\bQ$ has genus $\le 1$, that is, for which \Cref{low-genus}~\ref{LG-i} applies.
\eeg

To upgrade the finite list of \Cref{low-genus} to infinite families, in \Cref{shave-level} we develop general criteria for rational singularities of $X_0(N)$. For this, we use the following lemmas.

\blem \label{rat-sing-inv}
For an action of a finite group $G$ on a ring $R$, if both $R$ and $R^G$ are complete, $2$-dimensional, Noetherian, normal, local domains \up{when $\#G$ is invertible in $R$, it suffices to assume this for $R$} and $R$ has rational singularities, then, for every proper birational $Z \ra \Spec(R^G)$ with $Z$ normal, $\#G$ kills $H^1(Z, \sO_Z)$, in particular, $R^G$ also has rational singularities when $\#G \in R^\times$.
\elem

\bpf
We may assume  that $G$ acts faithfully and  begin with the parenthetical claim, in which $\#G \in R^\times$ and we consider
the $R^G$-linear operator $\sR\colon r \mapsto \frac{1}{\#G}\sum_{g \in G} gr$ that fixes each $a \in R^G$. By applying $\sR$ to any equality $a = \sum r_i a_i$ with $a, a_i \in R^G$ and $r_i \in R$, we get $ R^G \cap IR = I$ for any ideal $I \subset R^G$. In particular, $R^G$ inherits the ascending chain condition, so is a Noetherian domain. The $0$-dimensional localization $R \tensor_{R^G} K^G$ of $R$ is the fraction field $K$ of $R$, so, by Galois theory, it is a finite extension of the fraction field $K^G$ of $R^G$. We choose a $K^G$-basis $r_1, \dotsc, r_n \in R$ for $K$ and consider the $R^G$-module map $R \ra \bigoplus_{i = 1}^n R^G$ given by $r \mapsto (\sR(rr_i))_{i = 1}^n$. This map is injective because the version of $\sR$ for $K$ cannot kill $\sum_{i = 1}^n rr_i K^G = rK$ unless $r = 0$. Thus, $R$ is a finite $R^G$-module,\footnote{Finite generation of $R$ as an $R^G$-module holds much more generally, even for noncommutative $R$, see \cite{Mon80}*{Corollary 5.9}.} so $R^G \hra R$ is a finite, local map of Noetherian local domains that splits via $\sR$ as a map of $R^G$-modules, and hence $R^G$ is a complete, $2$-dimensional, Noetherian, normal, local domain.

	


Returning to general $G$, for $Z$ as in the statement we let $\wt{Z} \ra \Spec R$ be the proper birational map obtained by normalizing the base change $Z_R$ in $K \ce \Frac(R)$ (the finite type of $\wt{Z}$ over $R$ follows from \cite{EGAIV2}*{Proposition 7.8.6 (ii)}). The $G$-action on $R$ induces a compatible $G$-action on $\wt{Z}$, for which the integral map $\pi\colon \wt{Z} \ra Z$ is equivariant (with $G$ acting trivially on $Z$). Thus, since $Z$ is normal, $\pi$ induces an isomorphism $\wt{Z}/G \isomto Z$. Consequently, the trace map $s \mapsto \sum_{g \in G} gs$ defines an $\sO_Z$-linear morphism $\pi_*(\sO_{\wt{Z}}) \ra \sO_Z$ whose postcomposition with $\sO_Z \ra \pi_*(\sO_{\wt{Z}})$ is multiplication by $\#G$ on $\sO_Z$. The rational singularities assumption gives $H^1(Z, \pi_*(\sO_{\wt{Z}})) = 0$ (see \S\ref{pp:rat-sing}), so the induced maps on $H^1(Z, -)$ show that $\#G$ kills the $R^G$-module $H^1(Z, \sO_Z)$, as claimed. In particular, if $\#G$ is a unit in $R$, so also in $R^G$, then $H^1(Z, \sO_Z) = 0$. By choosing a $Z$ that is regular (see Lipman's \cite{SP}*{Theorem \href{https://stacks.math.columbia.edu/tag/0BGP}{0BGP}}), we then conclude that $R^G$ indeed has rational singularities.
\epf



\blem \lab{rat-sing-crit}
For a prime $p$, we have $p \nmid \#(\Aut(x)/\{\pm 1\})$ for each $x \in \sX_0(N)(\ov{\bF}_p)$ whenever
\benumr
\m \lab{RSC-i}
$p \ge 5$\uscolon or

\m\lab{RSC-ii}
$p = 3$ and there is a prime $p' \mid N$ with $p' \equiv 2 \bmod 3$\uscolon or

\m\lab{RSC-iii}
$p = 2$ and there is a prime $p' \mid N$ with $p' \equiv 3 \bmod 4$.
\eenum
\elem

\bpf
By \cite{modular-description}*{proof of Theorem 6.7}, for cuspidal $x$ we have $\Aut(x) = \{\pm 1\}$, so we may assume that $x$ corresponds to an elliptic curve $E$ over $\ov{\bF}_p$ equipped with a cyclic (in the sense of Drinfeld) subgroup $C \subset E$ of order $N$. Thus, since $\Aut(x) \subset \Aut(E)$ and $\# \Aut(E) \mid 24$ (see \cite{KM85}*{Corollary~2.7.2}), we have \ref{RSC-i}. For \ref{RSC-ii} and \ref{RSC-iii}, we consider the action of $\Aut(x)$ on $E[p'](\ov{\bF}_p)$. Firstly, if $p'$ is odd (resp.,~if $p' = 2$), then this action (resp.,~the induced action of $\Aut(x)/\{\pm 1\}$) is faithful, see \cite{KM85}*{Corollary 2.7.2}. Thus, since it also preserves both the Weil pairing and the cyclic subgroup $C' \ce C \cap E[p'] \subset E[p']$, any $p$-Sylow subgroup $G$ of $\Aut(x)$ (resp.,~of $\Aut(x)/\{\pm 1\}$) acts semisimply on $E[p']$ and embeds into $\Aut(C') \cong (\bZ/p'\bZ)^\times$. In particular, $\#G \mid p' - 1$, so that $G = 1$ in \ref{RSC-ii} and $G = \{\pm 1\}$ in \ref{RSC-iii}.
\epf



\bprop \label{def-thy}\label{shave-level}
For a prime $p$, an $N \in \bZ_{> 0}$, and $n \ce \val_p(N)$,
\benumr
\m \label{DT-i}
if $p \ge 5$\uscolon or

\m \label{DT-iii}
if $p = 3$ and there is a  $p' \mid N$ with $p' \equiv 2 \bmod 3$\uscolon or

\m \label{DT-iv}
if $p = 3$ and either $X_0(3^{n} \cdot 7)_{\bZ_{(3)}}$ or $(X_{\Gamma_0(3^{n})\, \cap\, \wt{C}_3})_{\bZ_{(3)}}$ has rational singularities where the subgroup $\wt{C}_3 \subset \GL_2(\wh{\bZ})$ is the preimage of the cyclic subgroup $C_3 \subset \GL_2(\bZ/2\bZ)$\uscolon or

\m \label{DT-v}
if $p = 2$ and there is a  $p' \mid N$ with $p' \equiv 3 \bmod 4$\uscolon or

\m \label{DT-vi}
if $p = 2$ and $X_0(2^{n} \cdot 5)_{\bZ_{(2)}}$ has rational singularities and $N \neq 2^n$\uscolon or 

\m \label{DT-ii}
if $p = 3$ \up{resp.,~if $p = 2$} and for the level $\Gamma_0(p^n)$ universal deformation ring $R$ of $(E, C)$, where $E/\ov{\bF}_p$ is the elliptic curve with $j = 0$ and $C \subset E$ the cyclic \up{in the sense of Drinfeld} subgroup of order $p^n$, and for every subgroup $G' \subset G \ce \Aut(E)/\{\pm 1\}$ of $p$-power order, $R^{G'}$ has rational singularities \up{resp.,~same, but if $N \neq 2^n$, then may restrict to cyclic $G'$}\uscolon
\eenum
then $X_0(N)_{\bZ_{(p)}}$ has rational singularities.
\eprop

\bpf
Since $X_0(N)_{\bZ_{(p)}}$ is regular away from the $\ov{\bF}_p$-points $x$ with $j = 0$ or $j = 1728$ (see \cite{modular-description}*{Theorem 6.7}), we need to show that $\sO_{X_0(N),\, x}$ has rational singularities for every such $x$. By Lipman's \cite{SP}*{Theorem \href{https://stacks.math.columbia.edu/tag/0BGP}{0BGP}}, there is a proper birational map $Z \ra \Spec(\sO_{X_0(N),\, x})$ with $Z$ regular and, by \cite{EGAIV2}*{Corollaire 6.4.2, Scholie 7.8.3 (v)} (see also \cite{Gre76}*{Corollary 5.6}), the $\wh{\sO}_{X_0(N),\, x}^\sh$-base change of $Z$ is regular. Thus, by checking the vanishing $H^1(Z, \sO_Z) = 0$ after flat base change, $\sO_{X_0(N),\, x}$ has rational singularities if and only if so does $\wh{\sO}^\sh_{X_0(N),\, x}$. However, by \cite{DR73}*{Chapitre I, Section (8.2.1)} (or \cite{Ols06}*{Theorem 2.12}), we have
\be \label{X0N-sh-id}
\wh{\sO}_{X_0(N),\, x}^\sh \cong (\wh{\sO}_{\sX_0(N),\, x}^\sh)^{\Aut(x)/\{\pm 1\}},
\ee
and $\wh{\sO}_{\sX_0(N),\, x}^\sh$ is regular by \cite{KM85}*{Theorem 6.6.1}. Thus, \ref{DT-i}, \ref{DT-iii}, and \ref{DT-v} follow from \Cref{rat-sing-inv,rat-sing-crit}.

In \ref{DT-ii}, the unique $E$ is supersingular, $C$ is the kernel of the $p^n$-fold relative Frobenius (see \cite{KM85}*{Lemma 12.2.1}) and hence is preserved by $\Aut(E)$, and $x$ maps to $(E, C)$. Moreover, $E[\frac{N}{p^n}]$ is \'{e}tale, so its subgroups $C' \subset E[\frac{N}{p^n}]$ deform uniquely, and hence $R \cong \wh{\sO}^\sh_{\sX_0(N),\, x}$ by the modular interpretation of $\sX_0(N)$. Since $G$ injects into (in fact, equals to) $\SL_2(\bF_3)/\{\pm 1\}$ if $p = 2$ and $\SL_2(\bF_2)$ if $p = 3$ (see \cite{KM85}*{Corollary 2.7.2}, also \cite{Del75}*{Proposition 5.9 (IV)--(V), Section 7.4}), its $p$-Sylow subgroup $G^{(p)} \subset G$ is normal. Thus, the same holds for $H \ce \Aut(x)/\{\pm 1\} \subset G$, to the effect that $R^{H} \cong (R^{H^{(p)}})^{H/H^{(p)}}$. The assumption of \ref{DT-ii} ensures that $R^{H^{(p)}}$ has rational singularities, so, by \Cref{rat-sing-inv}, so does $R^{H} \cong \wh{\sO}_{X_0(N),\, x}^\sh$ (see \eqref{X0N-sh-id}). To conclude \ref{DT-ii}, we note that $H$ is cyclic when $p = 2$ and $N \neq 2^n$: then the preimage of $H$ in $\Aut(E)$ lies in the cyclic group $(\bZ/p'\bZ)^\times$ for an odd prime $p' \mid N$ (see the proof of \Cref{rat-sing-crit}).

To show that \ref{DT-iv} and \ref{DT-vi} follow from \ref{DT-ii}, we set $\Gamma \ce \Gamma_0(3^n \cdot 7)$ or $\Gamma \ce \Gamma_0(3^n) \cap \wt{C_3}$ in \ref{DT-iv} and $\Gamma \ce \Gamma_0(2^n \cdot 5)$ in \ref{DT-vi} and, in the view of the above, especially, the analogue of \eqref{X0N-sh-id} for $\sX_\Gamma$ and the insensitivity of the universal deformation ring $R$ of $(E, C)$ in \ref{DT-ii} to tame level, need to show that every cyclic subgroup $G' \subset  \Aut(E)/\{\pm 1\}$ of $p$-power order is $\Aut(z)/\{\pm 1\}$ for some  $z \in \sX_\Gamma(\ov{\bF}_p)$. For $p = 3$, the unique $G'$ of $3$-power order is $\bZ/3\bZ$ and its preimage $\wt{G}' \subset \Aut(E)$ is $\bZ/6\bZ$. Since $\bF_7$ contains sixth roots of unity, the action of $\wt{G}'$ on $E[7]$ is diagonalizable and either of the resulting $\wt{G}'$-stable $\bF_7$-lines $C' \subset E[7]$ is the $7$-primary part of a level structure that determines the desired $z$ for $\Gamma = \Gamma_0(3^n \cdot 7)$. Similarly, the faithful action of $G'$ on $E[2]$ determines a $\wt{C}_3$-structure, and so a desired $z$ for $\Gamma = \Gamma_0(3^n) \cap \wt{C}_3$. For $p = 2$, the argument is analogous: now 
$G'$ is $\bZ/2\bZ$ but is no longer unique (the $2$-Sylow of $\SL_2(\bF_3)/\{\pm 1\}$ is $\bZ/2\bZ \times \bZ/2\bZ$), its preimage $\wt{G}'$ is $\bZ/4\bZ$, and one can diagonalize the action of $\wt{G}'$ on $E[5]$ because $\bF_5$ contains fourth roots of unity.
\epf

\brem \label{kill-by-6}
By the preceding proof, if $N \neq 2^n$, then the $p$-Sylow subgroup of the exceptional automorphism group at each $\ov{\bF}_p$-point of $\sX_0(N)$ is normal and either trivial or $\bZ/p\bZ$ (the latter can occur only for $p = 2$ and $p = 3$). In particular, \Cref{rat-sing-inv} and the preceding proof show that for any proper birational $\pi \colon Z \surjects X_0(N)$ with $Z$ normal, 
the $\sO_{X_0(N)}$-module $R^1\pi_*(\sO_Z)$ is killed by $6$.
\erem

A big portion of the following partial positive answer to \Cref{rat-sing-q}  appeared in \cite{Ray91}*{Th\'{e}or\`{e}me~2}: our main improvement to \emph{loc.~cit.}~is the inclusion of the cases $\val_p(N) = 2$ for $p \le 3$.

\bthm \label{rat-sing-main}
For a prime $p$, the modular curve $(X_0(N))_{\bZ_{(p)}}$ has rational singularities whenever
\benum
\m
$p \ge 5$\uscolon or

\m
$p = 3$ and either $\val_p(N) \le 2$ or there is a prime $p' \mid N$ with $p' \equiv 2 \bmod 3$\uscolon or

\m
$p = 2$ and either $\val_p(N) \le 2$ or there is a prime $p' \mid N$ with $p' \equiv 3 \bmod 4$.
\eenum
\ethm

\bpf
Thanks to \Cref{shave-level}, it suffices to check is that $X_0(7)$, $X_0(21)$, and $X_{\Gamma_0(9)\, \cap\, \wt{C}_3}$, as well as $X_0(5)$, $X_0(10)$, $X_0(20)$, $X_0(1)$, $X_0(2)$, and $X_0(4)$ have rational singularities. We have already done this in \Cref{low-genus} (see also \Cref{X0N-genus1}).
\epf

\brem \label{reduce-to-compute}
The method would show that $X_0(N)$ has rational singularities for every $N \neq 2^n$ equal to a conductor of an elliptic curve over $\bQ$ if one knew that $X_{\Gamma_0(27)\, \cap\, \wt{C}_3}$, $X_{\Gamma_0(81)\, \cap\, \wt{C}_3}$, and $X_{\Gamma_0(243)\, \cap\, \wt{C}_3}$ (or, if one prefers, $X_0(27 \cdot 7)$, $X_0(81 \cdot 7)$, and $X_0(243 \cdot 7)$), as well as $X_0(8 \cdot 5)$, $X_0(16 \cdot 5)$, $X_0(32 \cdot 5)$, $X_0(64 \cdot 5)$, $X_0(128 \cdot 5)$, $X_0(256 \cdot 5)$, $X_0(64)$, $X_0(128)$, and $X_0(256)$ have rational singularities (for well-known conductor exponent bounds for an elliptic curve over $\bQ$, see \cite{Pap93}*{Corollaire du Th\'{e}or\`{e}me~1}). 
\erem



\bcor \label{Neron-integral}
For a normalized newform $f \in H^0(\sX_0(N)_\bQ, \omega^{\tensor 2}(-\cusps))$ \up{see \uS\uref{omega}}
and the N\'{e}ron model $\cJ_0(N)$ over $\bZ$ of the Jacobian $J_0(N)$ of $X_0(N)_\bQ$,
\[
6\cdot \omega_f \in H^0(\cJ_0(N), \Omega^1), \qxq{where $\omega_f$ is the differential associated to $f$\uscolon}
\]
if $X_0(N)$ has rational singularities, then even $\omega_f \in H^0(\cJ_0(N), \Omega^1)$.
\ecor

\bpf
The Manin conjecture for the quotient $\pi \colon J_0(N) \surjects E$ with connected $\Ker(\pi)$ determined by $f$ predicts that $\omega_f$ is the pullback of a N\'{e}ron differential $\omega_E$ of the elliptic curve $E$. By the functoriality of N\'{e}ron models, this pullback lies in $H^0(\cJ_0(N), \Omega^1)$, so, by, for instance, Cremona's \cite{ARS06}*{Theorem 5.2} that verified the Manin conjecture for small $N$, we may assume that $N \neq 2^n$. By \Cref{prop:Raynaud-review}, there is an inclusion $H^0(\cJ_0(N), \Omega^1) \hra H^0(X_0(N), \Omega)$ that is an isomorphism if and only if $X_0(N)$ has rational singularities and, by \Cref{kill-by-6}, in general its cokernel is killed by $6$. Thus, it remains to recall from  \Cref{wf-integral} that $\omega_f \in H^0(X_0(N), \Omega)$.
\epf


\section{A relation between the Manin constant and the modular degree}

Our final goal is to use the work above to establish \Cref{main-general,main-X1n}. The following basic fact is the underlying source of the relationship between the modular degree and the Manin constant.

\blem \lab{congJ-deg}
For a field $k$, a proper, smooth $k$-curve $X$ with the Jacobian $J \ce \Pic^0_{X/k}$, a $k$-surjection $\phi \colon X \surjects E$ onto an elliptic curve, a point $P \in X(k)$ with $\phi(P) = 0$, the closed immersion $i_P\colon X \hra J$ given by $Q \mapsto \sO_X(Q - P)$, and the homomorphism $\pi\colon J \surjects E$ obtained from $\phi$ by the Albanese functoriality of $J$, the composition $\pi \circ \pi^\vee \colon E \ra J \ra E$ is multiplication by $\deg \phi$. 
\elem

\bpf
The existence of $\phi$ implies that $X$ has genus $> 0$, and the map $\pi \colon J \ra E$ is characterized by $\sO_X(Q - P) \mapsto \phi(Q)$, see \cite{Mil86b}*{Proposition 6.1}. Moreover, by \cite{Mil86b}*{Lemma 6.9 and Remark~6.10~(c)}, the map $\Pic^0(i_P)$ is the \emph{negative} of the inverse of the canonical principal polarization of $J$ and the canonical principal polarization of $E$ sends a $Q \in E(\ov{k})$ to $\sO_{E_{\ov{k}}}( [0] - [Q])$ (see also \cite{Con04}*{Example 2.5}). In particular, the map $\Pic^0(\phi) = \Pic^0(i_P) \circ \pi^\vee$ sends such a $Q$ to $\sO_{X_{\ov{k}}}([\phi\i(0)] - [\phi\i(Q)])$ and, by taking into account the canonical principal polarization of $J$, we find that $\pi \circ \pi^\vee$ sends $Q$ to $\deg \phi \cdot Q$.
\epf


\bthm \label{thm:main-result-proof}
For an elliptic curve $E$ over $\bQ$ of conductor $N$,  a N\'{e}ron differential $\omega_E \in H^0(E, \Omega^1)$, the normalized newform $f$ determined by $E$, its associated $\omega_f \in H^0(X_0(N)_\bQ, \Omega^1)$, a subgroup $\Gamma_1(N) \subset \Gamma \subset \Gamma_0(N)$, and a prime $p$, if for some subgroup $\Gamma \subseteq \Gamma' \subseteq \Gamma_0(N)$ the curve $(X_{\Gamma'})_{\bZ_{(p)}}$ has rational singularities \up{see Theorem \uref{rat-sing-main}}, then every surjection
\[
\phi \colon (X_\Gamma)_\bQ \surjects E \qxq{satisfies} \val_p(c_\phi) \le \val_p(\deg(\phi)) \qxq{with} c_\phi \in \bZ  \qxq{defined by} \phi^*(\omega_E) = c_\phi \cdot \omega_f.
\]
Without the rational singularities assumption, we still have
\[
\val_p(c_\phi) \le \val_p(\deg(\phi)) + \begin{cases}
1 &\x{if $p = 2$ with $\val_2(N) \ge 3$ and there is no $p'\mid N$ with $p' \equiv 3 \bmod 4$,} \\
1 &\x{if $p = 3$ with $\val_3(N) \ge 3$ and there is no $p'\mid N$ with $p' \equiv 2 \bmod 3$,} \\
0 &\x{otherwise.}
\end{cases}
\]
\ethm


\bpf
By \Cref{wf-integral}, we have $\omega_f \in H^0(X_{\Gamma'}, \Omega)$. Thus, by \Cref{prop:Raynaud-review},
the rational singularity assumption ensures that $\omega_f \in H^0((\cJ_{\Gamma'})_{\bZ_{(p)}}, \Omega^1)$ where $\cJ_{\Gamma'}$ is the N\'{e}ron model of the Jacobian $J_{\Gamma'}$ of $(X_{\Gamma'})_\bQ$. We choose a $P \in X_{\Gamma}(\bQ)$, for instance, a rational cusp, and consider the resulting embeddings $(X_\Gamma)_\bQ \hra J_\Gamma$ and $(X_{\Gamma'})_\bQ \hra J_{\Gamma'}$. By the Albanese functoriality of the Jacobian, the map $X_{\Gamma} \ra X_{\Gamma'}$ induces a morphism $\cJ_{\Gamma} \ra \cJ_{\Gamma'}$, and we conclude by pullback that
\be \lab{wf-temp-int}
\omega_f \in H^0((\cJ_{\Gamma})_{\bZ_{(p)}}, \Omega^1)
\ee
(here we use the compatibility of the identification $H^0((X_{\Gamma'})_\bQ, \Omega^1) \cong H^0(J_{\Gamma'}, \Omega^1)$ obtained by pullback along $(X_{\Gamma'})_\bQ \hra J_{\Gamma'}$ with its counterpart obtained by Grothendieck--Serre duality as in \eqref{GS-id}, see \cite{Con00}*{Theorem B.4.1}).
By postcomposing with a translation, we may assume that $\phi(P) = 0$, and we then let $\pi \colon J_\Gamma \surjects E$ be the map that $\phi$ induces via the Albanese functoriality.  \Cref{congJ-deg} ensures that $\pi \circ \pi^\vee\colon E \hra J_\Gamma \surjects E$ is multiplication by $\deg(\phi)$, so the same holds for the induced $\cE \ra \cJ_\Gamma \ra \cE$ on N\'{e}ron models. Thus, by pullback, $\deg(\phi) \cdot \omega_E = c_\phi \cdot (\pi^\vee)^*(\omega_f)$. Since $c_\phi \in \bZ$ by \Cref{const-in-Z} and   $(\pi^\vee)^*(\omega_f) \in H^0(\cE_{\bZ_{(p)}}, \Omega^1) \cong \bZ_{(p)}\cdot\omega_E$ by \eqref{wf-temp-int},  we obtain the sought
 \[
 \val_p(c_\phi) \le \val_p(\deg(\phi)).
 \]
Without the rational singularities assumption, by \Cref{Neron-integral} and the Albanese functoriality as above, we still have $6 \cdot \omega_f \in H^0(\cJ_\Gamma, \Omega^1)$, so the same argument gives $\val_p(c_\phi) \le \val_p(\deg(\phi)) + \val_p(6)$. In particular, by also using \Cref{rat-sing-main}, we obtain the claimed last display in the statement.
\epf


Since $X_1(N)$ almost always agrees with the regular $\sX_1(N)$, we now show that the above minor hypothetical exceptions to the divisibility $c_\phi \mid \deg(\phi)$ cannot occur for parametrizations by~$X_1(N)_\bQ$.

\bcor\label{X1N-main-pf}
For an elliptic curve $E$ over $\bQ$ of conductor $N$,  a N\'{e}ron differential $\omega_E \in H^0(E, \Omega^1)$, the normalized newform $f$ determined by $E$, and its associated $\omega_f \in H^0(X_1(N)_\bQ, \Omega^1)$, every surjection
\[
\phi \colon X_1(N)_\bQ \surjects E \qxq{satisfies} c_\phi \mid \deg(\phi) \qxq{with} c_\phi \in \bZ  \qxq{defined by} \phi^*(\omega_E) = c_\phi \cdot \omega_f.
\]
\ecor

\bpf
By \Cref{thm:main-result-proof}, we have $\val_p(c_\phi) \le \val_p(\deg(\phi))$ for every prime $p \ge 5$. For the remaining $p = 2$ and $p = 3$, \Cref{thm:main-result-proof} applied with $\Gamma = \Gamma' = \Gamma_1(N)$ gives the same as soon as $X_1(N)_{\bZ_{(p)}}$ is regular. By \cite{KM85}*{Corollary 2.7.3, Theorem 5.5.1} and \cite{modular-description}*{Lemma 4.1.3, Theorem 4.4.4}, this happens whenever $p' \mid N$ for a prime $p' \ge 5$. Thus, we may assume that $N = 2^a\cdot 3^b$, in fact, by the last aspect of \Cref{thm:main-result-proof}, even that $N = 2^a$ or $N = 3^b$ (so $a \le 8$ and $b \le 5$, see \cite{Pap93}*{Corollaire du Th\'{e}or\`{e}me~1}).
For any isogeny $\psi\colon E' \ra E$, since the composition with the dual isogeny is multiplication by $\deg(\psi)$, we have $\psi^*(\omega_E) = c_\psi\cdot \omega_{E'}$ for some $c_\psi \in \bZ$ with $c_\psi \mid \deg(\psi)$. Thus, we may assume that $\phi$ does not factor through any such $\psi$.
For low conductor curves, by Cremona's \cite{ARS06}*{Theorem 5.2}, the Manin constant of such optimal parametrizations by $X_0(N)_\bQ$ is $\pm 1$. Thus, \Cref{const-in-Z} allows us to conclude the same for parametrizations by $X_1(N)_\bQ$  with $N = 2^a$ and $N = 3^b$, so that indeed $\val_p(c_\phi) \le \val_p(\deg(\phi))$.
\epf

\end{document}